\documentclass{amsart}
\usepackage{float}
\usepackage{graphicx} 
\usepackage{comment}

\usepackage{amsmath, amssymb, amsthm}
\usepackage{xcolor}
\usepackage{tikz}
\usetikzlibrary{cd}
\usetikzlibrary{positioning, decorations.pathreplacing}
 \usetikzlibrary{calc}
 \usetikzlibrary{positioning, decorations.pathreplacing, calc}

\usepackage[margin=3.25cm]{geometry}

 \newtheorem{lemma}{Lemma}[section]
    \newtheorem{proposition}[lemma]{Proposition}
    \newtheorem{theorem}[lemma]{Theorem}          \newtheorem*{thma}{Theorem A} 
      \newtheorem*{thmb}{Theorem B}   
      \newtheorem*{thmc}{Theorem C} 
      \newtheorem*{thmd}{Theorem D} 
      \newtheorem*{thme}{Theorem E} 

    \newtheorem{corollary}[lemma]{Corollary}
    \theoremstyle{definition}
    \newtheorem{definition}[lemma]{Definition}
    \newtheorem{example}[lemma]{Example}
    \newtheorem{notation}[lemma]{Notation}
    \newtheorem{convention}[lemma]{Convention}
    \newtheorem{remark}[lemma]{Remark}

\usepackage{hyperref}

\DeclareMathOperator{\Br}{Br}

\DeclareMathOperator{\sing}{sing}
\DeclareMathOperator{\Proj}{\mathbb P}

\DeclareMathOperator{\sep}{sep}
\DeclareMathOperator{\sm}{sm}

\DeclareMathOperator{\Lo}{Loop}
\DeclareMathOperator{\E}{Eul}
\DeclareMathOperator{\qE}{qE}

\DeclareMathOperator{\lto}{\longrightarrow}

\DeclareMathOperator{\Ima}{Im}
\DeclareMathOperator{\U}{\mathcal U}
\DeclareMathOperator{\orient}{\mathfrak o}
\DeclareMathOperator{\Can}{Can}
\DeclareMathOperator{\id}{id}
\DeclareMathOperator{\CT}{CT}
\DeclareMathOperator{\Aut}{Aut}
\DeclareMathOperator{\Isom}{Isom}
\DeclareMathOperator{\SSAP}{SSAP}

\DeclareMathOperator{\action}{\curvearrowright}
\DeclareMathOperator{\codim}{codim}

\DeclareMathOperator{\Lin}{Lin}
\DeclareMathOperator{\Spec}{Spec}

\DeclareMathOperator{\Ram}{Ram}
\DeclareMathOperator{\spann}{span}

\DeclareMathOperator{\Irr}{Irr}

\DeclareMathOperator{\irr}{irr}
\DeclareMathOperator{\chara}{char}
\DeclareMathOperator{\len}{length}

\DeclareMathOperator{\h}{h}
\DeclareMathOperator{\Symm}{Sym}

\DeclareMathOperator{\supp}{supp}
\DeclareMathOperator{\st}{st}
\DeclareMathOperator{\iso}{iso}

\DeclareMathOperator{\pair}{pair}

\DeclareMathOperator{\Cone}{C_1}
\DeclareMathOperator{\TT}{T}

\newcommand{\V}{\mathcal V}

\newcommand{\G}{\mathcal G}

\newcommand{\F}{\mathcal F}

\newcommand{\SA}{\mathcal S_A}

\newcommand{\SC}{\mathcal S_{\widehat{C}}}

\newcommand{\mc}{\mathcal}

\usepackage[all]{xy}

\newcommand{\Pic}{\overline{P^{g-1}_C}}

\newcommand{\ud}{\underline{d}}

\newcommand{\p}{D}

\newcommand{\name}{nodal $1$-scheme}
\newcommand{\subname}{nodal $1$-subscheme}

\title{Extending Andreotti's proof of Torelli's theorem to nodal curves}

\author{Alex Abreu, Marco Pacini, Nicola Pagani}

\begin{document}

\begin{abstract}
Let \(C\) be a connected nodal curve of genus $g$ over an algebraically closed field $k$ of
characteristic zero. Let
\[
A(C)=\bigl(J(C)\curvearrowright \Pic ,\Theta(C)\bigr)
\]
be the stable semi-abelic pair (in the sense of Alexeev) associated with its canonical 
compactified Jacobian $\Pic$ in degree $g-1$. 
The
canonical triple of \(C\) consists of the closure of its canonical image
together with two sets of distinguished points encoding the information of the images of the nonseparating nodes and of the branch data of the canonical map.

 We prove a functorial reconstruction: given connected nodal curves $C$ and $C'$ over $k$,  
every isomorphism \(A(C)\cong A(C')\) of  stable semi-abelic pairs canonically induces a projectivity
between the corresponding canonical triples of $C$ and $C'$. Passing to isomorphism classes, this recovers the reconstruction statement in the compactified Torelli theorem of Caporaso--Viviani
\cite{Cap-Viv}.  

Our proof extends Andreotti's Gauss-map strategy. The branch locus of the
projection from the normalization of the Gauss graph recovers the nonlinear
canonical components and branch data by projective biduality; boundary strata
of the compactified Jacobian recover the node-images; and incidence arguments
recover the linear components. Along the way, we repair technical gaps in the
classical treatments of Andreotti's proof of Torelli's theorem for smooth curves
\cite{andreotti,acgh}.\end{abstract}
\maketitle


\tableofcontents

\section{Introduction}

\subsection{History, motivation and goal}

Throughout, we work over an algebraically closed field of characteristic zero.
The classical Torelli theorem states that a smooth projective curve $C$ is
determined by its principally polarized Jacobian $(J(C),\Theta(C))$. It was
first proved over the complex numbers by Torelli \cite{Tor}; algebraic proofs
include \cite{weil,matsusaka,andreotti,martens}.

Andreotti’s proof in \cite{andreotti} uses the Gauss map of the theta divisor. For a nonhyperelliptic curve, the branch locus of the projection
from the normalization of the Gauss graph is the projective dual of the
canonical curve, which can therefore be recovered by biduality. In the
hyperelliptic case, the corresponding branch data recover the branch
points of the canonical map and hence the curve. Jambois extended this
strategy to irreducible complex curves with planar singularities via
generalized Jacobians \cite{jambois}, but did not treat reducible curves
or establish the hyperelliptic case. Our aim is to extend Andreotti's
argument to all nodal curves.

The boundary of the Deligne--Mumford compactification
$\overline{\mathcal M}_g$ parametrizes singular nodal curves.
Caporaso constructed in \cite{caporaso} a compactification of the universal Picard
variety over $\overline{\mathcal M}_g$,  and later
studied the theta divisor on its degree $g-1$ fibers
\cite{captheta}. Alexeev proved in \cite{alexeev04} that the canonical compactified
Jacobian $\overline{P_C^{g-1}}$, endowed with its theta divisor and the action of
the generalized Jacobian, is a stable semi-abelic pair, a notion which he had earlier introduced in \cite{alexeev} to compactify the moduli space $\mathcal{A}_g$ of principally polarized abelian varieties. He also showed that this
construction compactifies the classical Torelli morphism $\mathcal{M}_g \to \mathcal{A}_g$.

Caporaso and Viviani characterized in \cite{Cap-Viv} the fibers of the compactified
Torelli morphism. For connected stable curves without separating nodes,
they proved that the stable semi-abelic pairs associated with the canonical
compactified Jacobians are isomorphic if and only if the curves are
\emph{Torelli equivalent}\footnote{In loc. cit. the authors call this equivalence relation ``$\Cone$-equivalence".} in the sense of  Definition~\ref{def:Torelli-eq}. The
direction relevant here is the reconstruction of the
Torelli-equivalence class from the stable semi-abelic pair.

The present paper offers an alternative description of the failure of injectivity via the canonical model, generalizing Andreotti's argument. This also provides a functorial refinement of Caporaso--Viviani's reconstruction: every isomorphism
between the stable semi-abelic pairs associated with connected nodal
curves canonically induces a projectivity between their canonical
triples. Introduced in Definition~\ref{def:canonical_model}, canonical triples consist of the closure of the canonical image of the curve, the marked
images of the nonseparating nodes, and the branch data of the
canonical map; more details on this are in the next subsection.

Along the way, we repair gaps in Andreotti's argument for smooth curves  \cite{andreotti} and in its modern treatment in \cite{acgh}.

\subsection{Overview of the results}

Let $C$ be a connected nodal curve of arithmetic genus $g$, and write
$C^{\mathrm{sep}}\subset C^{\mathrm{sing}}$ for the set of separating nodes.
Let $U_C\subset C$ be the complement of the base locus of $|\omega_C|$, and consider the canonical map
\[
\varphi_C\colon U_C\to
\Proj_C:=\mathbb P\bigl(H^0(C,\omega_C)^\vee\bigr).
\]
 Let $C^\dagger$ be obtained by normalizing $C$ at all
separating nodes and discarding the genus-zero connected components. The
canonical map induces a morphism $\varphi_{C^\dagger}\colon C^\dagger\to\Proj_C$.
We set
\[
\widehat C:=\overline{\varphi_C(U_C)},\qquad
\mathcal S_{\widehat C}:=
\varphi_C\bigl(C^{\mathrm{sing}}\setminus C^{\mathrm{sep}}\bigr),\qquad
\mathcal B_{\widehat C}:=
\Br(\varphi_{C^\dagger})\setminus\mathcal S_{\widehat C},
\]
where $\Br(\varphi_{C^\dagger})$ is the branch locus of $\varphi_{C^\dagger}$.
The \emph{canonical triple} of $C$ is
\[
\Can(C):=(\widehat C,\mathcal S_{\widehat C},\mathcal B_{\widehat C})
\]
(see Definition~\ref{def:canonical_model}). An isomorphism of canonical triples
is a projectivity of the ambient projective spaces carrying the three entries
of one triple to the corresponding entries of the other. We denote by $\Isom_{\CT}(\Can(C),\Can(C'))$ the set of
such isomorphisms. 

Let $\overline{P_C^{g-1}}$ be the canonical  compactified Jacobian in degree $g-1$, let $\Theta(C) \subseteq \overline{P_C^{g-1}}$ be its theta divisor, and let $J(C)$ be the generalized Jacobian, parameterizing line bundles of degree zero on every irreducible
component of $C$. By \cite{alexeev04}, the datum  $A(C) := (J(C) \action \overline{P_C^{g-1}}, \Theta(C))$ is a \emph{stable semi-abelic pair} ($\SSAP$) in the sense of Alexeev \cite{alexeev}. Our first result can be stated as follows.

\begin{thma}[Theorem \ref{thm:main1}]
Let $C$ and $C'$ be connected nodal curves. There is a natural function
\[
\rho_{C,C'}\colon\Isom_{\SSAP}(A(C),A(C')) \to \Isom_{\CT}(\Can(C), \Can(C')).
\]
\end{thma}

Our proof of Theorem~A is independent of Caporaso and Viviani's reconstruction result in \cite{Cap-Viv}; we follow Andreotti's strategy. 

The main 
input of the proof is the rational Gauss map
\[
\mathcal G\colon\Theta(C)\dashrightarrow \mathbb{P}_C^\vee,
\]
which sends a smooth point of the theta divisor to its tangent
hyperplane (see Section~\ref{sec:gauss-graph}). Our argument reconstructs the canonical triple in four steps.

In the first step, we reduce to stable curves without separating
nodes. Write $C^\dagger=\coprod_i C_i$. Lemma~\ref{lem:gauss-blocks-separating} identifies the irreducible
components of the image of the Gauss map with the subspaces
$\mathbb{P}_{C_i}^{\vee}$ and shows that each such subspace, together with
$A(C)$, intrinsically determines $A(C_i)$. The reconstruction can
therefore be carried out separately on the curves $C_i$, each of
which has no separating nodes. For the factors of genus at least
$2$, Remarks~\ref{rem:canonical_model_sep_stab_isomorphic} and
\ref{rem:SSAPstab} then allow us to pass to the stabilization; the
genus-one case is treated directly. This intrinsic
recovery of the factors also fills a small gap in the reduction to
curves without separating nodes in the proof of
\cite[Theorem~2.1.7]{Cap-Viv} (see Remark~\ref{rem:CV-gap}).

In the second step, we recover the nonlinear components of
$\widehat C$ and the  branch data $\mathcal B_{\widehat C}$.
Let
\[
\Upsilon\subseteq\Theta(C)\times \mathbb{P}_C^\vee
\]
be the closure of the graph of $\mathcal G$, which we call the
\emph{Gauss graph}, and let $\widehat\Upsilon$ be its normalization.
Projection onto the second factor induces a generically finite
morphism
\[
\widehat\gamma\colon\widehat\Upsilon\longrightarrow \mathbb{P}_C^\vee.
\]
Notice that $\Proj_C^\vee :=\mathbb {P}(H^0(C,\omega_C))$ is dual to the
projective space $\Proj_C := \mathbb{P}(H^0(C,\omega_C)^\vee)$ containing the canonical curve.

The irreducible components of the Gauss graph are indexed by the stable
multidegrees $\ud\in\Sigma(C)$. We call such a multidegree \emph{quasi-Eulerian} if it is induced by
an orientation of the dual graph for which, at every vertex, the
difference between the numbers of incoming and outgoing edges has
absolute value at most one.

We compute the branch locus of $\widehat\gamma$ componentwise. For each
$\ud$, we compare the restriction $\widehat\gamma_{\ud}$ with an
auxiliary morphism $\rho_{\ud}$ obtained from a Gauss map on the
symmetric product of the normalization of $C$. Proposition
\ref{prop:rho-branch} computes the branch locus of $\rho_{\ud}$, and
Proposition~\ref{prop:equal} identifies it with that of
$\widehat\gamma_{\ud}$.

For an arbitrary stable multidegree, this branch locus may detect only
some of the nonlinear components of $\widehat C$. If $\ud$ is
quasi-Eulerian, however, all the nonlinear components are detected.
Proposition~\ref{prop:q_orient} establishes the existence of such a
multidegree. We obtain the following formula.

\begin{thmb}[Theorem~\ref{thm:branch-gamma}]
Let $C$ be a connected  nodal curve of genus $g\geq 2$ with no separating nodes.  
Let $\ud\in \Sigma(C)$ be a stable multidegree. Let $\widehat\Upsilon_{\ud}$ be the normalization of the component of the graph of the Gauss map $\mathcal G\colon \Theta(C)\dashrightarrow\mathbb{P}(T_0J(C))^\vee$ corresponding to $\ud$. The branch locus of the morphism $\widehat\gamma_{\ud}\colon \widehat\Upsilon_{\ud}\longrightarrow \Proj_C^\vee$ is
 \[
    \Br(\widehat{\gamma}_{\ud}) = \left(\bigcup_{W\in \mathcal B_{\ud}} W^\vee\right) 
    \cup\left(\bigcup_{p\in \mathcal B_{\widehat{C}}}\Lambda_p^\vee\right),
    \]
    where $\mathcal B_{\ud}$ is a subset of components of the canonical model $\widehat{C}$ of $C$ explicitly given in terms of $\ud$ (Equation~\eqref{eq:Bd}), and where $W^\vee\subset \Proj_C^\vee$ and $\Lambda_p^\vee\subset \Proj_C^\vee$ denote the duals of the curve $W$ and of the point $p$ respectively.

     Now assume $\ud$ is also quasi-Eulerian. If $\widehat\gamma\colon \widehat\Upsilon\longrightarrow \Proj_C^\vee$ denotes the morphism from the normalization of the entire graph of the Gauss map, then  
\[
\Br(\widehat\gamma_{\ud})=\Br(\widehat{\gamma})=
\left(\bigcup_{W\in \Irr_{\geq 2}(\widehat C)} W^\vee\right)
\cup
\left(\bigcup_{p\in\mathcal B_{\widehat C}}\Lambda_p^\vee\right),
\]
where $\Irr_{\geq 2}(\widehat C)$ denotes the set of components of $\widehat{C}$ of degree at least 2.
\end{thmb}

The nonlinear components $W^\vee$ of the branch locus recover the nonlinear
components $W\subseteq\widehat C$ by projective biduality, while the
hyperplanes $\Lambda_p^\vee$ recover the points
$p\in\mathcal B_{\widehat C}$. Thus the second step reconstructs the
nonlinear part of $\widehat C$ together with the  branch data
$\mathcal B_{\widehat C}$.


The proof of this formula also repairs two genuine gaps in the
classical literature for smooth curves. Andreotti's computation relies on
\cite[Lemma~3]{andreotti}, whose asserted birational invariance of the
branch locus is incorrect; Remark
\ref{rem:contra-example-Andreotti} gives a counterexample. The
treatment in \cite[p.~245]{acgh} avoids this lemma, but its description
of the ramification locus on p.~247 is also incorrect, as already observed
by Debarre \cite[p.~691]{debarre-theta}. More concretely, Example
\ref{ex:acgh} exhibits ramification points corresponding to reduced
divisors and unramified points corresponding to nonreduced divisors. Consequently, the argument in \cite{acgh} does not justify its
description of the branch locus, although the resulting branch-locus
formula is correct. Proposition~\ref{prop:rho-ramif} identifies the
codimension-one ramification components needed for the computation,
 and
Proposition~\ref{prop:equal} identifies the resulting branch locus
with that of the intrinsic Gauss graph. These results provide the missing
argument; see Remark~\ref{rem:correction-ACGH}.

In the third step, we recover $\mathcal S_{\widehat C}$ intrinsically
from the stable semi-abelic pair. Let
$\zeta\colon\Upsilon\to\Theta(C)$ and
$\gamma\colon\Upsilon\to \mathbb{P}_C^\vee$ be the two projections, and set
\(
\Theta_H:=\zeta\bigl(\gamma^{-1}(H)\bigr).
\)
Proposition~\ref{prop:singularities-recovering} shows that a point
$q\in \mathbb{P}_C$ belongs to $\mathcal S_{\widehat C}$ if and only if, for a
general $H\in\Lambda_q^\vee$, the set $\Theta_H$ meets a
nonmaximal $J(C)$-orbit of $\overline{P_C^{g-1}}$. Thus the boundary
stratification recovers precisely the marked images of the
nonseparating nodes.

The fourth step reconstructs the linear components of the canonical curve $\widehat C$.
Their duals have codimension greater than one, so they do not produce
divisorial components of the Gauss branch locus and are not recovered
by the preceding biduality argument. Linear components of $\widehat{C}$ are classified into types $0,\ldots, 3$ in Definition~\ref{def: type}. Using the nonlinear components and the
sets $\mathcal S_{\widehat C}$ and $\mathcal B_{\widehat C}$ recovered
in the preceding two steps, Lemma~\ref{lem:reconstruct-type-one-two} reconstructs the components of
types $1$ and $2$ through intrinsic incidence conditions. The
components of type $0$ are  recovered as the lines containing at
least three of the remaining points of $\mathcal B_{\widehat C}$.

The reconstruction of the components of type $3$ is more delicate.
Their presence is detected by a collinearity--incidence criterion on
triples of points of $\mathcal S_{\widehat C}$, which is reduced by
projection to the corresponding problem for an auxiliary curve of
smaller complexity. \emph{Nodal $1$-schemes} (curves equipped with two
labeled sets of isolated points) provide the inductive framework for
this argument, allowing curves of smaller genus and isolated points to
be treated uniformly. This yields Proposition~\ref{prop:reconstruct-linear} and its specialization to curves,
Proposition~\ref{prop:curves-reconstruct-linear}, and completes the
proof outline for Theorem~A.

The compactified Jacobian $\Pic$ of a connected nodal curve $C$ carries a  canonical involution. We let $\overline{P^E_C}\subset \Pic$ be the closure of all fixed orbits of maximal dimension. This is a distinguished irreducible component of $\Pic$ if and only if the dual graph of $C$ is Eulerian, otherwise $\overline{P^E_C}$ is empty. The restriction of $\Theta(C)$  to $\overline{P^E_C}$ gives rise to a stable semi-abelic pair $A^E(C)$  (see Remark~\ref{rem:S-definition} for more details). We obtain the following improvement of Theorem A.

\begin{thmc}[Theorem \ref{thm:main-Euler}] Let $C$ and $C'$ be connected nodal curves. Assume that the dual graph of $C$  is Eulerian. Then there is a natural function
\[
\rho^E_{C,C'}\colon \Isom_{\SSAP}(A^E(C),A^E(C')) \to \Isom_{\CT}(\Can(C), \Can(C')).
\]
\end{thmc}

Next, we compare canonical-triple isomorphisms with $T$-isomorphisms. A node of a curve is \emph{separating} if its normalization disconnects the curve. Two
nonseparating nodes of a curve form a \emph{separating pair} if their simultaneous normalization disconnects the curve. Separating pairs generate an equivalence
relation on the set of nonseparating nodes.
 A
\emph{$T$-isomorphism} between connected stable curves without separating nodes $C$ and $C'$ is an isomorphism of their
normalizations carrying the collection of reduced divisors lying over these
equivalence classes bijectively onto the corresponding collection for $C'$. Two curves are Torelli equivalent if and only if there is a $T$-isomorphism between them (see Definition \ref{def:Torelli-eq}). 
 Write $\Isom_{\TT}(C,C')$ for the set of $T$-isomorphisms. 

\begin{thmd}[Theorem~\ref{thm:T-equivalence}]
Let $C$ and $C'$ be connected stable curves without separating nodes. There is
a natural surjection
\[
\pi_{C,C'}\colon\Isom_{\TT}(C,C')\longrightarrow
\Isom_{\CT}(\Can(C),\Can(C')).
\]
Its fibers are described explicitly in terms of hyperelliptic involutions on
subcurves of $C$. In particular, $C$ and $C'$ are Torelli equivalent if and
only if $\Can(C)\cong\Can(C')$.
\end{thmd}

 Notice that Theorems A and D imply that if there is an isomorphism of SSAPs $A(C)\cong A(C')$ for connected stable curves $C$ and $C'$ without separating nodes, then $C$ and $C'$ are Torelli equivalent (see Corollary \ref{cor:CV}). This recovers  the reconstruction implication of Caporaso and Viviani \cite{Cap-Viv}. 

Finally, for a smooth curve, the strong form of the classical Torelli theorem
asserts that an isomorphism of principally polarized Jacobians is
induced by an isomorphism of curves, up to the sign involution
\cite[Theorem~12.1]{MilneJac}. Passing to the canonical triple of a connected nodal curve $C$ with no separating nodes, we show that the kernel of the action of $\operatorname{Aut}_{\mathrm{SSAP}}(A(C))$
on $\operatorname{Can}(C)$ is generated by the canonical involution
of $A(C)$, and that every automorphism of $\Can(C)$ lifts to an
automorphism of $A(C)$. As a consequence, we give a formula comparing the automorphism group  of the SSAP $A(C)$ with that of the curve $C$ assuming that the dual graph of $C$ is 3-edge connected. This generalizes the classical statement for smooth curves.

\begin{thme}[Theorem~\ref{thm:main-aut}]
 Let $C$ be a connected nodal curve of genus $g \geq 2$ with no separating nodes. Then there is a short exact sequence of groups
\[
\{\id\}\to \langle\iota_{A(C)}\rangle \to \Aut_{\SSAP}(A(C)) \to \Aut_{\CT}(\Can(C)) \to \{\id\},
\]
where
\(\langle\iota_{A(C)}\rangle\cong\mathbb Z/2\mathbb Z\) and $\iota_{A(C)}$ is the canonical involution of $A(C)$. In particular, if the dual graph of $C$ is also $3$-edge-connected, we have 
\[
\Aut_{\SSAP}(A(C))\cong
\begin{cases}
    \begin{array}{ll}
      \Aut(C)\times \mathbb Z/2\mathbb Z,   &  \text{ if $C$ is not hyperelliptic;}  \\
       \Aut(C),   &  \text{ if $C$ is hyperelliptic.} 
    \end{array}
\end{cases}
\]
\end{thme}

 As auxiliary tools, we also establish a  general-position theorem for reducible canonical curves (Theorem~\ref{thm:general_position}), and a detailed analysis of different Gauss maps (Section~\ref{sec:gauss-graph}). 



\subsection{Acknowledgments} To be added after the refereeing process.

\subsection{AI disclosure} 
Generative AI suggested a strategy for proving Lemma~\ref{lem:forgetful-aut} and helped us refine our argument
for Theorem~\ref{thm:T-equivalence}. It also assisted with the
exposition in the introduction and with checking language and
grammar throughout the paper.

\section{Preliminaries}\label{sec:preliminaries}

\subsection{Basic notions on graphs.} \label{subsec:graphs}

A \emph{graph} means a finite weighted undirected multigraph, possibly disconnected. 

Let $\Gamma$ be a graph. We let $E(\Gamma)$ and $V(\Gamma)$ be the set of edges and vertices of $\Gamma$, respectively. We denote by $w_\Gamma\colon V(\Gamma)\to \mathbb Z_{\ge0}$  the weight function of $\Gamma$. We let $c_{\Gamma}$ be the number of connected components of $\Gamma$. 
The first Betti number of $\Gamma$ is $b_1(\Gamma):=|E(\Gamma)|-|V(\Gamma)|+c_\Gamma$. The \emph{(arithmetic) genus} of  $\Gamma$ is 
\[
g_{\Gamma}:=\sum_{v\in V(\Gamma)}w_{\Gamma}(v)+b_1(\Gamma)+1-c_\Gamma.
\]

Consider a vertex $v\in V(\Gamma)$. We let $\Lo_\Gamma(v)$ be the set of loops attached to $v$. We set $g_v:=w_\Gamma(v)+|\Lo_\Gamma(v)|$. 
We denote by $\deg(v)$ the valence of $v$ in $\Gamma$ (with loops counting twice). 

An edge $e \in E(\Gamma)$ is \emph{separating} if its removal causes the number of connected components of $\Gamma$ to increase by $1$. Two distinct edges form a   \emph{separating pair} if  each of them is not separating and their removal causes the number of connected components of $\Gamma$ to increase by $1$.   
If $\Gamma$ is connected and it has no separating edges, then we say that $\Gamma$ is \emph{$2$-edge connected}.

Given a subset $V\subseteq V(\Gamma)$, we set $V^c:=V(\Gamma)\setminus V$.
 Given subsets
$V_1,V_2\subseteq V(\Gamma)$, let $E(V_1,V_2)$ be the set of edges of $\Gamma$ that join a vertex of $V_1$ with a vertex of $V_2$.  A \emph{cut} of $\Gamma$ is a subset $K\subseteq E(\Gamma)$ such that $K=E(V,V^c)$ for some $\emptyset \neq V\subsetneq V(\Gamma)$.  
 Set $\delta_V:=|E(V,V^c)|$. If $V=\{v\}$, we write $\delta_v$ instead of $\delta_{\{v\}}$. Set $k_v:=2g_v-2+\delta_v$ for every $v\in V(\Gamma)$. If $g_\Gamma\ge2$, we say that $\Gamma$ is \emph{stable} if $k_v>0$ for every $v\in V(\Gamma)$ (note that $\Gamma$ need not be connected).

We say that the graph $\Gamma$ is a \emph{chain} if we can write $E(\Gamma)=\{e_1,\dots,e_{n}\}$ and $V(\Gamma)=\{v_0,\dots,v_{n}\}$ where $e_i$ joins $v_{i-1}$ and $v_i$ for $i=1,\dots,n$, and the vertices are distinct. The vertices $v_0$ and $v_n$ are the end-vertices. A chain has first Betti number equal to $0$, and it has exactly two vertices of valence $1$, which we call the \emph{end-vertices} of the chain; the other vertices of the chain are called \emph{internal}. 
We say that $\Gamma$ is \emph{Eulerian} if $\deg(v)$ is even for every $v\in V(\Gamma)$, and that it is a \emph{cycle} if it is connected, Eulerian, and $b_1(\Gamma)=1$.

A \emph{trail}  $P$ on $\Gamma$ is a sequence $v_0, e_1, v_1, e_2, v_2, \ldots, v_{n-1}, e_n, v_n$ for some $n \geq 0$, such that $e_i$ is an edge between the vertices $v_{i-1}$ and $v_i$ (note that $v_{i-1} = v_i$ if, and only if, $e_i$ is a loop), with $e_i\neq e_j$ for every $i\neq j$. We say that $P$ \emph{starts at} $v_0$ and \emph{ends} at $v_n$, or that $v_0$ is the \emph{initial vertex} of $P$ and $v_n$ is the \emph{end vertex} of $P$. The underlying subgraph of a trail $P$ is the subgraph of $\Gamma$, whose set of vertices is $\{v_0,\ldots, v_n\}$ and set of edges is $\{e_1,\ldots, e_n\}$. We will denote this graph by $\Gamma(P)$, but sometimes we will also abuse the notation and simply write $P$ for the subgraph. 

An \emph{orientation} on $\Gamma$ is a pair $\orient=(\sigma_{\orient},\tau_{\orient})$, where $\sigma_{\orient}$ and $\tau_{\orient}$ are functions  $\sigma_{\orient},\tau_{\orient}\colon E(\Gamma)\to V(\Gamma)$ such that, for every $e\in E(\Gamma)$, the set $\{\sigma_{\orient}(e),\tau_{\orient}(e)\}$ is the set of vertices of $\Gamma$ incident to $e$. 

Notice that a trail $P$ naturally induces an orientation $\orient$ on the subgraph $\Gamma(P)$, where $\orient$ is defined by the rule $\sigma_{\orient}(e_i)=v_{i-1}$. On the other hand, given an orientation $\orient$ on $\Gamma$, we say that a trail $P = (v_0, e_1, v_1, \ldots, e_n, v_n)$ is \emph{$\orient$-oriented} if $\sigma_{\orient}(e_i)=v_{i-1}$.

Let $\orient=(\sigma_{\orient},\tau_{\orient})$ be an orientation on $\Gamma$.
The \emph{dual orientation} $\orient^*$ of  $\orient$ is defined as $\orient^*=(\tau_{\orient},\sigma_{\orient})$. We set
\begin{align*}
    \sigma_{\orient,v}:=|\{e\in E(\Gamma) : \sigma_{{\orient}}(e)=v\}| \\
    \tau_{\orient,v}:=|\{e\in E(\Gamma) : \tau_{{\orient}}(e)=v\}|.
\end{align*}
We also set 
\begin{equation}\label{eq:prime}
\sigma'_{\orient,v}:=\sigma_{\orient,v}-|\Lo_\Gamma(v)|
\text{ \; and \; }
\tau'_{\orient,v}:=\tau_{\orient,v}-|\Lo_\Gamma(v)|.
\end{equation}

Given a cut $K=E(W,W^c)$ of $\Gamma$, for some $W\subseteq V(\Gamma)$, we say that $K$ is an \emph{${\orient}$-cut} of $\Gamma$ if either $\sigma_{\orient}(e)\in W$ for every $e\in K$, or $\tau_{\orient}(e)\in W$ for every $e\in K$. 
The orientation ${\orient}$ on $\Gamma$ is \emph{strong} if there are no ${\orient}$-cuts. Given $v_1,v_2\in V(\Gamma)$ we say that \emph{$v_1$ is ${\orient}$-connected to $v_2$} if there exists an $\orient$-oriented trail $P$ starting at $v_1$ and ending at $v_2$. 
We say that $v_1$ and $v_2$ are \emph{strongly ${\orient}$-connected} if $v_1$ is $\orient$-connected to $v_2$, and $v_2$ is $\orient$-connected with $v_1$. 
Recall that an orientation ${\orient}$ is strong if and only if  $v_1$ and $v_2$ are strongly ${\orient}$-connected for every $v_1,v_2\in V(\Gamma)$. Moreover, a graph admits a strong orientation if and only if it is $2$-edge connected (Robbins' theorem); in particular, in this case, it is connected. 

Consider a subset $E\subseteq E(\Gamma)$. We denote by $\Gamma/E$ the (unweighted) graph obtained by contracting the edges in $E$. Thus we can identify $E(\Gamma/E)$ with $E(\Gamma)\setminus E$. We let  $\Gamma \setminus E$ be the graph obtained by removing the edges in $E$ from $\Gamma$ and keeping the vertex set unchanged. Thus $E(\Gamma\setminus E)=E(\Gamma)\setminus E$ and $V(\Gamma\setminus E)=V(\Gamma)$.

Given a subset $V\subseteq V(\Gamma)$, we denote by $\Gamma[V]$ the induced subgraph on $V$ and $g_V:=g_{\Gamma[V]}$. We denote by $t_V$ the number of connected components of $\Gamma[V^c]$, i.e., $t_V:=c_{\Gamma[V^c]}$. If $V=\{v\}$ for $v\in V(\Gamma)$, we simply write $t_v$ instead of $t_{\{v\}}$. 

A \emph{multidegree} $\ud$ on $\Gamma$ is an element $\ud=(d_v)\in \mathbb Z^{V(\Gamma)}$. The \emph{degree} of $\ud$ is defined as $|\ud|=\sum_{v\in V(\Gamma)} d_v$. 
 The multidegree of an orientation $\orient$ on $\Gamma$ is defined as $\ud_{\orient} :=(d_{\orient,v})_{v\in V(\Gamma)}\in \mathbb Z^{V(\Gamma)}$, where
\begin{equation}\label{eq:d-orient}
d_{\orient,v}:=w_{\Gamma}(v)-1+\tau_{\orient,v}=g_v-1+\tau'_{\orient,v}.
\end{equation}
Notice that $\ud_{\orient}$ has degree $g_\Gamma-1$.

For a connected graph $\Gamma$, we let $\Sigma(\Gamma)$ be the set of  multidegrees induced by strong orientations, i.e., we let
\begin{equation}\label{eq:Sigma-strong}
\Sigma(\Gamma):=\{\ud_{\orient} \in \mathbb  Z^{V(\Gamma)} : \orient \text{ is a strong orientation on } \Gamma \}.
\end{equation}
Notice that $\Sigma(\Gamma)$ is empty if and only if $\Gamma$ is not $2$-edge connected.

Let $V \subseteq V(\Gamma)$.  We  define $d_{\orient, V} := \sum_{v\in V}d_{\orient,v}$, so we have 
\begin{equation}\label{eq:d-o-expression}
d_{\orient, V} = \sum_{v\in V}w_{\Gamma}(v) - |V| + \sum_{v\in V}\tau_{\orient, v}.
\end{equation}
We also set 
\[
\tau'_{\orient, V} := |\{e\in E(\Gamma); \tau_{\orient}(e)\in V, \sigma_{\orient}(e)\in V^c\}|,
\]
which yields the relation
\[
\sum_{v\in V}\tau_{\orient, v} = |E(V, V)| + \tau'_{\orient, V}.
\]

By combining with Equation \eqref{eq:d-o-expression}, we obtain
\begin{equation} \label{eq:another-strong}
d_{\orient, V} = \sum_{v\in V}w_{\Gamma}(v) - |V| + |E(V, V)| + \tau'_{\orient, V} = g_V - 1 + \tau'_{\orient, V}.
\end{equation}
If $\orient$ is strong, then we have $t_V\leq \tau'_{\orient, V}\leq\delta_V - t_V$, which yields the inequalities
\begin{equation}
\label{eq:strong}
g_V-1+t_V\leq d_{\orient,V}\leq g_V-1+\delta_V-t_V,
\end{equation}
for every subset $\emptyset \neq V \subsetneq V(\Gamma)$.

\begin{definition}  \label{d: C1set} Suppose that the graph $\Gamma$ is $2$-edge connected. A subset $S\subseteq E(\Gamma)$ is a \emph{{$\Cone$}-set} of $\Gamma$ if $\Gamma/(E(\Gamma) \setminus S)$ is a cycle and $\Gamma\setminus S$ contains no separating edges.
\end{definition} 
See Figure \ref{fig:C_1-set} for an example of a $\Cone$-set $S$ with 6 edges (the gray vertices denote the connected components of $\Gamma\setminus S$).

\begin{remark} \label{rem:propC1set} 
The following properties hold.
\begin{itemize}
    \item[(1)] The $\Cone$-sets of $\Gamma$ form a partition of $E(\Gamma)$ (see \cite[Remark 2.1.4]{Cap-Viv}).
    \item[(2)] Two distinct edges $e$ and $e'$ in $E(\Gamma)$ belong to the same $\Cone$-set if and only if $\{e,e'\}$ is a separating pair (see \cite[Lemma 2.3.2 (iv)]{cap-viv2}).
\end{itemize}
\end{remark}
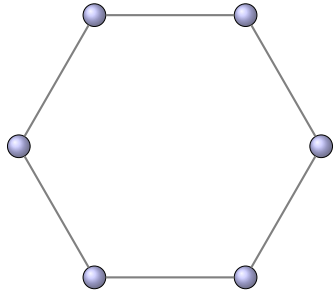
\begin{figure}[ht]
\centering
\begin{tikzpicture}[scale=0.5]  
    \foreach \i in {0,1,...,5} {
        \pgfmathsetmacro{\angle}{\i * 60}
        \pgfmathsetmacro{\x}{4*cos(\angle)}
        \pgfmathsetmacro{\y}{4*sin(\angle)}
        
        \shade[ball color=blue!30!white, opacity=0.8] (\x,\y) circle (0.3cm);
        
        \draw[very thin] (\x,\y) circle (0.3cm);
    }
    
    \foreach \i in {0,1,...,5} {
        \pgfmathsetmacro{\angleA}{\i * 60}
        \pgfmathsetmacro{\angleB}{mod(\i+1,6) * 60}
        \pgfmathsetmacro{\xA}{4*cos(\angleA)}
        \pgfmathsetmacro{\yA}{4*sin(\angleA)}
        \pgfmathsetmacro{\xB}{4*cos(\angleB)}
        \pgfmathsetmacro{\yB}{4*sin(\angleB)}
        
        \pgfmathsetmacro{\dirAngle}{atan2(\yB-\yA, \xB-\xA)}
        \pgfmathsetmacro{\startX}{\xA + 0.3*cos(\dirAngle)}
        \pgfmathsetmacro{\startY}{\yA + 0.3*sin(\dirAngle)}
        \pgfmathsetmacro{\endX}{\xB - 0.3*cos(\dirAngle)}
        \pgfmathsetmacro{\endY}{\yB - 0.3*sin(\dirAngle)}
        
        \draw[thick, black!50] (\startX,\startY) -- (\endX,\endY);
    }
\end{tikzpicture}
\caption{A $\Cone$-set with $6$ edges.}
\label{fig:C_1-set}
\end{figure}

\subsection{Schemes, ramification and branch locus}\label{sec:schemes}

  Let $k$ be an algebraically closed field of characteristic $0$. 
  Our schemes will be assumed to be  separated and of finite type over $k$. A property of points on a scheme is said to hold \emph{generically} if it holds for every point on a Zariski-open dense subset. We will set $\mathbb P^N:=\mathbb P^N_k$. 

  Let $X$ be a reduced scheme. We denote by $X^{\sm}$ be the smooth locus of $X$, and by $X^{\sing}$ the singular locus of $X$. We let $\Irr(X)$ be the set of irreducible components of $X$. 
  If $X$ is embedded in $\mathbb P^N$, we set 
  \begin{equation} \label{eq:Lambda}
  \Lambda_X:=\spann(X) 
  \;\text{ and } \;
  N_X:=\dim(\Lambda_X).
  \end{equation}
   Moreover, for every positive integer $d$ we set
\begin{equation}\label{eq:Irr=d}
\Irr_d(X):=\{Z\in \Irr(X) : \deg(Z)=d\}; 
\end{equation}
\begin{equation}\label{eq:Irrged}
\Irr_{\ge d}(X):=\{Z\in \Irr(X) : \deg(Z)\ge d\}.
\end{equation} 
We also set
\begin{equation} \label{eq:lin}
\Lin(X):=\{Z \in \Irr_1(X): \dim(Z)=1\}.
\end{equation}

Given a morphism of schemes $f\colon X\to Y$ and an open subset $U\subseteq Y$, we set $X_U := f^{-1}(U)$ and we denote by $f_U\colon X_U\to U$ the restriction of $f$ to $X_U$. 
  
Let $f\colon X\to Y$ be a proper morphism of integral $k$-schemes. Given $y\in Y$, we let $X_y = X\times_Y y $ be the scheme-theoretic fiber of $f$ over $y$.
We define $\Br_f\subseteq Y$ as the closure in $Y$ of the locus $Z_f$ of the $k$-points $y\in Y$ such that the fiber $X_y$ is a zero-dimensional singular scheme. 
When $f$ is dominant, we define $\Br^1_f$ as the union of the components of $\Br_f$ of codimension 1 in $Y$.

 Let $U\subseteq Y$ be the maximal open subset of $f(X)$ over which $f$ is finite. In particular, $U$ is open (possibly empty) in the image of $f$, by the theorem of dimension of the fibers. 
 Notice that $Z_f$ is closed in $U$ (recall that $\chara(k)=0$). 
We define $\Ram_f\subseteq X$ as the closure of the locus $W_f\subseteq X$ consisting of the points $x\in f^{-1}(Z_f)$ such that $X_{f(x)}$ is singular at $x$. Notice that $W_f$ is closed in the open subset $X_U$ of $X$. Moreover, since $f$ is proper, we have that $f(\Ram_f)=\Br_f$.

We can extend these definitions to the case where $X$ is any reduced scheme by setting $\Br_f:=\bigcup_{X_i} \Br_{f|_{X_i}}$ and $\Ram_f:=\bigcup_{X_i}\Ram_{f|_{X_i}}$, where the union runs over all irreducible components $X_i$  of $X$.

\begin{proposition} \label{prop:stein-factor-branch}
Let $f\colon X \to Y$ be a proper dominant morphism between integral $k$-schemes. Let $X\to Z\xrightarrow{g} Y$ be the Stein factorization of $f$. Then $\Br_f\subseteq\Br_g$. If moreover $\dim(X)=\dim(Y)$, then $\Br^1_f = \Br^1_g$.
\end{proposition}
\begin{proof}
Let $U\subset Y$ be the maximal open subset over which $f$ is finite.  
    In particular, $X_U\to Z_U$ is an isomorphism and     $Z_f=Z_g\cap U$. We deduce $\Br_f\subseteq\Br_g$ and $\Br_{f_U}=\Br_{g_U}$. 

    Assume now that  $\dim(X)=\dim(Y)$. 
    Since $X$ and $Y$ are irreducible of the same dimension and $f$ is proper and dominant, we have that $\codim_Y(Y\setminus U)\geq 2$, hence
    \[
    \Br^1_f = \Br^1_{f_U} = \Br^1_{g_U} = \Br^1_g.
    \]
    This finishes the proof.
\end{proof}

\begin{theorem}[Purity of the branch locus]\label{thm:purity} 
    Let $f\colon X\to Y$ be a proper dominant morphism of integral $k$-schemes with $X$ normal and $Y$ smooth. Then $\Br_f = \Br^1_f$. 
\end{theorem}
\begin{proof}
    If $\dim(X)>\dim(Y)$, then $\Br_f=\emptyset$ and there is nothing to prove. Hence, we will assume $\dim(X)=\dim(Y)$. Using the notation of Proposition~\ref{prop:stein-factor-branch}, the normality of $X$ implies that $Z$ is also normal. The result follows from the usual purity theorem (\cite[Expos\'e X, Th\'eor\`eme 3.1]{SGA1}) applied to the finite dominant morphism $Z \to Y$.
\end{proof}

 Now assume that $f\colon X\to Y$ is a finite dominant morphism of integral $k$-schemes. Let $d_f=[K(X):f^* K(Y)]$ be the cardinality of a general geometric fiber of $f$. 

\begin{proposition}
\label{prop:branch_fiber_number}
    Let $f\colon X\to Y$ be a finite dominant morphism of integral $k$-schemes with $X$ Cohen-Macaulay and $Y$ smooth. Then $f$ is flat. Moreover a $k$-point $y\in Y$ belongs to $\Br_f$ if and only if
\(
|f^{-1}(y)|<d_f.
\)  
\end{proposition}

\begin{proof}
    The fibers of $f$ over $k$-points of $Y$ are all $0$-dimensional, $X$ is Cohen-Macaulay, and $Y$ is smooth. Thus $f$ is flat by miracle flatness. 
    
    We have $\sum_{x\in f^{-1}(y)}\len(\mathcal{O}_{X_y,x})=d_f$. In particular, $X_y$ is singular if and only if $f^{-1}(y)$ has fewer than $d_f$ points. 
\end{proof}

\subsection{Abstract curves}
\label{sec:abstract_curves}
 Recall that we work over an algebraically closed field $k$ of characteristic $0$. A \emph{curve} is a (not necessarily connected) reduced projective scheme of pure dimension $1$ over $k$. A curve is \emph{nodal} if its singularities are nodes, i.e., \'etale locally like the origin in the union of the coordinate axes of the plane $\mathbb A^2_k$. 
 
Let $C$ be a curve over $k$.  A \emph{subcurve} of $C$ is a union of irreducible components of $C$. Given a subcurve $Z$, its \emph{complementary curve} $Z^c$ is defined as the closure $Z^c:=\overline{C\setminus Z}$. We let $t_Z$ be the number of connected components of $Z^c$ and we denote by $F^Z_1,\dots F^Z_{t_Z}$ the connected components of $Z^c$. 

\begin{convention} \label{conv:curve-part}
Let $f\colon X\to Y$ be a finite morphism, where $X$ is a curve, and let $Z\subseteq Y$ be a  subcurve. We
write $f^{-1}(Z)$ for the reduced one-dimensional part of the
scheme-theoretic inverse image of $Z$.
\end{convention}

For a Weil divisor $D$ on $C$, we write $D=\sum_{p\in C}\mu_p(D)\cdot p$, for integers $\mu_p(D)$, and we call the \emph{support of} $D$ the subset of points $p\in C$ such that $\mu_p(D)\ne0$.

Assume that $C$ is a nodal curve. The \emph{(arithmetic) genus} of $C$ is defined as $g_C:=1-\chi(C,\mathcal O_C)$. We denote by $\omega_C$ the dualizing sheaf of $C$. 
For every subcurve $Z\subseteq C$ we set 
\begin{equation}\label{eq:kZ}
k_Z:=\deg(\omega_C|_Z)=2g_Z-2+|Z\cap Z^c|.
\end{equation}

 We let $\Gamma_C$ be the \emph{dual graph} of $C$, whose vertices correspond to the irreducible components of $C$ and whose edges correspond to the nodes of $C$, where the edge corresponding to a node connects the (possibly coinciding) vertices corresponding to the two (possibly equal) components intersecting at that node. Given a vertex $v\in V(\Gamma_C)$, we denote by $C_v$ the irreducible component of $C$ corresponding to $v$. Finally, the weight function $w_{\Gamma_C}$ takes $v\in V(\Gamma_C)$ to the  genus of the normalization of $C_v$. So, $g_C=g_{\Gamma_C}$. The (not necessarily connected) nodal curve $C$ is \emph{stable} if its dual graph $\Gamma_C$ is stable. 

A node of $C$ is \emph{separating} if its corresponding edge in $E(\Gamma_C)$ is  separating. We let $C^{\sep}\subseteq C^{\sing}$ be the subset of separating nodes of $C$. Two distinct nodes of $C$ form a \emph{separating pair} of nodes if their corresponding edges in $E(\Gamma_C)$ form a separating pair of edges. If $Z\subseteq C$ is a subcurve and $V_Z\subseteq V(\Gamma_C)$ is the subset corresponding to the irreducible components of $Z$, we set $\delta_Z := \delta_{V_Z}$. Hence $\delta_Z=|Z\cap Z^c|$.

For every multidegree $\ud=(d_v)\in \mathbb Z^{V(\Gamma_C)}$ and every subcurve $Z\subseteq C$, we set 
\begin{equation}\label{eq:dZ}
d_Z:=\underset{C_v\subseteq Z}{\underset{v\in V(\Gamma_C)}{\sum}}d_v.
\end{equation}

Assume now that $C$ is nodal and \emph{connected}. We set $\Sigma(C):=\Sigma(\Gamma_C)$ (recall Equation~\eqref{eq:Sigma-strong}). A multidegree $\ud\in \mathbb Z^{V(\Gamma_C)}$ is \emph{stable} if $\ud\in \Sigma(C)$. For every $\ud\in \Sigma(C)$ and every connected subcurve $Z\subsetneq C$, using Equation~\eqref{eq:strong}, we have
\begin{equation}\label{eq:dZ-stable}
g_Z-1<d_Z<g_Z-1+\delta_Z
\end{equation}
(and  $|\ud|=g_C-1$). By \cite[Section~1.3]{captheta}, a multidegree $\ud\in \mathbb Z^{V(\Gamma_C)}$ belongs to $\Sigma(C)$ if and only if $|\ud|=g_C-1$ and Equation~\eqref{eq:dZ-stable} holds for every  subcurve $Z\subsetneq C$.

\subsection{Projective duality}

Given a linear subspace $W\subseteq\mathbb P^N$, we let $\Lambda^\vee_W\subseteq (\mathbb P^N)^\vee$ be defined as
\[
\Lambda^\vee_W:=\{H\in (\mathbb P^N)^\vee : W\subseteq H\}.
\]
For a subset $S\subseteq \mathbb P^N$, we set $\Lambda^\vee_S:=\Lambda^\vee_{\spann(S)}$.

Let $V\subseteq \mathbb{P}^N$ be a (possibly reducible) \emph{subvariety} (i.e., a reduced closed subscheme in $\mathbb{P}^N$). Let $CV^{\sm}$ be the set of pairs $(p, H)\in \mathbb{P}^N\times (\mathbb{P}^N)^\vee$ such that $p\in V^{\sm}$ and $H$ contains the embedded tangent space $T_pV$. The \emph{conormal variety} $CV$ is the closure of $CV^{\sm}$ in $\mathbb{P}^N\times (\mathbb{P}^N)^\vee$. Notice that, since $V$ is reduced, $V^{\sm}$ is an open dense subset of $V$. 

\begin{definition}
    The \emph{dual variety} $V^\vee$ of $V$ is the image of $CV$ via the projection $\mathbb{P}^N\times (\mathbb{P}^N)^\vee\to (\mathbb{P}^N)^\vee$ onto the second factor.
\end{definition}

Notice that if $V=\bigcup_{1\le i\le k} V_i$ is the decomposition of $V$ into irreducible components, then $CV = \bigcup_{1\le i\le k} CV_i$, and  hence 
\begin{equation}\label{eq:dual-union}
   V^\vee=\bigcup_{1\le i\le k} V_i^\vee. 
\end{equation}

\begin{proposition}\label{prop:dual}
    Let $V\subseteq \mathbb{P}^N$ be a curve with nodal irreducible components. Let $\widetilde{V}\to V$ be its normalization and denote by $\psi\colon \widetilde{V}\to \mathbb{P}^N$ the induced map. We have that $H\notin V^\vee$ if and only if $\psi^{-1}(H)$ is reduced and has dimension $0$.
\end{proposition}

\begin{proof}
    If $N=1$ the result is trivial. So assume $N\ge2$. By Equation \eqref{eq:dual-union}, we can reduce to the case where $V$ is irreducible. In this case, $CV\to V$ has relative dimension $N-2$. Indeed, given $p\in V^{\sm}$ and $(p,H)\in CV$, the hyperplane $H$ must contain the line $T_pV$ and taking the closure cannot make the dimension smaller (by the theorem of dimension of the fibers), nor greater (because then this fiber would have the same dimension as $CV^{\sm}$). 
    Define 
    \[
    C\widetilde{V} := \{(p, H) : \psi^{-1}(H) \text{ is not reduced at $p$, or } \widetilde{V}\subseteq \psi^{-1}(H))\} \subseteq \widetilde{V}\times (\mathbb{P}^{N})^\vee.
    \]
    
    Set $f=(\psi,\id)\colon \widetilde{V}\times (\mathbb{P}^N)^\vee\to V\times (\mathbb{P}^N)^\vee$. Notice that $CV^{\sm}\hookrightarrow C\widetilde{V}$, then $CV^{\sm}\subseteq f(C\widetilde{V})$, hence $CV\subseteq f(C\widetilde{V})$. On the other hand, since $V$ is nodal, the fibers of $C\widetilde{V}\to \widetilde{V}$ are isomorphic to $\mathbb{P}^{N-2}$, then $C\widetilde{V}$ is irreducible of the same dimension as $CV$, which implies that $CV=f(C\widetilde{V})$. This finishes the proof.
\end{proof}

\begin{proposition}\label{prop:Kleiman_dual}
    Let $V\subseteq \mathbb{P}^N$ be a curve of $\mathbb{P}^N$ containing no linear components. Then $(V^\vee)^\vee = V$.
\end{proposition}

\begin{proof}
    Let $V = \bigcup_{1\le i\le k} V_i$ be the decomposition into irreducible components of $V$. Since $V_i$ is a non-linear curve in $\mathbb{P}^N$, we have that $V_i^\vee$ is an irreducible hypersurface of $(\mathbb{P}^N)^{\vee}$, so Equation \eqref{eq:dual-union} is the decomposition of $V^\vee$ into
    irreducible components. Hence $(V^\vee)^\vee = \bigcup_{1\le i\le k} (V_i^\vee)^\vee$ and the result follows from \cite[Theorem~4]{kleiman} (recall that we are assuming that $k$ has characteristic $0$).
\end{proof}

\section{Nodal 1-schemes and their canonical triples}\label{sec:proj}

\subsection{Nodal 1-schemes}

This section introduces \name{}s and their canonical triples (see Definitions \ref{def:nodal-1-schem} and \ref{def:canonical_model}). A nodal 1-scheme generalizes nodal curves by allowing two labeled sets of isolated points. This formulation is convenient because, unlike the class of nodal curves, the class of \name{s} is closed under the operation of taking stable models. Here we define the fundamental invariants and operations of \name{s}, with a particular focus on the canonical map to projective space. We analyze their base locus, their behavior under stabilization, and when their canonical map fails to be injective.

\begin{definition}
    \label{def:nodal-1-schem}
A \emph{\name{}} $\mathfrak{C}$ is a triple $(A, B, C)$, where $C$ is a (possibly disconnected, possibly empty) nodal curve, and $A$ and $B$ are reduced zero-dimensional schemes over $k$. 
The underlying scheme of a \name{} $\mathfrak{C}=(A, B, C)$ is $A\sqcup B\sqcup C$. Abusing notation, we denote by $\mathfrak{C}$ the underlying scheme.
\end{definition}

 Any algebro-geometric object on $\mathfrak{C}$ will be defined as the object on its underlying scheme (e.g., invertible sheaves on $\mathfrak{C}$ are invertible sheaves on its underlying scheme). As it is customary for curves, we denote by $\omega_{\mathfrak{C}}$ the unique line bundle on $\mathfrak{C}$ whose restriction to each connected component is its dualizing sheaf.   We set $h_{\mathfrak C}:=h^0(\mathfrak{C}, \omega_{\mathfrak{C}})$. 
\begin{remark}
\label{rem:h_sum_g}
    If $\mathfrak{C} = (A, B, C)$ is a \name{}, then 
    \[
    h_{\mathfrak{C}} = |A| + |B| + \sum_{C'}g_{C'},
    \]
    where the sum runs through all connected components $C'$ of $C$. 
    In particular, if $A=B=\emptyset$ and $C$ is connected, then \[h_\mathfrak{C} = h_C = g_C=h^1(C, \mathcal{O}_C)=h^0(C, \omega_C).\]
\end{remark}

Let $\mathfrak C=(A, B, C)$ be a \name{}. Notice that $\mathfrak C^{\sing}=C^{\sing}$. A \emph{node} (respectively, a \emph{separating node}) of $\mathfrak{C}$ is a node (respectively, a separating node) of  $C$. A \emph{\subname{}} of $\mathfrak C$ is a \name{} $\mathfrak Z=(A', B', Z)$, where $Z$ is a subcurve of $C$, and where $A'$ and $B'$ are subschemes of $A$ and $B$. 

\begin{definition}\label{def:stabilization}
 Let $\mathfrak C=(A,B,C)$ be a \name{}.
 The \emph{stabilization} $\mathfrak{C}^{\st}$ of $\mathfrak{C}$ is defined as follows. 
Assume first that $\mathfrak{C}$ is  connected. In this case, only one among $A$, $B$, $C$ is non empty. Then we have two cases.
\begin{enumerate}
    \item[(1)] If $\mathfrak{C}=C$, then $C$ is a connected nodal curve. We have four sub-cases.
        \begin{enumerate}
            \item If $h_{\mathfrak C}\geq 2$, we set $\mathfrak{C}^{\st} :=( \emptyset, \emptyset,  C^{\st})$, where $C^{\st}$ is the usual stabilization of the curve $C$.
            \item If $h_{\mathfrak C}=1$ and $C$ is smooth, we set $\mathfrak{C}^{\st}:= ( \Spec(k), \emptyset, \emptyset)$.
            \item If $h_{\mathfrak C}=1$ and $C$ is singular, we set $\mathfrak{C}^{\st}:= (\emptyset, \Spec(k), \emptyset)$.
            \item If $h_{\mathfrak C}=0$, we set $\mathfrak{C}^{\st}:=(\emptyset, \emptyset, \emptyset)$.
        \end{enumerate}
    \item[(2)] In the other cases, that is, $\mathfrak{C}=A$ or $\mathfrak{C}=B$, we set $\mathfrak{C}^{\st} := \mathfrak{C}$.
\end{enumerate}
Notice that, except when $h_{\mathfrak C}=0$, we have a stabilization morphism $\mathfrak{C}\to \mathfrak{C}^{\st}$.

More generally, if $\mathfrak{C}$ is a not necessarily connected \name{}, the \emph{stabilization} $\mathfrak{C}^{\st}$ of $\mathfrak{C}$ is the disjoint union of the stabilizations of its connected components. Notice that we have $h_{\mathfrak{C}^{\st}} = h_{\mathfrak{C}}$.  

We say that $\mathfrak{C}$ is \emph{stable} if it is equal to its stabilization. If $\mathfrak{C}$ is a curve,  this agrees with the  definition of stability that we recalled in Subsection~\ref{sec:abstract_curves}. 
\end{definition}

\begin{remark}
    \label{rem:proj_C_span}
Given a \name{} $\mathfrak{C}=(A,B,C)$,
we set 
\[
\Proj_{\mathfrak{C}}:=\mathbb P(H^0(\mathfrak C,\omega_{\mathfrak{C}})^\vee),
\]
a projective space of dimension $h_{\mathfrak{C}}-1$. We have
\[
H^0(\mathfrak C,\omega_{\mathfrak{C}})^{\vee} = \bigoplus_{a\in A} H^0(a,\omega_a)^{\vee} \oplus \bigoplus_{b\in B}H^0(b,\omega_b)^{\vee} \oplus H^0(C,\omega_C)^{\vee}.
\]
Hence $\Proj_{\mathfrak{C}} = \spann(\{q_a, a\in A\}, \{q_b, b\in B\}, \Proj_C)$, where $q_a$ and $q_b$ are the points corresponding to  $H^0(a,\omega_{a})^{\vee}\subseteq H^0(\mathfrak C,\omega_{\mathfrak{C}})^{\vee}$ and  $H^0(b,\omega_{b})^{\vee}\subseteq H^0(\mathfrak C,\omega_{\mathfrak{C}})^{\vee}$ (recall that the dualizing sheaf of a reduced zero-dimensional scheme $p$ is trivial, hence $H^0(p, \omega_p)$ is one-dimensional). 
\end{remark}

\begin{definition}\label{def:canonical-map}
 Let $\mathfrak{C}=(A,B,C)$ be a \name{}.
 We let $U_{\mathfrak{C}}$ be the (not necessarily dense) open subset of $\mathfrak{C}$ where $\omega_{\mathfrak{C}}$ is base-point free. Notice that $U_{\mathfrak{C}}$ contains $A\sqcup B$. The \emph{canonical map} of  $\mathfrak{C}=(A,B,C)$ is the morphism $\varphi_{\mathfrak{C}}\colon U_{\mathfrak{C}}\to \Proj_{\mathfrak{C}}$ given by the complete linear system $|\omega_{\mathfrak{C}}|$. 
When no confusion arises, we denote  $\varphi_{\mathfrak{C}}$ simply by $\varphi$.
\end{definition}

 By the observation above, the following properties hold
\begin{enumerate}
    \item[(1)] $\varphi_{\mathfrak{C}}(a) = q_a$, for every $a\in A$.
    \item[(2)] $\varphi_{\mathfrak{C}}(b) = q_b$, for every $b\in B$.
    \item[(3)] $\varphi_{\mathfrak{C}}|_{U_{\mathfrak{C}}\cap C} = \iota_C\circ \varphi_C|_{U_C}$, where $\iota_C \colon \Proj_C\to \Proj_{\mathfrak{C}}$ is the natural inclusion, and $\varphi_C\colon U_C\to \Proj_C$ is the canonical map of $C$ (clearly, $U_C= U_{\mathfrak{C}}\cap C$).
\end{enumerate}

Therefore,  using \cite[Theorem D]{catanese},  we conclude that 
\begin{equation}\label{eq:base-locus}
{\mathfrak{C}}\setminus U_{\mathfrak{C}} = D_C \cup C^{\sep},
\end{equation}
where $D_C\subseteq C$ is the subcurve  given by the union of all irreducible components $Z$ of $C$ such that $g_Z=0$ and $Z\cap Z^c \subseteq C^{\sep}$. 
Moreover the \subname{} $\mathfrak{C}' = (A, B, \overline{C\setminus D_C})$ of $\mathfrak C$ satisfies $U_{\mathfrak{C}} \subseteq U_{\mathfrak{C}'}$ and $U_{\mathfrak{C}}$ an open dense subset of $\mathfrak C'$, and we have $\varphi_{\mathfrak{C}'}|_{U_{\mathfrak C}} = \varphi_{\mathfrak{C}}|_{U_{\mathfrak C}}$.

Next, assume that $E$ is an irreducible component of the \name{} $\mathfrak{C} = (A, B, C)$ such that  $g_E=0$ and $|E\cap E^c|=2$ (then $E$ has dimension~$1$ and $E\subseteq C$). Consider the \name{} $\mathfrak{C}'=(A, B, C')$ where $C'$ is obtained from $C$ by contracting $E$, and let $\psi_E\colon \mathfrak{C}\to \mathfrak{C}'$ be the contraction morphism. Then $\psi_E(U_{\mathfrak{C}})\subseteq U_{\mathfrak{C}'}$ and $\varphi_{\mathfrak{C}} = \varphi_{\mathfrak{C}'}\circ \psi_E|_{U_{\mathfrak{C}}}$. In particular, we have the following remark.

\begin{remark}    \label{rem:canonical_map_factors_stab}
If $\mathfrak{C}$ is  connected and $h_{\mathfrak{C}}\geq 1$, and if
$\mathfrak{C}^{\st}$ is the stabilization of $\mathfrak{C}$ with stabilization morphism $\psi\colon \mathfrak{C}\to \mathfrak{C}^{\st}$, then $\psi(U_{\mathfrak{C}})\subseteq U_{\mathfrak{C}^{\st}}$ and $\varphi_{\mathfrak{C}}|_{U_{\mathfrak C}}=\varphi_{\mathfrak{C}^{\st}}\circ \psi|_{U_{\mathfrak{C}}}$. On the other hand, if $\mathfrak{C}$ is connected and $h_{\mathfrak C}=0$, then $U_{\mathfrak{C}}=\Proj_{\mathfrak{C}}=\emptyset$, and $\psi$ is just the identity on the empty set.
\end{remark}

We collect the above considerations in the following result.

\begin{proposition}\label{prop:model-stabilization}
    Let $\mathfrak{C}$ be a \name{} and $\mathfrak{C}^{\st}$ be its stabilization. Let $\psi\colon U_{\mathfrak C}\to \mathfrak C^{\st}$ be the restriction of the stabilization morphism. Then $\psi(U_{\mathfrak{C}}) \subseteq U_{\mathfrak{C}^{\st}}$, the subset $\psi(U_{\mathfrak{C}})$ of $\mathfrak{C}^{\st}$ is dense, and $\varphi_{\mathfrak{C}}|_{U_{\mathfrak C}}=\varphi_{\mathfrak C^{\st}}\circ \psi$.  
\end{proposition}

\begin{notation}\label{not:dagger}
    Let $\mathfrak C=(A,B,C)$ be a 
 \name{}. Let $\widetilde{\mathfrak C}$ be the normalization of $\mathfrak C$
at all separating nodes. Write the decomposition into connected components as
\[
\widetilde{\mathfrak C}=A\sqcup B\sqcup \coprod_{1\le i\le r} C_i \sqcup \coprod_{1\le j \le s}R_j,
\]
where $C_i$ and $R_j$ are curves with $g_{C_i}>0$ and $g_{R_j}=0$ respectively. We let $\mathfrak C^\dagger$ be obtained from $\mathfrak{\widetilde C}$ by  discarding all genus-zero connected components, i.e., we set 
\[
\mathfrak C^\dagger=A\sqcup B\sqcup \coprod_{1\le i\le r} C_i.
\]
\end{notation}

\begin{remark}\label{rem:1-scheme-property}
    Let $\mathfrak{C}=(A,B,C)$ be a \name{}. One can easily check the following properties. 
    \begin{enumerate}
     \item[(i)] A \name{} $\mathfrak C$ has no separating nodes and no connected components $\mathfrak C'$ with $h_{\mathfrak C'}=0$ if and only if $U_{\mathfrak{C}}=\mathfrak{C}$.
    \item[(ii)] We have that $\mathfrak C^\dagger$ is a \name{} such that $U_{\mathfrak C^\dagger}=\mathfrak C^\dagger$.
        \item[(iii)] Assume that $\mathfrak{C}$ is the disjoint union of nodal 1-schemes $\mathfrak{C}_1$ and $\mathfrak{C}_2$.  There are  natural embeddings $\mathbb{P}_{\mathfrak C_1}\hookrightarrow \mathbb{P}_{\mathfrak C}$ and $\mathbb{P}_{\mathfrak C_2}\hookrightarrow \mathbb{P}_{\mathfrak C}$, which identify $\mathbb{P}_{\mathfrak C_1}$ and $\mathbb{P}_{\mathfrak C_2}$ with disjoint linear subspaces in $\mathbb{P}_{\mathfrak C}$ spanning $\mathbb{P}_{\mathfrak C}$. The canonical map $\varphi_{\mathfrak C}$ of $\mathfrak C$ is the composition
        \[
        \varphi_{\mathfrak C}\colon U_{\mathfrak C}=U_{\mathfrak{C}_1}\sqcup U_{\mathfrak{C}_2}\stackrel{\varphi_{\mathfrak C_1}\sqcup \varphi_{\mathfrak C_2}}{\longrightarrow} \mathbb{P}_{\mathfrak C_1}\sqcup \mathbb{P}_{\mathfrak C_2} \hookrightarrow \mathbb{P}_{\mathfrak C}=\spann(\mathbb{P}_{\mathfrak C_1},\mathbb{P}_{\mathfrak C_2}).
        \]
        \item[(iv)] 
        We have that $H^0(\mathfrak C, \omega_{\mathfrak C})$ can be canonically identified with $H^0(\mathfrak{C}^\dagger, \omega_{\mathfrak C^\dagger})$, hence we can identify $\mathbb P_{\mathfrak C}$ with $\mathbb P_{\mathfrak C^\dagger}$. Moreover, since  $U_{\mathfrak C^\dagger}=\mathfrak C^\dagger$, we have a well-defined morphism $\varphi_{\mathfrak{C}^\dagger}\colon \mathfrak{C^\dagger}\to \mathbb P_{\mathfrak C^\dagger}$ extending $\varphi_{\mathfrak C}\colon U_{\mathfrak C}\to \mathbb P_{\mathfrak C}$. 
    \end{enumerate}
\end{remark}

We now analyze when the canonical map fails to separate points.

\begin{proposition}
\label{prop:varphi_same_image}
    Let $\mathfrak{C} = (A, B, C)$ be a stable \name{} such that $U_{\mathfrak{C}}=\mathfrak{C}$. Given distinct points $p_1$ and $p_2$ of $\mathfrak{C}$, we have $\varphi_{\mathfrak C}(p_1)=\varphi_{\mathfrak C}(p_2)$ if and only if one of the following holds:
    \begin{enumerate}
        \item[(i)] $p_1$ and $p_2$ lie in the same connected component $C'$ of $C$  and $\{p_1,p_2\}$ is a separating pair of nodes of $C'$.
        \item[(ii)] 
        $p_1$ and $p_2$ are smooth points of $C$ and there is a subcurve $Z\subseteq \mathfrak C$ such that $p_1, p_2\in Z$, the restriction $\varphi_{C|_Z}$ is $2:1$, $\varphi_{C|_Z}(p_1) = \varphi_{C|_Z}(p_2)$, and $\varphi_{ C|_Z}(Z)$ is irreducible.  
    \end{enumerate}
    In particular, $\varphi_{\mathfrak C}\colon \mathfrak C\to\mathbb P_{\mathfrak C}$ is finite over its image.
\end{proposition}
\begin{proof}
    By Remark \ref{rem:1-scheme-property} (iii), 
    we can assume that $\mathfrak{C}$ is a connected stable curve. The results follow from \cite[Theorems~E and~F, and Propositions~3.10 and ~3.19]{catanese}.
\end{proof}

\begin{corollary}\label{cor:C_1-set}
    Let $\mathfrak{C} = (A, B, C)$ be a stable \name{} with no separating nodes. For $q\in\Proj_{\mathfrak{C}}$, exactly one of the following properties holds
    \begin{enumerate}
        \item[(i)] $\varphi_{\mathfrak{C}}^{-1}(q)=\emptyset$.
        \item[(ii)] $\varphi_{\mathfrak{C}}^{-1}(q) = \{p\}$, with $p\in C^{\sm} \cup A \cup B$.
        \item[(iii)] $\varphi_{\mathfrak{C}}^{-1}(q) = \{p_1, p_2\}$ and $p_1, p_2$ are as in Item~(ii) of Proposition \ref{prop:varphi_same_image}.
        \item[(iv)] $\varphi_{\mathfrak{C}}^{-1}(q)$ is a set of nodes of $C$ corresponding to a $\Cone$-set of $\Gamma_C$.
    \end{enumerate}
\end{corollary}

\begin{proof}
 The result is a consequence of   Remark~\ref{rem:propC1set} and Proposition \ref{prop:varphi_same_image}.
\end{proof}

\begin{definition} \label{def: hh}
    Let $C$ be a stable curve. The \emph{honest hyperelliptic} subcurves of $C$ are the  minimal subcurves $Z\subseteq C$ such that $\varphi_{C|_Z}$ is generically a $2:1$ morphism.
\end{definition}

\begin{proposition}\label{prop:catanese_hh}
Let $C$ be a stable curve with $U_C=C$, and let $Z\subseteq C$ be a
subcurve. Then $Z$ is honestly hyperelliptic if and only if it satisfies
\emph{(iii)}, \emph{(iv)}, and either \emph{(i)} or \emph{(ii)}:
\begin{enumerate}
    \item[(i)] $Z$ is irreducible;
    \item[(ii)] $Z=Z_1\cup Z_2$, with $Z_1\cong Z_2\cong\Proj^1$;
    \item[(iii)] there is a finite degree-$2$ morphism $\mu\colon Z\to\Proj^1$;
    \item[(iv)] $Z\cap Z'$ is a reduced fibre of $\mu$ for every connected
    component $Z'$ of $Z^c$ meeting $Z$.
\end{enumerate}
Moreover, $\varphi_C|_Z=\nu\circ\mu$ for an embedding
$\nu\colon\Proj^1\to\Proj_C$ onto a (usually degenerate) rational normal curve.
\end{proposition}

\begin{proof}
The necessity of the conditions follows from
\cite[Theorem~F]{catanese}. Conversely, assume that they hold and set
$L=\mu^*\mathcal O_{\Proj^1}(1)$. By {(iv)}, the divisor
$Z\cap Z'$ belongs to $|L|$ for every connected component $Z'$ of
$Z^c$ meeting $Z$. Hence \cite[Theorem~F]{catanese} gives the asserted
factorization of $\varphi_C|_Z$, which is therefore generically $2:1$.
Minimality is automatic in Case {(i)}; in Case {(ii)},
$\mu|_{Z_i}$ has degree~$1$ for $i=1,2$.
\end{proof}

We call an honest hyperelliptic subcurve as in Proposition \ref{prop:catanese_hh} (ii) an \emph{honest hyperelliptic pair of $C$}. Notice that, in this case, a $2:1$ map $Z=Z_1\cup Z_2\to \mathbb{P}^1$ corresponds to an isomorphism $\alpha\colon Z_1\to Z_2$, such that $\alpha(p) = p$ for every node $p\in Z_1\cap Z_2$ and $\alpha(p_1)=p_2$ for every $\{p_1\} = Z_1\cap Z'$ and $\{p_2\}= Z_2\cap Z'$ and every connected component $Z'$ of $Z^c$.

Motivated by Proposition \ref{prop:catanese_hh}, we now define honest hyperelliptic subcurves of \emph{any} \name{} (not necessarily connected or stable).

\begin{definition}\label{def:non-hyp} 
    Let $\mathfrak{C} = (A, B, C)$ be a \name{}. Let $Z\subseteq C$ be a subcurve such that $Z\cap U_{\mathfrak C}$ is dense in $Z$. We say that $Z$ is an \emph{honest hyperelliptic} subcurve of $\mathfrak{C}$ if either (1) and (3) or (2) and (3) below hold.
    \begin{enumerate}
        \item[(1)] $Z$ is irreducible.
        \item[(2)] $Z$ is the union of two smooth connected rational curves
$Z_1$ and $Z_2$; we call $Z$ an honest hyperelliptic pair.
        \item[(3)] There exists a $2:1$ morphism $\mu\colon Z\to \mathbb{P}^1$ such that
    \begin{itemize}
        \item[(a)] for $p_1,p_2\in Z\cap U_{\mathfrak C}$, we have $\varphi_C(p_1)=\varphi_C(p_2)$ if and only if $\mu(p_1)=\mu(p_2)$.
        \item[(b)] for each connected component $Z'$ of $Z^c$,  either $|Z\cap Z'|\leq 1$, or $Z\cap Z' = \{p_1, p_2\}$, with $\mu(p_1)=\mu(p_2)$ and, in Case (2), $p_1\in Z_1$ and $p_2\in Z_2$.
    \end{itemize}    
    \end{enumerate}  
        We denote by $\iota$ the involution associated to the morphism $\mu$ in (3) and we call it \emph{the involution} of $Z$.\end{definition}
When $\mathfrak C=C$ is a stable curve, this definition agrees with
Definition~\ref{def: hh}; this follows from
\cite[Theorem~F]{catanese}. When $U_C=C$, this is precisely
Proposition~\ref{prop:catanese_hh}.

\begin{corollary}\label{cor:honest-ramif}
    Let $\mathfrak{C}=(A,B,C)$ be a stable \name{} such that $U_{\mathfrak{C}}=\mathfrak{C}$. Let $p$ be a smooth point of $C$. If $p$ is a ramification point of $\varphi_{\mathfrak{C}}$, then $p$ belongs to an irreducible honest hyperelliptic subcurve of $\mathfrak{C}$.
\end{corollary}
\begin{proof}
    If $p$ is a smooth point of $\mathfrak C$ which is a ramification point of $\varphi_{\mathfrak{C}}$, then, by \cite[Proposition~3.10 and its proof]{catanese},
$\varphi_{\mathcal C}$ has degree at least two on the
irreducible component containing $p$, which means that there exist $p_1$ and $p_2$ in the condition of Proposition~\ref{prop:varphi_same_image} (ii). From Proposition~\ref{prop:catanese_hh}, either $Z$ is irreducible or the union of two smooth irreducible rational curves. However, in the latter case, there is no ramification over the rational curves.
\end{proof}

We end this section with the important definitions of canonical model and canonical triple of a \name{}.

\begin{definition}    \label{def:canonical_model}
   Let $\mathfrak{C} = (A, B, C)$ be a \name{} and $\varphi_{\mathfrak C}\colon U_{\mathfrak C}\to \mathbb P_{\mathfrak C}$ be its canonical map (see Definition \ref{def:canonical-map}). 
   
The \emph{canonical model} of $\mathfrak{C}$ is  defined as the closure $\widehat{\mathfrak C}:=\overline{\varphi_{\mathfrak{C}}(U_{\mathfrak{C}})}\subseteq \Proj_{\mathfrak C}$. 
   The \emph{canonical triple} of $\mathfrak{C}$ is defined as
   \[
   \Can(\mathfrak{C}) := (\widehat{\mathfrak{C}}, \mathcal{S}_{\widehat{\mathfrak{C}}}, \mathcal{B}_{\widehat{\mathfrak{C}}}),
   \]
   where
\begin{enumerate}
    \item[(1)]  $\mathcal{S}_{\widehat{\mathfrak{C}}} := \varphi_{\mathfrak{C}}(C^{\sing}\setminus C^{\sep}) \cup \varphi_{\mathfrak{C}}(B)\subseteq \widehat{\mathfrak C}$.
    \item[(2)]  $\mathcal{B}_{\widehat{\mathfrak{C}}} := \Br_{\varphi_{\mathfrak{C}^\dagger}}\setminus \mathcal{S}_{\widehat{\mathfrak{C}}}$ (see Section \ref{sec:schemes},  Notation \ref{not:dagger} and Remark \ref{rem:1-scheme-property}(iv)).
\end{enumerate}

We let $\widehat{\mathfrak{C}}^{\iso}$ be the set of isolated points of $\widehat{\mathfrak{C}}$. When  $\mathfrak{C}$ is stable and its normalization at all separating nodes has no connected component of genus $1$, we have 
\begin{equation}\label{eq:iso}
\widehat{\mathfrak{C}}^{\iso} = \varphi_{\mathfrak C}(A)\cup \varphi_{\mathfrak C}(B).
\end{equation}

Finally, an \emph{isomorphism} of  canonical triples $\Can(\mathfrak C)=(\widehat{\mathfrak{C}},\mathcal{S}_{\widehat{\mathfrak{C}}}, \mathcal{B}_{\widehat{\mathfrak{C}}})$ and $\Can(\mathfrak C')=(\widehat{\mathfrak{C}'},\mathcal{S}_{\widehat{\mathfrak{C}}'}, \mathcal{B}_{\widehat{\mathfrak{C}}'})$ is a linear isomorphism $\zeta\colon \mathbb P_{\mathfrak C}\to \mathbb {P}_{\mathfrak C'}$ such that 
\[
(\zeta(\widehat{\mathfrak{C}}), \zeta(\mathcal{S}_{\widehat{\mathfrak{C}}}), \zeta( \mathcal{B}_{\widehat{\mathfrak{C}}}))=(\widehat{\mathfrak{C}}',\mathcal{S}_{\widehat{\mathfrak{C}}'}, \mathcal{B}_{\widehat{\mathfrak{C}}'}).
\]
We denote by $\Isom_{\CT}(\Can(\mathfrak C),\Can(\mathfrak C'))$ the set of isomorphisms of canonical triples.
\end{definition}

Some remarks on our definitions are in order.

\begin{remark} \label{rem:branchminusS}
    We have defined $\mathcal{B}_{\widehat{\mathfrak{C}}} := \Br_{\varphi_{\mathfrak{C}^\dagger}}\setminus \mathcal{S}_{\widehat{\mathfrak{C}}}$ instead of $\mathcal{B}_{\widehat{\mathfrak{C}}} := \Br_{\varphi_{\mathfrak{C}^\dagger}}$ to remove the case of an internal node in an honest hyperelliptic component (which is the only case where $\Br_{\varphi_{\mathfrak{C}^\dagger}}$ could possibly intersect $\mathcal{S}_{\widehat{\mathfrak{C}}}$). 
\end{remark}

\begin{remark}
\label{rem:canonical_model_sep_stab_isomorphic}
    Let $\mathfrak{C}$ be a \name{} with no separating nodes. Then  $\Can(\mathfrak{C})$ and $ \Can(\mathfrak{C}^{\st})$ are canonically isomorphic (see Proposition~\ref{prop:model-stabilization}).
\end{remark}

We now  illustrate the  definition of canonical triple with an example of a \name{} consisting of a nodal curve only.

\begin{example}\label{exa:banana-example}
Let $C_1$ and $C_2$ be smooth non-hyperelliptic curves and consider distinct points $p_i,q_i\in C_i$ for $i=1,2$. 
Consider the \emph{banana curve} $C=C_1\cup C_2$ as in Figure~\ref{fig:banana}, obtained from $C_1$ and $C_2$ by identifying $p_1$ with $p_2$ and $q_1$ with $q_2$. The two nodes in $C_1\cap C_2$ are identified by the canonical map $\varphi_C$ (see Proposition \ref{prop:varphi_same_image}). Fix $i\in\{1,2\}$. 
The exact sequence in cohomology associated to the exact sequence
\[
0\to \omega_{C_{3-i}}\to \omega_C \to \omega_C|_{C_i}\cong\omega_{C_i}(p_i+q_i)\to 0
\]
shows that the restriction map $\rho_{C_i}\colon H^0(C,\omega_C)\to H^0(C_i,\omega_C|_{C_i})$ is surjective. Thus, $\Lambda_{\varphi_C(C_i)}=\mathbb P(\Ima(\rho_{C_i})^\vee)=\mathbb P(H^0(C_i,\omega_C|_{C_i})^\vee)$. Let $C'_i$ be the nodal curve obtained from $C_i$ by identifying $p_i$ with $q_i$.  Let $\nu_i\colon C_i\to C'_i$ be the normalization map. Notice that $H^0(C_i,\omega_C|_{C_i})=H^0(C_i,\nu_i^*(\omega_{C'_i}))$. Then $\widehat{C} = \widehat{C}'_1\cup \widehat{C}'_2$. Moreover, $\mathbb P_{C'_1}$ and $\mathbb P_{C'_2}$ intersect at exactly one point (equal to $\varphi_C(C_1\cap C_2)$) as subspaces of $\mathbb P_C$. This means that $\mathcal{S}_{\mathfrak{C}} = \mathbb P_{C'_1} \cap \mathbb P_{C'_2}$. Moreover $\mathcal{B}_{\mathfrak{C}} = \emptyset$.

Notice that the canonical triple of $C$ only depends on the nodal curves $C'_1$ and $C'_2$. Thus if $C'$ is the banana curve obtained from $C_1$ and $C_2$ by identifying $p_1$ with $q_2$ and $p_2$ with $q_1$, there is an isomorphism $\mathbb P_C\to \mathbb P_{C'}$ taking $\widehat{C}$ to $\widehat{C}'$.  
\begin{figure}[ht]
\centering
\begin{tikzpicture}[scale=1]
\draw[->] (-3,-1) to (-2, -1);
\draw [blue] plot [smooth]  coordinates {(-5,0) (-4,-1) (-5,-2)};
\draw [red] plot [smooth]  coordinates {(-4,0) (-5,-1) (-4,-2)};
\draw [red] (-1,-1) .. controls (-1,0) .. (0.3,-1);
\draw [red] (-1,-1) .. controls (-1,-2) .. (0.3,-1);
\draw[-] [red] (0.3,-1) to (0.8,-0.4);
\draw[-] [red] (0.3,-1) to (0.8,-1.6);
\draw  [blue] (1.6,-1) .. controls (1.6,0) .. (0.3,-1);
\draw  [blue] (1.6,-1) .. controls (1.6,-2) .. (0.3,-1);
\draw[-] [blue] (0.3,-1) to (-0.2,-0.4);
\draw[-] [blue] (0.3,-1) to (-0.2,-1.6); 
\node[below] at (-5.4,-0.5) [red] {$C_1$};
\node[below] at (-3.6,-0.5) [blue]{$C_2$};
\node[below] at (-4.5,0.8) {$C$};
\node[below] at (0.1,0.8) {$\widehat{C}$};
\end{tikzpicture}
\caption{The canonical map of a banana curve.}
\label{fig:banana}
\end{figure}
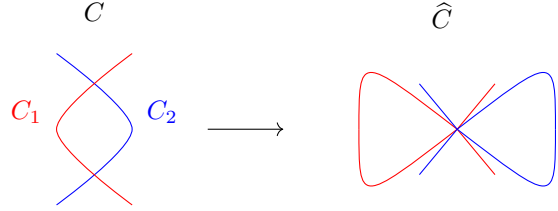
\end{example}

\subsection{Reconstruction of linear components of the canonical model}\label{sub:linear}

The main goal here is to prove that the complement of the union of all linear components of the canonical triple of a \name{} with no separating nodes, together with the data of the branch points and of the canonical image of the singular points, uniquely determines the canonical triple of the \name{} (Proposition~\ref{prop:reconstruct-linear}). To reconstruct the linear components, we first classify them in different types (Definition~\ref{def: type}). The main steps are the reconstruction of the linear components of types $1$ and $2$ (via Lemma~\ref{lem:reconstruct-type-one-two}) and then the reconstruction of the linear components of type $3$ (via Lemmas \ref{lem:linear-invariant} and \ref{lem:reconstruct-type-three}).  

Proposition~\ref{prop:reconstruct-linear} stated for nodal curves reads as follows:

\begin{proposition}
    \label{prop:curves-reconstruct-linear}
    Let $C$ and $C'$ be two nodal curves with no separating nodes. Assume that $\zeta\colon \Proj_{C}\to \Proj_{C'}$ is a linear isomorphism. The following are equivalent.
    \begin{itemize}
        \item[(i)] 
        $\zeta$ is an isomorphism of the canonical triples $\Can(C)$ and $\Can(C')$.
        \item[(ii)] If $X_C$ and $X_{C'}$ are the (closures of the) complement in $\widehat C$ and $\widehat{C}'$ of the union of the linear components, then $(\zeta(X_{\widehat{C}}),\zeta(\mathcal S_{\widehat{C}}), \zeta(\mathcal{B}_{\widehat{C}}))=(X_{\widehat{C}'},\mathcal S_{\widehat{C}'}, \mathcal{B}_{\widehat{C}'})$.
    \end{itemize} 
\end{proposition}

 Throughout this section we will consider \name{}s that have no separating nodes. For one such \name{} $\mathfrak{C} = (A, B, C)$, the open subset $U_{\mathfrak C}$ where the canonical map is a morphism is also closed, and it is obtained by removing all genus $0$ connected components from $\mathfrak{C}$. We start with some preliminary results.

Let $\mathfrak{C}$ be a \name{} with no separating nodes. Let $(\widehat{\mathfrak C}, \mathcal S_{\widehat{\mathfrak{C}}}, \mathcal{B}_{\widehat{\mathfrak{C}}})$ be the canonical triple of $\mathfrak{C}$ defined in \ref{def:canonical_model}. For a subscheme $W\subseteq \widehat{\mathfrak{C}}$ that is a union of irreducible components of $\widehat{\mathfrak{C}}$, we set 
\begin{equation}\label{eq:tau}
\tau_W:=|W\cap \mathcal S_{\widehat{\mathfrak C}}|.
\end{equation}

\begin{definition} \label{def: type}
     Let $W\in \Lin(\widehat{\mathfrak{C}})$. Given a non-negative integer $k$, we say that $W$ is \emph{of type $k$} if $\tau_W=k$. We let \[\Lin_{k}(\widehat{\mathfrak{C}}):= \{W \in \Lin(\widehat{\mathfrak{C}}): \tau_W=k\}\] be the set of all type-$k$ lines contained in $\widehat{\mathfrak{C}}$. 
\end{definition}

\begin{lemma}\label{lem:linear-classification}
Let $\mathfrak{C}=(A,B,C)$ be a stable \name{} with no separating nodes. Consider an irreducible component $W\subseteq \widehat{\mathfrak C}$ and let $Z=\varphi_{\mathfrak C}^{-1}(W)$ in the sense of
Convention~\ref{conv:curve-part}. 
We have that $W\in \Lin(\widehat{\mathfrak C})$ if and only if one of the following cases holds 
\begin{enumerate}
    \item[(i)] $g_Z=0$, $Z$ is irreducible, and $\delta_Z=3$. In this case, $W\in \Lin_{3}(\widehat{\mathfrak C})$ and $\mathcal B_{\widehat{\mathfrak C}}\cap W=\emptyset$.
    \item[(ii)] 
    $Z=Z_1\cup Z_2$ is an honest hyperelliptic pair of $C$ and $\delta_{Z_1}=\delta_{Z_2}=3$. In this case, $W\in \Lin_{3}(\widehat{\mathfrak C})$ and $\mathcal B_{\widehat{\mathfrak C}}\cap W=\emptyset$.
        \item[(iii)] 
    $g_Z = 0$, $Z$ is irreducible and the decomposition into connected components of $Z^c$ is given by $Z^c=F^Z_1\cup F^Z_2$, with $|Z\cap F^Z_1|=|Z\cap F^Z_2|=2$. In this case, $W\in \Lin_{2}(\widehat{\mathfrak C})$ and $|\mathcal B_{\widehat{\mathfrak C}}\cap W|=2$.
    \item[(iv)] $g_Z=1$, $Z$ is irreducible with $j$ nodes for $j\in \{0,1\}$, and $|Z\cap Z^c|=2$. In this case, $W\in \Lin_{1+j}(\widehat{\mathfrak C})$ and $|\mathcal B_{\widehat{\mathfrak C}}\cap W|=4-2j$.
    \item[(v)] $g_Z=2$, $Z$ is irreducible with $j$ nodes for  $j\in \{0,1,2\}$, and $Z$ is a connected component of $C$. In this case $W\in \Lin_{j}(\widehat{\mathfrak C})$  and $|\mathcal{B}_{\widehat{\mathfrak{C}}}\cap W|=6-2j$. 
 \end{enumerate}
In particular, $W$ is a linear component of type $0$ of $\widehat{\mathfrak{C}}$ if and only if $Z$ is a smooth connected component of $\mathfrak{C}$ of genus $2$. 
 \end{lemma}
 
The dual graphs of all possible curves $C$ and their component(s) $Z$ mapping to a linear component of the canonical curve are depicted in Figure \ref{fig:linear} below, with the gray vertices denoting the connected components of the curve $Z^c$.

 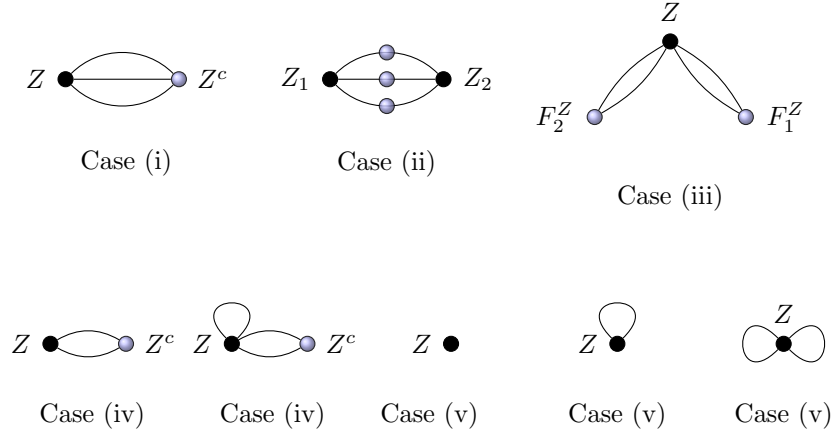
\begin{figure}[ht]
\centering
\begin{tikzpicture}[
    node distance=0.5cm,
    every node/.style={circle, draw, inner sep=2pt}
]
    
    \begin{scope}[local bounding box=row1, yshift=2cm]
        \begin{scope}[local bounding box=g1, xshift=-4.5cm]
            \node[label=left:$Z$, fill=black, text=white] (A1) at (0,0) {};
            \node[label=right:$Z^c$, ball color=blue!30!white, opacity=0.8] (B1) at (1.5,0) {}; 
            \draw (A1) to[bend left=45] (B1);
            \draw (A1) -- (B1);
            \draw (A1) to[bend right=45] (B1);
        \end{scope}
        \node[below, yshift=-0.4cm, rectangle, draw=none, fill=none] at (g1.south) {Case (i)};
        
        \begin{scope}[local bounding box=g2, xshift=-1cm]
            \node[label=left:$Z_1$, fill=black, text=white] (A2) at (0,0) {};
            \node[label=right:$Z_2$, fill=black, text=white] (B2) at (1.5,0) {}; 
            \draw (A2) to[bend left=45] (B2);
            \draw (A2) -- (B2);
            \draw (A2) to[bend right=45] (B2);
            
            \path (A2) to[bend left=45] node[midway, ball color=blue!30!white, opacity=0.8, inner sep=2pt] (C1) {} (B2);
            \path (A2) -- node[midway, ball color=blue!30!white, opacity=0.8, inner sep=2pt] (C2) {} (B2);
            \path (A2) to[bend right=45] node[midway, ball color=blue!30!white, opacity=0.8, inner sep=2pt] (C3) {} (B2);
            
            \foreach \vertex in {C1, C2, C3} {
                \foreach \angle in {30, 90, 150} {
                     (\vertex) -- ++(\angle:0.4);
                }
            }
        \end{scope}
        \node[below, yshift=-0.4cm, rectangle, draw=none, fill=none] at (g2.south) {Case (ii)};
        
        \begin{scope}[local bounding box=g3, xshift=3.5cm]
            \node[label=above:$Z$, fill=black, text=white] (A3) at (0,0.5) {};
            \node[label=right:$F^Z_1$, ball color=blue!30!white, opacity=0.8] (B3) at (1, -0.5) {};
            \node[label=left:$F^Z_2$, ball color=blue!30!white, opacity=0.8] (C3) at (-1, -0.5) {};
            \draw (A3) to[bend left=15] (B3);
            \draw (A3) to[bend right=15] (B3);
            \draw (A3) to[bend left=15] (C3);
            \draw (A3) to[bend right=15] (C3);
        \end{scope}
        \node[below, yshift=-0.4cm, rectangle, draw=none, fill=none] at (g3.south) {Case (iii)};
    \end{scope}
    
    \begin{scope}[local bounding box=row2, yshift=-1.5cm]
        \coordinate (lowest-point) at (0,-0.8);
        
        \begin{scope}[local bounding box=r1, xshift=-4.2cm]
            \node[label=left:$Z$, fill=black, text=white] (R1A) at (-0.5,0) {};
            \node[label=right:$Z^c$, ball color=blue!30!white, opacity=0.8] (R1B) at (0.5,0) {}; 
            \draw (R1A) to[bend left=30] (R1B);
            \draw (R1A) to[bend right=30] (R1B);
        \end{scope}
        \node[below, yshift=0.1cm, rectangle, draw=none, fill=none] at (r1.south |- lowest-point) {Case (iv)};
        
        \begin{scope}[local bounding box=r2, xshift=-1.8cm]
            \node[label=left:$Z$, fill=black, text=white] (R2A) at (-0.5,0) {};
            \node[label=right:$Z^c$, ball color=blue!30!white, opacity=0.8] (R2B) at (0.5,0) {}; 
            \draw (R2A) to[bend left=30] (R2B);
            \draw (R2A) to[bend right=30] (R2B);
            \draw (R2A) to[out=45, in=135, looseness=15] (R2A);
        \end{scope}
        \node[below, yshift=0.1cm, rectangle, draw=none, fill=none] at (r2.south |- lowest-point) {Case (iv)};
        
        \begin{scope}[local bounding box=r3, xshift=0.6cm]
            \node[label=left:$Z$, fill=black, text=white] (R3) at (0,0) {};
        \end{scope}
        \node[below, yshift=0.1cm, rectangle, draw=none, fill=none] at (r3.south |- lowest-point) {Case (v)};
        
        \begin{scope}[local bounding box=r4, xshift=2.8cm]
            \node[label=left:$Z$, fill=black, text=white] (R4) at (0,0) {};
            \draw (R4) to[out=45, in=135, looseness=15] (R4);
        \end{scope}
        \node[below, yshift=0.1cm, rectangle, draw=none, fill=none] at (r4.south |- lowest-point) {Case (v)};
        
        \begin{scope}[local bounding box=r5, xshift=5.0cm]
            \node[label=above:$Z$, fill=black, text=white] (R5) at (0,0) {};
            \draw (R5) to[out=135, in=225, looseness=15] (R5);
            \draw (R5) to[out=45, in=-45, looseness=15] (R5);
        \end{scope}
        \node[below, yshift=0.1cm, rectangle, draw=none, fill=none] at (r5.south |- lowest-point) {Case (v)};
    \end{scope}
\end{tikzpicture}
\caption{A scheme of a classification of all linear components of canonical models.}
\label{fig:linear}
\end{figure}

\begin{proof}
    For the first implication, assume $W\in \Lin(\mathfrak C)$. The restriction of $\varphi_{\mathfrak{C}}$ to $Z$ is generically $1:1$ or $2:1$.
    
    If it is $1:1$, that means that $Z$ is irreducible and the degree of $W$ is the degree of $\omega_{\mathfrak{C}}|_Z$ that, in our notation, is equal to $k_Z = 1$. Hence $k_Z=2g_Z-2+\delta_Z =1$, which can only happen if $g_Z=0$ and $\delta_Z=3$ (recall that $\mathfrak{C}$ is stable and  $\varphi_{\mathfrak{C}}|_Z$ is generically $1:1$). Also, the map $\varphi_{\mathfrak{C}}|_Z$ has no ramification. This yields Case (i).

    Now we consider the case where  the map $\varphi_{\mathfrak{C}}$ restricted to $Z$ is $2:1$. If $Z$ is irreducible, then $k_Z=2g_Z - 2 + \delta_Z = 2$. This can happen in the following cases
    \begin{enumerate}
        \item[(a)] $g_Z = 0$ and $\delta_Z=4$. Since $Z$ is honest hyperelliptic, we are in Case (iii).
        \item[(b)] $g_Z=1$ and $\delta_Z = 2$. We are in Case (iv).
        \item[(c)] $g_Z=2$ and  $\delta_Z=0$. We are in Case (v).
    \end{enumerate}

    Finally, we have to consider the case where $Z$ is reducible. This means that $Z$ is an honest hyperelliptic pair $Z=Z_1\cup Z_2$ and $k_{Z_1}=k_{Z_2}=1$. So, $g_{Z_1}=g_{Z_2}=0$ and $\delta_{Z_1}=\delta_{Z_2}=3$. Hence, we are in Case (ii). 

    The other implication of the lemma is straightforward. 
\end{proof}

Let $\mathfrak{C} = (A, B, C)$ be a stable \name{} with no separating nodes. For the next definition, recall Definitions~\ref{d: C1set} and \ref{def:canonical_model}. Moreover, observe that, given $q\in \mathcal{S}_{\widehat{\mathfrak{C}}}$, by Corollary \ref{cor:C_1-set}, $\varphi_{\mathfrak{C}}^{-1}(q)$ is either a singleton contained in $B$, or it is a set of nodes of $C$ whose corresponding set in $E(\Gamma_C)$ is a $\Cone$-set.

\begin{definition}\label{def:Cq}
   Let $\mathfrak{C} = (A, B, C)$ be a stable \name{} with no separating nodes. Let $q\in \mathcal{S}_{\widehat{\mathfrak{C}}}$. We define the \name{} $\mathfrak{C}_q$ as follows.
\begin{enumerate} 
    \item[(1)] If $\varphi_{\mathfrak{C}}^{-1}(q)\subseteq B$, then $\mathfrak{C}_q := (A, B\setminus \varphi_{\mathfrak{C}}^{-1}(q), C)$.
    \item[(2)] If $\varphi_{\mathfrak{C}}^{-1}(q)$ is a set of nodes of $C$, then $\mathfrak{C}_q := (A, B, C_q)$, where $\nu_q\colon C_q\to C$ is the partial normalization of $C$ at the nodes in $\varphi_{\mathfrak{C}}^{-1}(q)$.
\end{enumerate}
\end{definition}
In either case, let
\[
    \nu_q\colon\mathfrak C_q\to\mathfrak C
\]
be the natural morphism: the inclusion in Case {(1)}, and the
identity on $A\sqcup B$ together with the partial normalization
$C_q\to C$ in Case {(2)}.
There is a natural codimension $1$ injection $H^0(\mathfrak{C}_q, \omega_{\mathfrak{C}_q}) \to H^0(\mathfrak{C}, \omega_{\mathfrak{C}})$, whose image consists of the elements of $H^0(\mathfrak{C}, \omega_{\mathfrak{C}})$ that vanish on $\varphi_{\mathfrak{C}}^{-1}(q)$. In particular, the unique (up to scaling) non-zero linear functional $s\in H^0(\mathfrak{C}, \omega_{\mathfrak{C}})^\vee$ that vanishes on $H^0(\mathfrak{C}_q, \omega_{\mathfrak{C}_q})$ is precisely the evaluation at any of the points $p\in \varphi_{\mathfrak{C}}^{-1}(q)$. 

By the discussion above, given $q\in \mathcal S_{\widehat{\mathfrak C}}\subseteq \mathbb P_{\mathfrak C}$, the projection from $q$ can be written as $\pi_q\colon \Proj_{\mathfrak C}\dashrightarrow \mathbb{P}_{\mathfrak C_q}$ and we have a natural commutative diagram (recall the definition of $U_{\mathfrak{C}}$ from Definition~\ref{def:canonical-map} and see also \cite[Remark 3.8]{catanese})
\begin{equation}\label{eq:nat-diag}
\begin{tikzcd}
    U_{\mathfrak{C}_q} \ar[d, "\nu_q"] \ar[r,  "{\varphi_{\mathfrak{C}_q}}"] & \widehat{\mathfrak{C}}_q\\
    U_{\mathfrak{C}} \ar[r, "{\varphi_{\mathfrak{C}}}"] & \widehat{\mathfrak{C}}.\ar[u,dashed,  "\pi_q"]
\end{tikzcd}
\end{equation}

\begin{lemma}\label{lem:reconstruct-type-one-two}
Let $\mathfrak{C}$ be a \name{}  
with no separating nodes. For a projective line $L\subseteq \Proj_{\mathfrak C}$, the following conditions are equivalent
\begin{itemize}
\item[(i)]
$L\subseteq \widehat{\mathfrak C}$ is a linear component of type $1$ or $2$.
\item[(ii)]
$L=\spann(q,q_1,q_2)$ for some $q\in\mathcal S_{\widehat{\mathfrak{C}}}$ and distinct points $q_1,q_2\in \mathcal B_{\widehat{\mathfrak{C}}}$.
\end{itemize}
\end{lemma}

\begin{proof}
The statement depends only on the canonical triple $\Can(\mathfrak C)=(\widehat{\mathfrak{C}}, \mathcal{S}_{\widehat{\mathfrak{C}}}, \mathcal{B}_{\widehat{\mathfrak{C}}})$ of $\mathfrak C$. So, by Remark \ref{rem:canonical_model_sep_stab_isomorphic}, we can assume that $\mathfrak{C}$ is stable. 

Assume that $L$ is a linear component of type $1$ or $2$ of $\widehat{\mathfrak C}$. By Lemma \ref{lem:linear-classification}, we are in one of the cases (iii), (iv), or (v). 
In all cases, $|\mathcal B_{\widehat{\mathfrak C}}\cap L|\geq 2$, hence (ii) holds.

Now assume that $L=\spann(q, q_1, q_2)$ for some $q\in\mathcal S_{\widehat{\mathfrak{C}}}$ and distinct $q_1,q_2\in \mathcal B_{\widehat{\mathfrak{C}}}$. Consider the diagram in Equation \eqref{eq:nat-diag}. Notice that $U_{\mathfrak{C}} = \mathfrak{C}$ (because $\mathfrak{C}$ is stable with no separating nodes), hence the diagram becomes
\[
\begin{tikzcd}
    U_{\mathfrak{C}_q} \ar[d, "\nu_q"] \ar[r, "{\varphi_{\mathfrak C_q}}"] & {\widehat{\mathfrak{C}}_q}\\
    \mathfrak{C} \ar[r, "{\varphi_{\mathfrak C}}"] & \widehat{\mathfrak{C}}\ar[u, dashed, "\pi_q"]
\end{tikzcd}
\]

 Let $p_1, p_2$ be the points of $C^{\sm}\subseteq \mathfrak{C}$ such that  $\varphi_{\mathfrak{C}}(p_i)=q_i$. Since $q,q_1,q_2$ are collinear, we have $\pi_{q}(q_1)=\pi_{q}(q_2)$. Thus we must have that $\varphi_{\mathfrak{C}_q}(p_1) = \varphi_{\mathfrak{C}_q}(p_2)$ by the above diagram (we are thinking of $C^{\rm {sm}}$ as an open subset of $U_{\mathfrak{C}_q}$).
 
 Let $Z_i$ be the (one-dimensional) irreducible component of $\mathfrak C$ containing $p_i$ for $i=1,2$. Thus, by Corollary \ref{cor:honest-ramif}, each $Z_i$ is an honest hyperelliptic subcurve of $C$.
Let $Z_i'$ be the (one-dimensional) irreducible component of $\mathfrak{C}_q$ such that $\nu_q(Z'_i)=Z_i$. Fix $i\in\{1,2\}$. Then one of the following holds. Either $Z_i'$ is an honest hyperelliptic subcurve of $\mathfrak{C}_q$ having $p_i$ as a ramification point, or $Z_i'$ is contracted by $\varphi_{\mathfrak C_q}$. Since $\mathfrak{C}$ is stable, the latter case can only occur if
\begin{enumerate}
    \item[(a)] $g_{Z_i} = 0$, $|Z_i\cap Z_i^c| = 4$, $|Z_i\cap Z_i^c\cap \varphi_{\mathfrak{C}}^{-1}(q)| = 2$;  or
    \item[(b)]  $g_{Z_i}=1$, $|Z_i\cap Z_i^c|=2$, $Z_i\cap Z_i^c\subseteq \varphi_{\mathfrak{C}}^{-1}(q)$.
    \item[(c)] $g_{Z_i} = 1$, $|Z_i\cap Z_i^c|=2$, and $Z_i$ has an internal node contained in $\varphi_{\mathfrak{C}}^{-1}(q)$.
    \item[(d)] $g_{Z_i} = 2$, $Z_i$ is a connected irreducible component of $\mathfrak{C}$  with at least one node contained in $\varphi_{\mathfrak{C}}^{-1}(q)$.
\end{enumerate}

If both $Z_1'$ and $Z_2'$ are not contracted by $\varphi_{\mathfrak C_q}$, then they are honest hyperelliptic curves of $\mathfrak{C}_q$. 

In this case, $p_1$ and $p_2$ remain ramification points for the map $\varphi_{\mathfrak{C}_q}$, which implies that $\varphi_{\mathfrak C_q}(p_1)\neq \varphi_{\mathfrak C_q}(p_2)$, a contradiction. 

If $Z_1'$ is contracted and $Z_2'$ is not contracted  by  $\varphi_{\mathfrak C_q}$, then $\varphi_{\mathfrak C_q}(p_1)\in \widehat{\mathfrak C}^{\sing}\cup \widehat{\mathfrak C}^{\iso}$ (recall Equation~\eqref{eq:iso}) while $\varphi_{\mathfrak C_q}(p_2)\in \widehat{\mathfrak C}^{\sm}\setminus \widehat{\mathfrak C}^{\iso}$, a contradiction.

If both $Z'_1$ and $Z'_2$ are contracted by $\varphi_{\mathfrak C_q}$ then, since $\mathfrak{C}$ is stable, we must have that $Z'_1=Z'_2$, otherwise they would be contracted to distinct points contradicting that $\varphi_{\mathfrak{C}_q}(p_1) = \varphi_{\mathfrak{C}_q}(p_2)$. By Lemma \ref{lem:linear-classification} we have that, in all cases above where $Z'_1$ is contracted, the line $L$ is a linear component of $\widehat{\mathfrak{C}}$ of type 1 or type 2 (they correspond to Items (iii), (iv), (v) of Lemma \ref{lem:linear-classification}), hence we are done.
\end{proof}

Next, we discuss how to reconstruct the linear components of type 3. Let $\mathfrak{C}$ be a \name{} with no separating nodes and $\widehat{\mathfrak{C}}$ be its canonical model. Consider the subscheme  
$Y_{\widehat{\mathfrak{C}}}\subseteq\widehat{\mathfrak{C}}$ defined as the closure
\[
Y_{\widehat{\mathfrak{C}}}:=\overline{\widehat{\mathfrak{C}}\setminus\bigcup_{W\in \Lin_{3}(\widehat{\mathfrak{C}})} W}.
\]
Notice that $Y_{\mathfrak{C}}$ contains the set of isolated points $\widehat{\mathfrak C}^{\iso}$. 

We first detect isolated points of $\widehat{\mathfrak C}$, then show how projection from $q \in \mathcal{S}_{\widehat{C}}$ recovers the projected pair $(Y_{\widehat{\mathfrak{C}}_q}, \mathcal S_{\widehat{\mathfrak{C}}_q})$, and finally we define an invariant $\lambda$ (based on an intrinsic count $\zeta$) that allows to detect the linear components of type $3$. 

\begin{lemma}
\label{lem:recover_isolated}
    Let $\mathfrak{C}$ be a \name{} with no separating nodes. Let $\widehat{{\mathfrak{C}}}$ be the canonical model of $\mathfrak{C}$.  A point $q\in \widehat{\mathfrak{C}}$ belongs to $\widehat{\mathfrak{C}}^{\iso}$ if and only if
    \[
    q\notin  \spann((Y_{\widehat{\mathfrak{C}}}\cup \mathcal S_{\widehat{\mathfrak{C}}})\setminus \{q\}).
    \]
\end{lemma}

\begin{proof}
    By Remark \ref{rem:canonical_model_sep_stab_isomorphic}, we can assume that $\mathfrak{C}$ is stable. If $q\in \widehat{\mathfrak{C}}^{\iso}$, then $q\in \varphi_{\mathfrak{C}}(A\cup B)$. By Remark \ref{rem:proj_C_span}, we have that $q\notin \spann(\widehat{\mathfrak{C}}\setminus \{q\})$ and we are done.

    Let us prove the other inclusion. Assume that $q \in \widehat{\mathfrak{C}}\setminus \widehat{\mathfrak{C}}^{\iso}$. If $q\in Y_{\widehat{\mathfrak{C}}}$, then it is clear that $q\in \spann(Y_{\widehat{\mathfrak C}}\setminus \{q\})$ (recall that $q$ is not an isolated point). On the other hand, if $q$ lies in a type $3$ linear component $W$ of $\mathfrak C$, then $q\in \spann(S_{\widehat{\mathfrak{C}}}\setminus \{q\})$, because $\tau_W=3$ (recall Equation \eqref{eq:tau}), and we are done.
    \end{proof}

\begin{lemma}
\label{lem:recover_YCq}
    Let $\mathfrak{C}$ be a stable \name{} with no separating nodes. Let $\widehat{\mathfrak C}$ be the canonical model of $\mathfrak{C}$. The following properties hold.
    \begin{itemize}
        \item[(i)] The set of linear subspaces $\{\Lambda_W\}$, where $W$ runs through all connected components of $\widehat{\mathfrak C}$, can be recovered from the pair $(Y_{\widehat{\mathfrak{C}}}, \mathcal S_{\widehat{\mathfrak{C}}})$. 
        \item[(ii)] For a point $q\in \mathcal S_{\widehat{\mathfrak{C}}}$, the pair $(Y_{\widehat{\mathfrak{C}}_q}, \mathcal S_{\widehat{\mathfrak{C}}_q})$ can be recovered from the pair $(Y_{\widehat{\mathfrak{C}}}, \mathcal S_{\widehat{\mathfrak{C}}})$. Also, $\mathcal S_{\widehat{\mathfrak{C}}_q}=\pi_q(\mathcal S_{\widehat{\mathfrak{C}}}\setminus\{q\})$ where $\pi_q\colon \mathbb P_{\mathfrak C}\dashrightarrow \mathbb P_{\mathfrak C_q}$ is the projection from $q$.
    \end{itemize}
\end{lemma}

\begin{proof}
  We prove (i). By Remark~\ref{rem:1-scheme-property}(iii), the connected components of $\mathfrak{C}$ are in bijection with the connected components of $\widehat{\mathfrak{C}}$. Notice that, by Remark~\ref{rem:proj_C_span} (cf.\ Remark~\ref{rem:1-scheme-property}(iii)), we have
$H^0(\mathfrak{C},\omega_{\mathfrak{C}})=\bigoplus_i H^0(Z_i,\omega_{Z_i})$  where $Z_i$ corresponds to $W_i$ in the decomposition $\widehat{\mathfrak{C}}=\bigcup_i W_i$ into connected components. Dually,
$H^0(\mathfrak{C},\omega_{\mathfrak{C}})^\vee=\bigoplus_i H^0(Z_i,\omega_{Z_i})^\vee$, so the subspaces
$\mathbb{P}\bigl(H^0(Z_i,\omega_{Z_i})^\vee\bigr)\subseteq \mathbb{P}\bigl(H^0(\mathfrak{C},\omega_{\mathfrak{C}})^\vee\bigr)$
are pairwise disjoint, and in particular the spans $\Lambda_{W_i}$ are pairwise disjoint.

  
  Let then $F$ be the scheme given by the union of $Y_{\widehat{\mathfrak{C}}}$ together with all lines through any triple of distinct collinear points in $\mathcal S_{\widehat{\mathfrak{C}}}$. Since the $\Lambda_{W_i}$'s are projectivizations of independent direct summands, any line through three distinct points of $S_{\widehat{C}}$ lies in a unique $\Lambda_{W_i}$,
hence adding such lines cannot connect different components. We also have an inclusion $\widehat{\mathfrak{C}}\subseteq F$ because $Y_{\widehat{\mathfrak{C}}}$ is obtained from $\widehat{\mathfrak{C}}$ by removing precisely the type-3 linear components, and every such component is a line $L$ with $\tau_L=|L\cap \mathcal{S}_{\widehat{\mathfrak{C}}}|=3$, hence it is spanned by three collinear points of $\mathcal{S}_{\widehat{\mathfrak{C}}}$ and is therefore contained in the union of lines added in the definition of $F$. 

Thus the inclusion $\widehat{\mathfrak C}\subseteq F$ induces a bijection, say $T \mapsto T_F$, from the connected components of $\widehat{\mathfrak C}$ to those of  $F$. The statement follows from the equality $\Lambda_T=\Lambda_{T_F}$ of the corresponding spans.
 
We prove (ii). Recall Diagram \eqref{eq:nat-diag}. Set $\mathfrak{C} = (A, B, C)$ and $\mathfrak{C}_q = (A_q, B_q, C_q)$ (recall Definition~\ref{def:Cq}). Note that $A=A_q$ and either $B=B_q$ or $C=C_q$. 

We claim that $\mathcal S_{\widehat{\mathfrak{C}}_q}=\pi_q(\mathcal S_{\widehat{\mathfrak{C}}}\setminus\{q\})$. Indeed, notice that there is a bijection between $(\mathfrak{C}^{\sing}\cup B)\setminus\varphi_{\mathfrak{C}}^{-1}(q)$ and $\mathfrak{C}_q^{\sing}\cup B_q$. If $C_q=C$, meaning that $\varphi_{\mathfrak{C}}^{-1}(q)\subseteq B$, then the claim is clear. On the other hand, if $B=B_q$, then $\varphi_{\mathfrak{C}}^{-1}(q)$ is a set of nodes of $C$ corresponding to a $\Cone$-set (see Corollary \ref{cor:C_1-set}). Since $C$ has no separating nodes, this means that  $C_q$ also has no separating nodes. The claim  then follows from the definition of $\mathcal{S}_{\widehat{\mathfrak{C}}_q}$ in Definition~\ref{def:canonical_model}.

Next, we prove that $(Y_{\widehat{\mathfrak{C}}}, \mathcal S_{\widehat{\mathfrak{C}}})$ recovers $Y_{\widehat{\mathfrak{C}}_q}$.
Assume that we are given the point $q$ and the pair $(Y_{\widehat{\mathfrak{C}}}, \mathcal S_{\widehat{\mathfrak{C}}})$. If $\varphi_{\mathfrak{C}}^{-1}(q)\subseteq B$, then the result is clear because $Y_{\widehat{\mathfrak{C}}_q} = \overline{\pi_q(Y_{\widehat{\mathfrak{C}}})}$. So we assume that $\varphi_{\mathfrak{C}}^{-1}(q)\subseteq C$.

  Let $W\subseteq \widehat{\mathfrak{C}}$ be a type-$3$ linear component. If $q\not\in W$, we have that $\pi_q(W)$ is a type-$3$ linear component of $\widehat{\mathfrak{C}}_q$. If $q\in W$, then $\pi_q(W)$ is a point. So every one-dimensional component of $Y_{\widehat{\mathfrak{C}}_q}$ is contained in $\overline{\pi_q(Y_{\widehat{\mathfrak{C}}})}$. Thus
\[
Y_{\widehat{\mathfrak{C}}_q}\setminus \widehat{\mathfrak{C}}_q^{\iso} \subseteq \overline{\pi_q(Y_{\widehat{\mathfrak{C}}})}.
\]

More than that, given a point $q'\in \widehat{\mathfrak{C}}_q^{\iso}\setminus \mathcal S_{\widehat{\mathfrak{C}}_q}$, we have the following two cases:
\begin{enumerate}
    \item[(a)] $q'=\varphi_{\mathfrak C_q}(Z)$, where $Z$ is a smooth genus~$1$ connected component $Z\subseteq \mathfrak{C}_q$. Since $\mathfrak{C}_q$ is the normalization of $\mathfrak{C}$ at the nodes in the set $\varphi_{\mathfrak{C}}^{-1}(q)$, the curve $Z$ is attached at two nodes to the rest of $\mathfrak{C}$, hence $W = \varphi_{\mathfrak{C}}(Z)$ is a  linear component of $\widehat{\mathfrak C}$ of type $1$, so $W\subseteq Y_{\mathfrak{C}}$, and $q'=\pi_q(W)$;
    \item[(b)] $q'=\varphi_{\mathfrak C_q}(p')$, where $p'\in A_q$. Since $A_q=A$, we have that 
    \[
    q'\in \pi_q(\varphi_{\mathfrak{C}}(A))\subseteq \overline{\pi_q(Y_{\widehat{\mathfrak C}})}.
    \]
\end{enumerate}
Hence we deduce that 
\begin{equation}\label{eq:inclusion}
Y_{\widehat{\mathfrak{C}}_q}\setminus (\widehat{\mathfrak{C}}_q^{\iso}\cap \mathcal S_{\widehat{\mathfrak{C}}_q}) \subseteq \overline{\pi_q(Y_{\widehat{\mathfrak{C}}})}.
\end{equation}

By the claim, we have already recovered $\mathcal S_{\widehat{\mathfrak{C}}_q}$, hence for every linear component $W$ in $\overline{\pi_q(Y_{\mathfrak{\widehat C}})}$ we can compute the value $\tau_W=|W \cap \mathcal{S}_{\widehat{\mathfrak{C}}_q}|$. So, by the definition of type-3 components as those with $\tau_W=3$, this allows us 
to reconstruct all type-$3$ linear components of $\widehat{\mathfrak{C}}_q$ that are in $\overline{\pi_q(Y_{\mathfrak{\widehat C}})}$. After removing these, we recover $Y_{\widehat{\mathfrak{C}}_q}$, possibly with the exception of some points in $\widehat{\mathfrak{C}}_q^{\iso}
\cap \mathcal S_{\widehat {\mathfrak{C}}_q}$. So, we are left to recover the isolated points of the set $\mathcal S_{\widehat {\mathfrak{C}}_q}$. However, using Equation~\eqref{eq:inclusion}, we have
\[
Y_{\widehat{\mathfrak{C}}_q}\cup \mathcal S_{\widehat{\mathfrak{C}}_q} = \overline{(\overline{\pi_{q}(Y_\mathfrak{C})}\setminus \bigcup_{W \in \Lin_{3}(\widehat{\mathfrak{C}}_q)} W)} \cup \mathcal S_{\widehat{\mathfrak{C}}_q}.
\]
The right-hand side is known, so we recover
$Y_{\widehat{\mathfrak C}_q}\cup
\mathcal S_{\widehat{\mathfrak C}_q}$. By
Lemma~\ref{lem:recover_isolated}, we can then recover
$\widehat{\mathfrak C}_q^{\iso}\cap
\mathcal S_{\widehat{\mathfrak C}_q}$, and hence
$Y_{\widehat{\mathfrak C}_q}$. This completes the reconstruction.
\end{proof}

\begin{definition} \label{def:zeta}
    Let $\mathfrak{C}=(A,B,C)$ be a \name{} with no separating nodes (hence $\widehat{\mathfrak{C}}=\widehat{\mathfrak{C}}^{\st}$ by Remark \ref{rem:canonical_model_sep_stab_isomorphic}), and write $\mathfrak{C}^{\st}=(A',B',C')$. For $q\in \mathcal{S}_{\widehat{\mathfrak{C}}}=\varphi_{\mathfrak{C}}(C^{\sing}) \cup \varphi_{\mathfrak{C}}(B)$, we set 
    \[
    \zeta_{\widehat{\mathfrak{C}}}(q) = |\varphi_{\mathfrak{C}^{\st}}^{-1}(q)\cap C'|.
    \]

\end{definition}

 \begin{remark} \label{rem:computezeta}  The  number $\zeta_{\widehat{\mathfrak{C}}}(q)$ depends only on $\widehat{\mathfrak{C}}$ and not on $\mathfrak{C}$ or $\mathfrak{C}^{\st}$. 
Also, 
 \begin{enumerate}
     \item if $\varphi_{\mathfrak{C}^{\st}}^{-1}(q) \subseteq B'$, we have $\varphi_{\mathfrak{C}^{\st}}^{-1}(q) \cap C'= \emptyset$ and therefore $\zeta_{\widehat{\mathfrak{C}}}(q)=0$;
     \item if $\varphi_{\mathfrak{C}^{\st}}^{-1}(q) \subseteq C'$, we have that $\varphi_{\mathfrak{C}^{\st}}^{-1}(q)$ corresponds to a $\Cone$-set of $\Gamma_{C'}$. Thus, if $Z \subseteq \mathfrak{C}^{\st}$ is the connected component containing $\varphi_{\mathfrak{C}^{\st}}^{-1}(q)$, by Corollary~\ref{cor:C_1-set} we have that $\zeta_{\widehat{\mathfrak{C}}}(q)$ equals the number of connected components of 
     $Z_q$, the normalization of $Z$ at $\varphi_{\mathfrak{C}^{\st}}^{-1}(q)$. 
 \end{enumerate}
 \end{remark}
\begin{lemma}
\label{lem:sq_depends_Y}
    For a point $q\in \mathcal{S}_{\widehat{\mathfrak{C}}}$, the number $\zeta_{\widehat{\mathfrak{C}}}(q)$  depends only on $(Y_{\widehat{\mathfrak{C}}}, \mathcal S_{\widehat{\mathfrak{C}}})$. 
\end{lemma}

\begin{proof}
    We begin with some general considerations. Assume that we are given $(Y_{\widehat{\mathfrak{C}}}, \mathcal S_{\widehat{\mathfrak{C}}})$. We can assume that $\mathfrak{C} = (A, B, C)$ is stable (taking the stabilization does not change $\zeta_{\widehat{\mathfrak{C}}}(q)$ and $(Y_{\widehat{\mathfrak{C}}}, \mathcal{S}_{\widehat{\mathfrak{C}}})$). 
     By Remark~\ref{rem:1-scheme-property}, the connected components of $\mathfrak{C}$ are in bijection with the connected components of $\widehat{\mathfrak{C}}$. Fix $q \in \mathcal{S}_{\widehat{\mathfrak{C}}}$ and let $W\subseteq \widehat{\mathfrak{C}}$ be the connected component containing $q$ and  $Z=\varphi_{\mathfrak{C}}^{-1}(W)$ be the  connected component of $\mathfrak{C}$ containing $\varphi_{\mathfrak{C}}^{-1}(q)$. 
    By Lemma~\ref{lem:recover_YCq}~(i), we can recover the linear space $\Lambda_W$ from $(Y_{\widehat{\mathfrak{C}}}, \mathcal{S}_{\widehat{\mathfrak{C}}})$. Moreover, by Corollary~\ref{cor:C_1-set}, since $q\in \mathcal{S}_{\widehat{\mathfrak{C}}}$, either $Z\subseteq B$ or $Z\subseteq C$. By Lemma~\ref{lem:recover_isolated}, the pair $(Y_{\widehat{\mathfrak C}},S_{\widehat{\mathfrak C}})$ determines
whether $q$ is isolated in $\widehat{\mathfrak C}$, which is equivalent to $Z\subseteq B$. In this case, $\zeta_{\widehat{\mathfrak{C}}}(q) = 0$ by Remark \ref{rem:computezeta}.
    
    We may therefore assume $Z\subseteq C$. Then, by Remark~\ref{rem:computezeta},   $\zeta_{\widehat{\mathfrak{C}}}(q)$ is the number of connected components of $Z_q$, which by Remark \ref{rem:1-scheme-property} and Diagram \eqref{eq:nat-diag}, also equals the number of connected components of $\overline{\pi_q(W)}$, where $\pi_q\colon \mathbb P_{\mathfrak C}\dashrightarrow \mathbb P_{\mathfrak C_q}$ is the projection from $q$.
    We find that $\zeta_{\widehat{\mathfrak{C}}}(q)$ is the number of connected components of $\widehat{\mathfrak{C}}_q = \pi_q(\widehat{\mathfrak{C}})$ inside $\pi_q(\Lambda_W)$.
    
    By Lemma~\ref{lem:recover_YCq}(ii), we can recover
$(Y_{\widehat{\mathfrak C}_q},\mathcal S_{\widehat{\mathfrak C}_q})$
from $(Y_{\widehat{\mathfrak C}},\mathcal S_{\widehat{\mathfrak C}})$.
Applying Lemma~\ref{lem:recover_YCq}(i) to $\mathfrak C_q^{\st}$ and
denoting by $K$ the set of connected components of
$\widehat{\mathfrak C}_q$, we can then recover the spaces
$\Lambda_{W'}$, for $W'\in K$.
 
 We now prove that $\zeta_{\widehat{C}}(q)$ depends only on $(Y_{\widehat{\mathfrak{C}}}, \mathcal{S}_{\widehat{\mathfrak{C}}})$. Starting from $q$ and $(Y_{\widehat{\mathfrak{C}}}, \mathcal{S}_{\widehat{\mathfrak{C}}})$ we first identify the set of linear subspaces $\{\Lambda_W\}$ where $W$ runs through all connected components of $\widehat{\mathfrak{C}}$ (again by Lemma \ref{lem:recover_YCq} (i)). From this, we identify the space $\Lambda_W$ that contains $q$. We already recovered the set $\{\Lambda_{W'}\}$ for $W'$ that runs through the set $\mathcal K$ of all connected components of $\widehat{\mathfrak{C}}_q$. Hence
    \[
\zeta_{\widehat{\mathfrak{C}}}(q) = |\{\Lambda_{W'};\  W'\in\mathcal K \text{ and } \Lambda_{W'}\subseteq \pi_q(\Lambda_W)\}|.
    \]
    This concludes the proof.
\end{proof}

The following lemma and definition are needed in the statement and in the proof of Lemma~\ref{lem:linear-invariant}.

\begin{lemma}\label{lem:4-collinear}
Let $\mathfrak{C}$ be a stable \name{} with no separating nodes. There are no $4$ pairwise distinct singular points of $\mathfrak C$ such that their images via the canonical map are pairwise distinct and collinear.
\end{lemma}

\begin{proof}
    Write $\mathfrak{C}=(A,B,C)$. By Remark \ref{rem:1-scheme-property} (iii), we can assume that $\mathfrak{C}=C$, with $C$   connected. Let $g$ be the genus of $C$. Consider a set $N=\{p_1,p_2,p_3,p_4\}$ of $4$ pairwise distinct nodes of $C$. Assume that their images via the canonical map $\varphi_C$ are pairwise distinct and, by contradiction, assume that these images are collinear. 
    
    Given a subset $M\subseteq N$, we let $\nu_M\colon C_M\to C$ be the partial normalization of $C$ at the nodes in $M$. Since the sections in $H^0(C, \omega_C)$ that vanish at all the points in $M$ are precisely the sections in $H^0(C, \nu_{M*}(\omega_{C_M}))\subseteq H^0(C, \omega_C)$, in order for the images of the $4$ points to be collinear, we must have  $h^0(C_{N}, \omega_{C_{N}})=g-2$. 

    The points $\varphi_C(p_i)$ are pairwise distinct, hence $h^0(C_{\{p_i,p_j\}}, \omega_{C_{\{p_i,p_j\}}})=g-2$, for every  $1\leq i<j\leq 4$ (that means that $\{p_i,p_j\}$ is not a separating pair of nodes). In particular, for every permutation  $(i_1,i_2,i_3,i_4)$ of $(1, 2, 3,4)$, we must have that $p_{i_3}$ and $p_{i_4}$ are  separating nodes of $C_{\{p_{i_1}, p_{i_2}\}}$.
    

    Next, notice that $C_N$ has exactly $3$ connected components. Indeed, $C_{\{p_1,p_2\}}$ has exactly one connected component, while $p_3$ and $p_4$ are separating nodes of $C_{\{p_1,p_2\}}$ so, after normalizing each one, we get one extra connected component.
    
    Since the inverse image via $\nu_N$ of the set $N$ is a set of eight points in $C_N$, one of the connected components of $C_N$ contains $\alpha$ elements of the set of eight points, with $1\leq \alpha\leq 2$. 
    
    If $\alpha=1$, let $p_i$ be the image of this point via $\nu_N$. Then $p_i$ is a separating node, a contradiction. 
    If $\alpha=2$, let $p_i$ and $p_j$ be the images of these $\alpha$ points via $\nu_N$, for $1\le i\le j\le 4$.  If $i\ne j$, then $\{p_i,p_j\}$ is a separating pair of nodes, a contradiction. If $i=j$, the two points of $\nu_N^{-1}(N)$ on this connected component
of $C_N$ are the two branches over $p_i$. Since the component contains
no other point of $\nu_N^{-1}(N)$, it remains disconnected from the
others after regluing $N$, contradicting the connectedness of $C$.
\end{proof}

 \begin{definition} \label{def:lambda}
         Let $\mathfrak{C}$ be a stable \name{} with no separating nodes. Consider $3$ collinear and pairwise distinct points $q_1, q_2, q_3$ in $\mathcal S_{\widehat{\mathfrak{C}}}\setminus \widehat{\mathfrak C}^{\iso}$. Set $q:=q_1$ and $q' = \pi_q(q_2) = \pi_q(q_3)$, where $\pi_q\colon \mathbb P_{\mathfrak C}\dashrightarrow \mathbb P_{\mathfrak C_q}$ is the projection from $q$. By Lemma \ref{lem:recover_YCq} (ii), we have that $q'\in \mathcal S_{\widehat{\mathfrak{C}}_{q}}$. We define  (recall Definition~\ref{def:zeta})
  \[
  \lambda(q_1; q_2, q_3) := \zeta_{\widehat{\mathfrak{C}}}(q_2) + \zeta_{\widehat{\mathfrak{C}}}(q_3) - \zeta_{\widehat{\mathfrak{C}}_q}(q').
  \]
    \end{definition}
  
   The presence of linear components of type~$3$ is detected by this invariant $\lambda$, as we prove in the following result.
   
  \begin{lemma}\label{lem:linear-invariant}
    Let $\mathfrak{C}$ be a stable \name{} with no separating nodes. 
    Let $q_1, q_2, q_3$ be pairwise distinct collinear points in $\mathcal S_{\widehat{\mathfrak{C}}}\setminus \widehat{\mathfrak C}^{\iso}$. Let $L$ be the line through $q_1,q_2,q_3$. Then $\lambda(q_1; q_2, q_3)\in \{0,1,2\}$ and 
    \[
    \lambda(q_1; q_2, q_3)=\begin{cases}          0, & \text{ if } L\not\subseteq \widehat{\mathfrak C}.  \\
        1,  & \text{ if } L\subseteq \widehat{\mathfrak C} \text{ and } \varphi_{\mathfrak C}^{-1}(L) \text{ is irreducible.}  \\
        2,  & \text{ if } L\subseteq \widehat{\mathfrak C}  \text{ and } \varphi_{\mathfrak C}^{-1}(L) \text{ is not irreducible.} 
    \end{cases}
\]
  \end{lemma}

  \begin{proof}

Write $\mathfrak C=(A,B,C)$ and, for $i=1,2,3$, set
\[
N_i:=\varphi_{\mathfrak C}^{-1}(q_i)\subseteq C^{\sing}.
\]
By Corollary~\ref{cor:C_1-set}, the edges corresponding to each $N_i$
form a $\Cone$-set of $\Gamma_C$. Set $q:=q_1$ and
$q':=\pi_q(q_2)=\pi_q(q_3)$, and recall Diagram~\eqref{eq:nat-diag}.

Since $q,q_2,q_3$ are collinear, $\nu_q^{-1}(N_2)$ and
$\nu_q^{-1}(N_3)$ lie in the same one-dimensional connected component
$Z$ of $\mathfrak C_q$ (otherwise, $\pi_q(q_2)$ and $\pi_q(q_3)$ would
lie in disjoint linear subspaces of $\mathbb P_{\mathfrak C_q}$).
By Lemma~\ref{lem:4-collinear} and Diagram~\eqref{eq:nat-diag}, 
\[
N_{q'}:=\varphi_{\mathfrak C_q}^{-1}(q')\cap C_q^{\sing}
       =\nu_q^{-1}(N_2\sqcup N_3),
\]
and
\begin{equation}\label{eq:phiCqst}
\varphi_{(\mathfrak C_q)^{\st}}^{-1}(q')=\st(N_{q'}).
\end{equation}
where $\st\colon Z\to Z^{\st}$ denotes the stabilization morphism. (Here $N_{q'}$ denotes only the nodal part of the fibre, which may also
contain one-dimensional components). Let
$\nu_{q'}\colon\mathfrak C_{q,q'}\to\mathfrak C_q$ be the partial
normalization at $N_{q'}$, and set $Z_{q'}:=\nu_{q'}^{-1}(Z)$, where the inverse image is understood
as in Convention~\ref{conv:curve-part}, see
Figure~\ref{fig:lambda}.

Since $\mathfrak C_q$ is obtained from the stable \name{} $\mathfrak C$
by normalizing $N_1$, the $\Cone$-set property of the $N_j$ shows that
there are exactly two connected components $Z_1,Z_2$ of $Z_{q'}$ such
that
\[
\nu_q(\nu_{q'}(Z_i))\cap N_j\neq\emptyset
\qquad (i=1,2,\ j=1,2,3);
\]
see Figure~\ref{fig:lambda}. Viewed as subcurves of $C$, each $Z_i$
meets its complement in exactly three nodes, one in each $N_j$.
Consequently, if $g_{Z_i}=0$, stability forces $Z_i$ to be a smooth
rational curve. Every other connected component of $Z_{q'}$ has genus
at least $1$. After relabelling, assume that
$g_{Z_1}\leq g_{Z_2}$. We distinguish four cases, using
\[
\zeta_{\widehat{\mathfrak C_q}}(q')
=\zeta_{\widehat{Z^{\st}}}(q'),
\]
Remark~\ref{rem:computezeta}, and Equation~\eqref{eq:phiCqst}. For the computations in Cases~(a), (b), (c), we may assume that every
positive-genus connected component of $Z_{q'}$, with its two points over
$N_{q'}$ marked, is stable, since replacing it by its stable model does
not change $Z^{\st}$ or the nodes over $q'$.

\begin{itemize}
\item[(a)]
Assume that $g_{Z_1}\geq1$, and hence $g_{Z_2}\geq1$. Then $Z$ is
stable, so
\[
\zeta_{\widehat{Z^{\st}}}(q')
=\zeta_{\widehat{\mathfrak C}}(q_2)
+\zeta_{\widehat{\mathfrak C}}(q_3),
\qquad
\lambda(q_1;q_2,q_3)=0.
\]
Moreover, $L\not\subseteq\widehat{\mathfrak C}$. Otherwise $L$ would
be a linear component of type~$3$. By
Lemma~\ref{lem:linear-classification}, every irreducible component $T$
of $C$ mapping onto $L$ is rational and has exactly one boundary node
over each $q_j$. Its strict transform in $Z_{q'}$ is therefore a
connected component whose image meets every $N_j$, hence is $Z_1$ or
$Z_2$, contradicting $g_{Z_1},g_{Z_2}\geq1$.

\item[(b)]
If $g_{Z_1}=0$ and $g_{Z_2}\geq1$, then $Z^{\st}$ is obtained by
contracting the smooth rational curve $Z_1$. Hence
\[
\zeta_{\widehat{Z^{\st}}}(q')
=\zeta_{\widehat{\mathfrak C}}(q_2)
+\zeta_{\widehat{\mathfrak C}}(q_3)-1,
\qquad
\lambda(q_1;q_2,q_3)=1.
\]
By Lemma~\ref{lem:linear-classification},
$L=\varphi_{\mathfrak C}(Z_1)$, and the argument in Case~(a), together
with $g_{Z_2}\geq1$, gives
$\varphi_{\mathfrak C}^{-1}(L)=Z_1$.

\item[(c)]
If $g_{Z_1}=g_{Z_2}=0$ and $g_Z\geq2$, then $Z^{\st}$ is obtained by
contracting the smooth rational curves $Z_1$ and $Z_2$. Hence
\[
\zeta_{\widehat{Z^{\st}}}(q')
=\zeta_{\widehat{\mathfrak C}}(q_2)
+\zeta_{\widehat{\mathfrak C}}(q_3)-2,
\qquad
\lambda(q_1;q_2,q_3)=2.
\]
Moreover,
\[
L=\varphi_{\mathfrak C}(Z_1)=\varphi_{\mathfrak C}(Z_2).
\]
By the argument in Case~(a), these are the only irreducible components
mapping onto $L$, so
$\varphi_{\mathfrak C}^{-1}(L)=Z_1\cup Z_2$
(an honest hyperelliptic pair).

\item[(d)]
Finally, if $g_{Z_1}=g_{Z_2}=0$ and $g_Z=1$, then $Z_1,Z_2$ are the
irreducible components of $Z$ and $|Z_1\cap Z_2|=2$. Writing
$(\mathfrak C_q)^{\st}=(A_q',B_q',C_q')$, we have
$Z^{\st}\subseteq B_q'$. Hence
\[
\zeta_{\widehat{Z^{\st}}}(q')=0,
\qquad
\zeta_{\widehat{\mathfrak C}}(q_2)
=\zeta_{\widehat{\mathfrak C}}(q_3)=1,
\]
and so $\lambda(q_1;q_2,q_3)=2$. We obtain the same conclusion
as in Case~(c).
\end{itemize}
The proof is complete.
\end{proof}

  \begin{figure}[ht]
\centering

\begin{tikzpicture}[
    every node/.style={font=\small},
    v/.style={circle, fill=black, inner sep=1.5pt}
]

\node[v,label=left:$Z_1$] (z1) at (0,0) {};
\node[v,label=right:$Z_2$] (z2) at (8,0) {};

\node[v] (t1) at (1.5,1.4) {};
\node[v] (t2) at (3.0,1.6) {};
\node[v] (t3) at (5.0,1.6) {};
\node[v] (t4) at (6.5,1.4) {};

\node[v] (m1) at (2.5,0) {};
\node[v] (m2) at (5.5,0) {};

\node[v] (b1) at (1.5,-1.4) {};
\node[v] (b2) at (3.0,-1.6) {};
\node[v] (b3) at (5.0,-1.6) {};
\node[v] (b4) at (6.5,-1.4) {};

\draw (z1) to[bend left=18] node[above left] {$e_1$} (t1);
\draw (t1) -- node[above] {$e_2$} (t2);
\draw (t2) -- node[above] {$e_3$} (t3);
\draw (t3) -- node[above] {$\ldots$} (t4);
\draw (t4) to[bend left=18] node[above right] {$e_{k_1}$} (z2);

\draw (z1) -- node[above] {$f_1$} (m1);
\draw (m1) -- node[above] {$\ldots$} (m2);
\draw (m2) -- node[above] {$f_{k_2}$} (z2);

\draw (z1) to[bend right=18] node[below left] {$g_1$} (b1);
\draw (b1) -- node[below] {$g_2$} (b2);
\draw (b2) -- node[below] {$g_3$} (b3);
\draw (b3) -- node[below] {$\ldots$} (b4);
\draw (b4) to[bend right=18] node[below right] {$g_{k_3}$} (z2);
\end{tikzpicture}

\caption{The picture shows a quotient of the dual graph of the relevant connected component of the curve $C$ (vertices correspond to connected but not necessarily irreducible components), highlighting the two components $Z_1$ and $Z_2$ and the $\Cone$-sets of $q_1$, $q_2$, $q_3$ consisting of edges labelled $e$, $f$, $g$ respectively. The dual graph of $Z$ consists of the $f$ and the $g$ edges and their endpoints. All vertices other than $Z_1$ and $Z_2$ have genus $\geq 1$. We have $k_i=\zeta_{\mathfrak{\widehat{C}}}(q_i)$ for $i=2,3$. In Case~(a) $Z_1$ and $Z_2$ also have positive genus hence $\zeta_{\widehat{Z^{\st}}}(q')=k_2+k_3$. In Case~(b) $Z_1$ has genus $0$, thus $Z$ is unstable, and $Z_2$ has positive genus, hence $\zeta_{\widehat{Z^{\st}}}(q')=k_2+k_3-1$. In Cases~(c) and (d) $Z_1$ and $Z_2$ have genus $0$, hence $\zeta_{\widehat{Z^{st}}}(q')=k_2+k_3-2$. Case~(d) is the special case when $k_2=k_3=1$ and $\zeta_{\widehat{Z^{\st}}}(q') = 0$.}
\label{fig:lambda}
\end{figure}
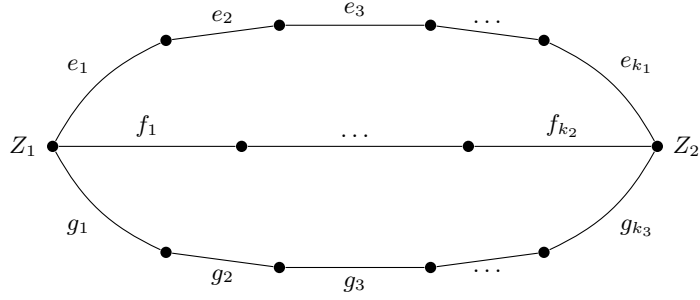

  \begin{remark}
In the proof of Lemma \ref{lem:linear-invariant}, the argument for Case (d) proceeds similarly to that of Case (c), the key distinction being the stabilization of a curve of genus $1$. This technical requirement motivated our definition of \name{}s, a class of objects that always admits stable models.
  \end{remark}

\begin{remark} For every triple $q_1,q_2,q_3$ as in Definition~\ref{def:lambda}, we have the equality $\lambda(q_1; q_2,q_3)=\lambda(q_1; q_3, q_2)$. From Lemma~\ref{lem:linear-invariant} we also have $\lambda(q_1; q_2,q_3)=\lambda(q_2; q_1,q_3)$. Thus $\lambda$ does not depend on the order of its $3$ arguments.
\end{remark}

With all this, we are now able to give our reconstruction result of lines of type~$3$.

\begin{lemma}\label{lem:reconstruct-type-three}
    Let $\mathfrak C$ and $\mathfrak{C}'$ be \name{}s without separating nodes. Assume that $\zeta\colon \Proj_{\mathfrak{C}}\to \Proj_{\mathfrak{C}'}$ is a linear isomorphism. The following are equivalent
    \begin{itemize}
        \item[(i)] $(\zeta(\widehat{\mathfrak{C}}),\zeta(\mathcal S_{\widehat{\mathfrak{C}}}))=(\widehat{\mathfrak{C}}', \mathcal S_{\widehat{\mathfrak{C}}'})$.
        \item[(ii)] $(\zeta(Y_{\widehat{\mathfrak{C}}}),\zeta(\mathcal S_{\widehat{\mathfrak{C}}}))=(Y_{\widehat{\mathfrak{C}}'},\mathcal S_{\widehat{\mathfrak{C}}'})$.
    \end{itemize}  
\end{lemma}

    \begin{proof}
The implication (i) $\Rightarrow$ (ii) is immediate.  

Conversely, assume that
$(Y_{\widehat{\mathfrak C}},\mathcal S_{\widehat{\mathfrak C}})$ is given. For convenience, we will identify $\Proj_{\mathfrak{C}}$ and $\Proj_{\mathfrak{C}'}$, and assume that the isomorphism $\zeta$ is the identity.
By Remark~\ref{rem:canonical_model_sep_stab_isomorphic}, we may assume that
$\mathfrak C$ and $\mathfrak{C}'$ are stable. It remains to reconstruct the linear components of type $3$. By
Lemma~\ref{lem:recover_isolated}, the isolated points of $\widehat{\mathfrak{C}}$ are determined. Let
$q_1,q_2,q_3\in
\mathcal S_{\widehat{\mathfrak C}}
\setminus\widehat{\mathfrak C}^{\mathrm{iso}}$
be pairwise distinct and collinear, and let $L$ be the line through them. By
Lemma~\ref{lem:linear-invariant},
\[
L\subseteq\widehat{\mathfrak C}
\quad\Longleftrightarrow\quad
\lambda(q_1;q_2,q_3)>0.
\]
By Lemmas~\ref{lem:recover_YCq}(ii) and~\ref{lem:sq_depends_Y}, the value $\lambda(q_1;q_2,q_3)$ is
determined by
$(Y_{\widehat{\mathfrak C}},\mathcal S_{\widehat{\mathfrak C}})$.
Hence all type-$3$ components, and therefore $\widehat{\mathfrak C}$, are
determined by
$(Y_{\widehat{\mathfrak C}},\mathcal S_{\widehat{\mathfrak C}})$. 
\end{proof}

    We are now able to complete the reconstruction of all linear components. To this end, let $\mathfrak{C}$ be a \name{} with no separating nodes and $\widehat{\mathfrak{C}}$ be its canonical model. Consider the subscheme $X_{\widehat{\mathfrak{C}}}\subseteq\widehat{\mathfrak{C}}$ defined as the closure
    \begin{equation}\label{eq:X-def}
X_{\widehat{\mathfrak{C}}}=\overline{\widehat{\mathfrak{C}}\setminus\bigcup_{W\in \Lin(\widehat{\mathfrak{C}})} W}.
    \end{equation}

\begin{proposition}\label{prop:reconstruct-linear}
    Let $\mathfrak{C}$ and $\mathfrak{C}'$ be \name{}s without separating nodes. Assume that $\zeta\colon \Proj_{\mathfrak{C}}\to \Proj_{\mathfrak{C}'}$ is a linear isomorphism. The following are equivalent
    \begin{itemize}
        \item[(i)] 
        the canonical triples $\Can(\mathfrak C)$ and $\Can(\mathfrak C')$  are isomorphic  via $\zeta$.
        \item[(ii)] $(\zeta(X_{\widehat{\mathfrak{C}}}),\zeta(\mathcal S_{\widehat{\mathfrak{C}}}), \zeta(\mathcal{B}_{\widehat{\mathfrak{C}}}))=(X_{\widehat{\mathfrak{C}}'},\mathcal S_{\mathfrak{\widehat C}'}, \mathcal{B}_{\mathfrak{\widehat C}'})$.
    \end{itemize} 
\end{proposition}

\begin{proof}
The fact that (i) implies (ii) is straightforward. Let us prove that (ii) implies (i).  For convenience, we identify $\Proj_{\mathfrak{C}}$ and
$\Proj_{\mathfrak{C}'}$, and assume that $\zeta$ is the identity. By Remark~\ref{rem:canonical_model_sep_stab_isomorphic}, we may replace
$\mathfrak{C}$ and $\mathfrak{C}'$ by their stabilizations, and hence assume
that both are stable. Assume that
\(
(X_{\widehat{\mathfrak{C}}},\mathcal S_{\widehat{\mathfrak{C}}},
\mathcal{B}_{\widehat{\mathfrak{C}}})
=
(X_{\widehat{\mathfrak{C}}'},\mathcal S_{\mathfrak{\widehat C}'},
\mathcal{B}_{\mathfrak{\widehat C}'}).
\)

By Lemma~\ref{lem:reconstruct-type-one-two}, the linear components
of types $1$ and $2$ are determined by this common data; let $\mathcal L_{1,2}:= \Lin_1(\widehat{\mathfrak{C}}) \sqcup \Lin_2(\widehat{\mathfrak{C}})$. Set
\[
\mathcal B_0:=
\mathcal B_{\widehat{\mathfrak C}}
\setminus
\left(
X_{\widehat{\mathfrak C}}
\cup\bigcup_{L\in\mathcal L_{1,2}}L
\right).
\]
By Lemma~\ref{lem:linear-classification}, $\mathcal B_0$ consists of
the six branch points on each component of type $0$, and no others.
Since distinct such components have disjoint spans, they are precisely
the lines containing at least three points of $\mathcal B_0$. Hence the
common data determine all linear components of types $0$, $1$, and $2$.

Adding these components to $X_{\widehat{\mathfrak C}}$ recovers
$Y_{\widehat{\mathfrak C}}$. Hence
$Y_{\widehat{\mathfrak C}}=Y_{\widehat{\mathfrak C}'}$. Therefore $(Y_{\widehat{\mathfrak{C}}},\mathcal S_{\widehat{\mathfrak{C}}})
=
(Y_{\widehat{\mathfrak{C}}'},\mathcal S_{\mathfrak{\widehat C}'})$,
and Lemma~\ref{lem:reconstruct-type-three} gives
$\widehat{\mathfrak C}=\widehat{\mathfrak C}'$. The result follows.
\end{proof}

\section{General position and quasi-Eulerian orientations}\label{sec:qauxiliary}

\subsection{Curves in projective space}

This section establishes some results on the geometry of curves embedded in projective space. First, we generalize to the case of reducible curves the classical general-position theorem, which guarantees that the points in a general hyperplane section of an irreducible curve are in linearly general position (see \cite[Section~1, Chapter III]{acgh}). Our Theorem \ref{thm:general_position} proves this for curves in $\mathbb P^N$, provided the degree of the divisor restricted to each subcurve is appropriately bounded. Second, we prove a series of auxiliary results, collected in Lemma \ref{lem:dimdeg}, which provide some essential numerical control for the canonical model of a stable nodal curve. These will be instrumental in the remainder of this paper.

\begin{lemma}\label{lem:pq}
    Let $C$ be a connected curve in $\mathbb P^N$. Let $Z_1,Z_2\in \Irr(C)$ (possibly $Z_1=Z_2$) such that $Z_1\cup Z_2$ is not contained in a $2$-dimensional linear subspace of $\mathbb P^N$. If $(p_1,p_2)$ is a general pair in $Z_1\times Z_2$, then $\spann(p_1,p_2)\cap C=\{p_1,p_2\}$.  
\end{lemma}

   \begin{proof}
Suppose the conclusion fails. Then for a general pair $(p_1,p_2)\in Z_1\times Z_2$, the line $L_{p_1p_2}$ through $p_1,p_2$ meets $C$ in a subscheme of length at least $3$ supported at smooth points of $C$. Thus \cite[Theorem~1]{zivran} applies with $X=C$ and $\ell=1$, so
\[
\dim \Big(\bigcup_{(p_1,p_2)\in Z_1\times Z_2}L_{p_1p_2}\Big)\le 2.
\]
But the union on the left is the join of $Z_1$ and $Z_2$, which has dimension at least $3$ because $Z_1\cup Z_2$ is not contained in a $2$-dimensional linear subspace of $\mathbb P^N$ (see \cite[Corollary~1.5]{joins}). This contradiction proves the lemma.
\end{proof}

We  state the general-position theorem as follows.

\begin{theorem}
\label{thm:general_position}
    Let $C$ be a connected curve in $\mathbb{P}^N$ and $H$ be a general hyperplane in $\mathbb P^N$. Let $D$ be an effective divisor on $C$ such that $D\leq H\cdot C$ and $\deg(D|_Z)\leq N_Z$ (recall Equation~\eqref{eq:Lambda}) for every subcurve $Z\subseteq C$. Then $\dim (\spann(D)) = \deg(D) - 1$.
\end{theorem}

In order to prove Theorem \ref{thm:general_position}, we will use an equivalent formulation. Let $C$ be a curve in $\mathbb{P}^N$. 
Consider the incidence variety
\[
\mathcal V_C:=\{(p,H): p\in C\cap H\} \subseteq C\times (\mathbb P^N)^\vee.
\]

\begin{proposition} \label{prop: equivalent}
Let $C$ be a connected curve in $\mathbb{P}^N$. The following are equivalent.
\begin{itemize}
    \item[(i)] Let $(p, H)$ be a general pair in the incidence variety $\mathcal V_C$. Let $D$ be an effective divisor on $C$ such that $p\leq D\leq H\cdot C$ and $\deg(D|_Z)\leq N_Z$, for every subcurve $Z\subseteq C$. Then $\dim (\spann(D)) = \deg(D)-1$.
    \item[(ii)] Let $H$ be a general hyperplane in $\mathbb P^N$. Let $D$ be an effective divisor on $C$ such that $D\leq H\cdot C$ and $\deg(D|_Z)\leq N_Z$ for every subcurve $Z\subseteq C$. Then $\dim (\spann(D)) = \deg(D) - 1$.
\end{itemize}
\end{proposition}
\begin{proof}
The argument is identical to \cite[Step (ii), p.110]{acgh}. 
\end{proof}

\begin{proof}[Proof of Theorem \ref{thm:general_position}]
 The statement is the same as Item~(ii) in Proposition~\ref{prop: equivalent}. First, we can assume that $N_C=N$, otherwise just change the ambient space to $\spann(C)$.  We will proceed by induction on $N$. If $N=1$, the result is trivial. If $N=2$, then $\deg(D) = \deg(D|_C)\leq N_C \leq 2$, and $H$ intersects $C$ transversally, so $D$ is either empty, or one point, or two distinct points. In all cases the result follows.   

Now let $N\ge 3$. It is enough to prove Item~(i) in Proposition~\ref{prop: equivalent}. Consider a general pair $(p, H)$ in the incidence variety $\mathcal V_C$. Let $Z_0\in \Irr(C)$ be such that $p\in  Z_0$. Let $\pi_p\colon \mathbb{P}^N\dashrightarrow \mathbb{P}^{N-1}$ be the projection with center $p$.  
Given a subcurve $Z\subseteq C$, by the generality of $p$ we have that $p\in \Lambda_{Z}$ if and only if $Z_0\subseteq\Lambda_{Z}$. 

Let $Z\subseteq C$ be a subcurve. If $Z_0$ is a line and $Z_0\subseteq Z$,  we set $\overline{Z}$ to be the closure of $\pi_p(Z\setminus Z_0)$. In the remaining cases, we let  $\overline{Z}$ be the closure of the image of $Z$ via $\pi_p$. Then  $\overline{Z}$ is a curve, except when $Z=Z_0$ and $Z_0$ is a line in which case $\overline{Z}$ is empty. Notice that it may happen that $\overline{Z_1}=\overline{Z_2}$ for $Z_1 \neq Z_2$ (e.g., if $Z_1$ and $Z_2$ lie in a plane that also contains the point $p$). 
For $\overline{Z}\ne\emptyset$, we have 
\begin{equation}\label{eq:relation}
N_{\overline{Z}} = \begin{cases}
    N_{Z} & \text{ if }\Lambda_{Z_0}\not\subseteq \Lambda_{Z};\\
    N_{Z} - 1& \text{ if }\Lambda_{Z_0}\subseteq \Lambda_{Z}.
\end{cases}
\end{equation}

Now consider a divisor $D$ on $C$ such that $p\leq D\leq H\cdot C$ and $\deg(D|_Z)\leq N_Z$ for every subcurve $Z\subseteq C$. We define the divisor $\overline{D}$ on the curve $\overline{C}$ as 
\[
\overline{D}:=\underset{q\ne p}{\underset{q\le D}{\sum}} \pi_p(q).
\]
We have that $D$ is reduced, because $(p, H)$ is general. Hence $\deg(\overline{D})=\deg(D)-1$.

We claim that $\overline{D}$ is reduced. Indeed, let $q_1, q_2$ be two distinct points in $C$ such that $q_1+q_2+p\leq D$. Let $Z_1,Z_2\in \Irr(C)$ such that $q_1\in Z_1$ and $q_2\in Z_2$. Since $N\geq 3$, for each pair $(q_1, q_2)\in Z_1\times Z_2$, there is a hyperplane through $q_1, q_2, p$. In particular, since $H$ is a general hyperplane containing $p$, we have that $(q_1,q_2)$ is a general pair of $Z_1\times Z_2$. Hence we have two cases. If $N_{Z_1\cup Z_2} \geq 3$, Lemma~\ref{lem:pq} says that $p, q_1, q_2$ are not collinear, hence $\pi_p(q_1) \ne \pi_p(q_2)$. So $\overline{D}$ is reduced.

On the other hand, if $N_{Z_1\cup Z_2}\leq 2$ and if, by contradiction, $p, q_1, q_2$ are collinear, then $p\in \Lambda_{Z_1\cup Z_2}$. Moreover, since $p$ is general in $Z_0$, we have $Z_0\subseteq \Lambda_{Z_1\cup Z_2}$, so $N_{Z_0\cup Z_1\cup Z_2}\leq 2$. Now, $\deg(D|_{Z_0\cup Z_1\cup Z_2})\geq \deg(p+q_1+q_2) = 3$ while, by the assumption, we have $\deg(D|_{Z_0\cup Z_1\cup Z_2})\leq N_{Z_0\cup Z_1\cup Z_2}\le 2$, which is a contradiction. So $\overline{D}$ is reduced. This ends the proof of the claim.

 We also have that $\overline{D}$ is supported on $\overline{C}^{\sm}$. Indeed, by generality, $H$  can be chosen to avoid the finite number of points in $\pi_p^{-1}(\overline{C}^{\sing})$ or $\pi_p^{-1}(\overline{C}^{\sing})\setminus Z_0$ (if $Z_0$ is a line). In this last case, notice that $D|_{Z_0} = p$, but $\pi_p(Z_0)$ is a point not contained in the support of $\overline{D}$ (by the latter's definition). 

We will now prove that, for every subcurve $W\subseteq \overline{C}$, we have
\begin{equation}\label{eq:Dbar}
\deg(\overline{D}|_{W})\leq N_{W}.
\end{equation}
To prove Equation \eqref{eq:Dbar}, we will use Equation \eqref{eq:relation}. Let $Z\subseteq C$ be the maximal subcurve of $C$ such that $\overline{Z}=W$. Since $\overline{D}$ is supported on $\overline{C}^{\sm}$ we have the following comparisons for $\deg(\overline{D}|_{\overline{Z}})$ and $\deg(D|_Z)$.

If $Z_0\subseteq Z$, then 
\[
\deg(\overline{D}|_{\overline{Z}})=\deg(D|_Z)-1\text{ and } N_{\overline{Z}}=N_Z-1,
\]
hence Equation \eqref{eq:Dbar} is true. If $Z_0\not\subseteq Z$ and $\Lambda_{Z_0}\not\subseteq \Lambda_Z$, then 
\[
\deg(\overline{D}|_{\overline{Z}})=\deg(D|_Z)\text{ and } N_{\overline{Z}}=N_Z,
\]
which, again, implies Equation \eqref{eq:Dbar}.
Finally, if $Z_0\not\subseteq Z$ and $\Lambda_{Z_0}\subseteq \Lambda_Z$, then 
   \[
 \deg(\overline{D}|_{\overline{Z}})  = \deg(D|_{Z\cup Z_0} ) - \deg(D|_{Z_0}) 
     \leq N_{Z\cup Z_0} - 1 
     = N_Z-1 \\
     = N_{\overline{Z}},
   \] 
    concluding the proof that Equation \eqref{eq:Dbar} always holds.

Since $(p,H)$ is general, the hyperplane $\pi_p(H)$ is general in $\mathbb{P}^{N-1}$ and, since $\overline{D}$ is reduced, we have $\overline{D}\leq \pi_p(H)\cdot \overline{C}$. By Equation \eqref{eq:Dbar}, we can use the induction hypothesis on $\overline C$ (note that $\overline C$ is a connected curve), with the formulation of Proposition~\ref{prop: equivalent}~(ii), and deduce $\dim (\spann(\overline{D})) = \deg(\overline{D}) - 1= \deg(D)-2$, so $\dim(\spann(D))=\dim(\spann(\overline{D}))+1=\deg(D) - 1$. 
\end{proof}

In the remainder of this section, we let $C$ be a connected and stable nodal curve without separating nodes of genus $g\ge2$. We let  $\varphi\colon C\to \mathbb P_C$ be the canonical map of $C$. 
 Recall that, for a subcurve $Z\subseteq C$, we let $F^Z_1,\dots F^Z_{t_Z}$ be the connected components of $Z^c$. Recall the definition of $N_W$ for a subscheme $W$ of the canonical model $\widehat{C}$ of $C$ in Equation \eqref{eq:Lambda} and the definition of $k_Z$ in Equation \eqref{eq:kZ}.

\begin{lemma}\label{lem:dimdeg}
    Let $C$ be a connected and stable nodal curve without separating nodes of genus $g\ge 2$.  The following properties hold.
    \begin{itemize}
        \item[(i)] For a subcurve $Z\subseteq C$, we have  $N_{\varphi(Z)}=g_Z+\delta_Z-t_Z-1$. 
        \item[(ii)] 
        For  $Z\in \Irr(C)$, we have
        \[
        \deg(\varphi(Z))=
        \begin{cases}
        k_Z/2, & \text{ if $Z$ is honest hyperelliptic};\\
        k_Z, & \text{otherwise}.
        \end{cases}
        \] 
        Furthermore, if $Z$ is an honest hyperelliptic pair of $C$ (hence, $Z$ has two irreducible components), then $\deg(\varphi(Z))=\frac{k_Z}{2}$.
    \item[(iii)] Let $\ud\in \Sigma(C)$. For every subcurve $Z$ of $C$, we have $d_Z\leq N_{\varphi(Z)}$. 
    \item[(iv)] 
    Let $\ud\in \Sigma(C)$. For every honest hyperelliptic subcurve $Z$ of $C$, we have
    \[
    1\le d_Z=\frac{k_Z}{2}=g_Z-1+\frac{\delta_Z}{2}=N_{\varphi(Z)}< k_Z.
    \]
    \end{itemize}  
\end{lemma}

\begin{proof} 
    Consider the natural exact sequence
    \[
    0\to \omega_C|_{Z^c}(-Z\cap Z^c)\to \omega_C\to\omega_C|_{Z}\to 0.
    \]
    Consider the associated exact sequence in cohomology
    \[
    0\to H^0(Z^c, \omega_C|_{Z^c}(-Z\cap Z^c))\to H^0(C, \omega_C)\overset{\rho_Z}{\to}H^0(Z, \omega_C|_Z).
    \]
    We have $N_{\varphi(Z)}=\dim \Lambda_{\varphi(Z)}$ and $\Lambda_{\varphi(Z)}=\Proj(\Ima(\rho_Z)^\vee)$. Moreover,
    \[
    h^0(Z^c,\omega_C|_{Z^c}(-Z\cap Z^c))=h^0(Z^c,\omega_{Z^c})=\sum_{1\le i\le t_Z} g_{F^Z_i}.
    \]
    On the other hand, the genus $g$ of $C$ is given by 
\begin{equation}\label{eq:gZ}
g=g_Z+\delta_Z-t_Z+ \underset{1\le i\leq t_Z} {\sum}g_{F^Z_i}.
\end{equation}
  This concludes the proof of (i) because, by Equation \eqref{eq:gZ}, we have
   \[
 N_{\varphi(Z)}+1  = \dim(\Ima(\rho_Z))  =g-\sum_{1\le i\le t_Z} g_{F^Z_i}  = g_Z+\delta_Z-t_Z.
    \]
    
        Let $Z\in \Irr(C)$. Then $\varphi|_{Z}\colon Z\to \varphi(Z)$ is a $2$ to $1$ map if $Z$ is honest hyperelliptic, and it is birational otherwise, which proves the first assertion in  (ii). The second part of (ii) follows from the first part and the fact that, by Proposition~\ref{prop:catanese_hh}, we have $\varphi(Z) = \varphi(Z_1)=\varphi(Z_2)$  and $k_Z=2k_{Z_1}=2k_{Z_2}$, for every honest hyperelliptic pair $Z=Z_1\cup Z_2$ of $C$.

        Next, let $\ud\in \Sigma(C)$. Let $Z\subseteq C$ be a subcurve. 
         By Equation \eqref{eq:dZ-stable} we have $d_{F^Z_i} > g_{F^Z_i} - 1$ for every $i=1,\dots,t_Z$,  hence $d_{Z^c}\geq \sum_{1\le i\le t_Z}  g_{F^Z_i}$. Thus
        \[
        d_Z\leq g - 1 - \sum_{1\le i\le t_Z} g_{F^Z_i} = g_Z + \delta_Z - t_Z -1 = N_{\varphi(Z)},
        \]
        where we use Equation \eqref{eq:gZ} in the first equality and (i) in the second, proving (iii). 
        
        Finally, let $\orient$ be a strong orientation of $\Gamma_C$ such that $\ud_{\orient}=\ud\in \Sigma(C)$.  Let $Z$ be an honest hyperelliptic subcurve  of $C$.  Let $V_Z \subseteq V(\Gamma_C)$ be the set of vertices corresponding to the irreducible components of $Z$. Thus $|V_Z| \leq 2$ and $|V_Z|=2$ if and only if $Z$ is an honest hyperelliptic pair.  By Proposition~\ref{prop:catanese_hh} (iv) every connected component $Z'$ of $Z^c$ intersects $Z$ in exactly two nodes. Since $\orient$ is strong, that means that the edges corresponding to these two nodes must be one oriented towards $Z$ and the other to $Z^c$. That means that exactly half the edges in $E(Z, Z^c)$ must be oriented towards $Z$. Recall Equation \eqref{eq:another-strong} and the definition of $\tau'$  above it. We have that $\tau_{\orient, V_Z}' = t_Z= \frac{\delta_Z}{2}$, and hence
        \[
        d_{Z} = d_{\orient, V_Z} = g_Z - 1 +\frac{\delta_Z}{2} = \frac{k_Z}{2}.
        \]
    
       Moreover, again by Proposition~\ref{prop:catanese_hh} (iv), we have $\delta_Z=2t_Z$, hence, by using (i), we obtain $N_{\varphi(Z)}=g_Z-1+\frac{\delta_Z}{2}$. Finally, we have to check that $1\le d_Z<k_Z$. If $Z=C$, then $d_C=g-1$ and $k_C=2g-2$, hence clearly $1\le d_C<k_C$. So assume $Z\ne C$, hence $\delta_Z$ is positive, and it is also even since $\delta_Z=2t_Z$. Since $C$ is stable, we see that $2g_Z-2+\delta_Z$ is positive and even, so 
        $d_Z=g_Z-1+\frac{\delta_Z}{2}\geq1$, from which we also get $d_Z<2g_Z-2+\delta_Z=k_Z$. This concludes the proof of (iv).
\end{proof}

\subsection{Stable quasi-Eulerian orientations}

In this section, we introduce stable quasi-Eulerian orientations and their associated multidegrees. The key result is Proposition~\ref{prop:q_orient}, where we establish the existence and properties of such orientations. Later, quasi-Eulerian multidegrees will be crucial to get a nice description of the branch locus of the Gauss map (see Theorem~\ref{thm:branch-gamma}). We will use the notation established in Section \ref{subsec:graphs} for graphs, trails, orientations, and their multidegrees.

An \emph{Eulerian} orientation on $\Gamma$ is an orientation ${\orient}=(\sigma_{\orient},\tau_{\orient})$ such that there is an $\orient$-oriented trail $P$ that begins and ends at the same vertex of $\Gamma$, and such that $E(P)= E(\Gamma)$. An Eulerian orientation on a connected graph is always strong.  A graph is Eulerian if and only if it admits an Eulerian orientation on each connected component. If $\Gamma$ is a cycle not consisting of a single loop, then $\Gamma$ has $2$ Eulerian orientations. 

Let $\Gamma$ be a graph. Let ${\orient}=(\sigma_{\orient},\tau_{\orient})$ be an orientation on $\Gamma$. For $v\in V(\Gamma)$, we set
\begin{equation}\label{eq:Delta}
\Delta_{\orient}(v) := \tau_{\orient,v}-\sigma_{\orient,v} = \tau_{\orient,v}'-\sigma_{\orient,v}'.
\end{equation}

\begin{definition}
 Let $\Gamma$ be a graph. An orientation ${\orient}$ on $\Gamma$ is \emph{quasi-Eulerian} if $|\Delta_{\orient}(v)|\le 1$, for every $v\in V(\Gamma)$. 
\end{definition}

Notice that an orientation $\orient$ on a connected graph $\Gamma$ is Eulerian
if and only if $\Delta_{\orient}(v)=0$ for every $v\in V(\Gamma)$. Thus every
Eulerian orientation is quasi-Eulerian. Conversely, if $\Gamma$ is Eulerian, then
every quasi-Eulerian orientation $\orient$ satisfies $\Delta_{\orient}(v)=0$ for every $v\in V(\Gamma)$.

\begin{definition} \label{def: qe}
    A multidegree $\ud$ on $\Gamma$ is \emph{quasi-Eulerian} (respectively \emph{Eulerian}) if $\ud=\ud_{\orient}$ for some quasi-Eulerian (respectively Eulerian) orientation ${\orient}$ on $\Gamma$.
\end{definition}

 If $\Gamma$ is connected, we let $\Sigma^{\qE}(\Gamma)\subseteq \Sigma(\Gamma)$ (respectively $\Sigma^{\E}(\Gamma) \subseteq \Sigma(\Gamma)$) be the subset of the stable multidegrees on $\Gamma$ that are also quasi-Eulerian (respectively that are also  Eulerian). Notice that $\Sigma^{\E}(\Gamma)\ne\emptyset$  if and only if $\Gamma$ is a connected Eulerian graph (and in this case $\Sigma^{\qE}(\Gamma) = \Sigma^{\E}(\Gamma)$ and $|\Sigma^{\E}(\Gamma)|=1$).

Let $\Gamma'$ be a subgraph of the graph $\Gamma$. A \emph{$\Gamma'$-ear} in $\Gamma$ is a subgraph $\Gamma''$ of $\Gamma$ satisfying one of the following conditions
\begin{itemize}
    \item[(1)]  
    $\Gamma''$ is a cycle in $\Gamma$ with $|V(\Gamma')\cap V(\Gamma'')|=1$ and $E(\Gamma') \cap E(\Gamma'')=\emptyset$.
\item[(2)]
$\Gamma''$ is a chain having its two end vertices in $V(\Gamma')$, with no internal vertex in $V(\Gamma')$ and with $E(\Gamma')\cap E(\Gamma'')=\emptyset$. 
\end{itemize}

In either case, we can think of an ear $\Gamma''$ as the underlying graph of a trail that starts and ends in $\Gamma'$ with no edges belonging to $\Gamma'$.

\begin{definition}
    Let $\Gamma$ be a graph. An \emph{ear-decomposition} of $\Gamma$ is a sequence of subgraphs $(\Gamma_1, \dots,\Gamma_n)$ such that $\Gamma=\cup_{1\le i\le n} \Gamma_i$, where $\Gamma_1$ is an isolated vertex and $\Gamma_j$ is a $\left(\cup_{1\le i\le j-1}\Gamma_i\right)$-ear for every $j=2,\dots,n$. We denote by $v_i$ and $w_i$ the (possibly coinciding) end-vertices of $\Gamma_i$.
\end{definition}

Recall that $\Gamma$ admits an ear decomposition if and only if $\Gamma$ is $2$-edge connected. If that is the case, then $\Gamma_2$ is necessarily a cycle and  $b_1(\Gamma)=n-1$. 

We first record a simple balancing lemma: after prescribing the orientation of one edge, one can  orient the  graph so that the incoming and outgoing valences differ by at most $1$.

\begin{lemma}
\label{lem:balanced-prescribed-edge}
Let $\Gamma$ be a graph, and let $h\in E(\Gamma)$. Fix an orientation on
$h$. Then there exists a quasi-Eulerian orientation $\orient$ on $\Gamma$ extending the 
orientation of $h$.
\end{lemma}

\begin{proof}
It is enough to argue on each connected component of $\Gamma$. Let $\Gamma_0$ be
a connected component.

Pair the odd-valent vertices of $\Gamma_0$
and connect the two vertices of each pair with one temporary additional edge. The resulting connected graph
$\Gamma_0^+$ is Eulerian. Choose an Eulerian orientation $\orient^+$ of $\Gamma_0^+$
starting by traversing $h$ in the prescribed direction if $h\in E(\Gamma_0)$. 
Orient the non-temporary edges of $\Gamma_0$ according to $\orient^+$, and then
discard the temporary edges. In $\Gamma_0^+$ every vertex $v$ has $\Delta_{\orient^+}(v)=0$, and
it is incident to at most one temporary edge. Hence
deleting the temporary edges gives rise to an orientation $\orient$ on $\Gamma$ such that  $\Delta_{\orient}(v)$ differs from $\Delta_{\orient^+}(v)$ at each vertex $v$ by $0$, $1$, or
$-1$. Therefore, $|\Delta_{\orient}(v)|\leq 1$ for every $v\in V(\Gamma_0)$.
\end{proof}

The next lemma gives a criterion for strong connectivity for orientations constructed along an ear decomposition.

\begin{lemma}\label{lem:strong-qE}
    Let $\Gamma$ be a $2$-edge connected graph. Let ${\orient}$ be an orientation on $\Gamma$. Assume that $(\Gamma_1,\dots,\Gamma_n)$ is an ear decomposition of $\Gamma$ such that ${\orient}|_{\Gamma_i}$ is a quasi-Eulerian orientation on $\Gamma_i$ for $i=1,\dots,n$. Then ${\orient}$ is strong.
\end{lemma}

\begin{proof}
 We proceed by induction on $n$. If $n=1$ then $\Gamma$ is an isolated vertex and we have nothing to prove. If $n=2$, then $\Gamma$ is a cycle. Any quasi-Eulerian orientation on a cycle has
$\Delta_{\orient}(v)=0$ for every $v\in V(\Gamma)$, hence it is Eulerian and strong.
    
    Let $n\ge3$. Set $\Gamma'=\cup_{1\le i\le n-1} \Gamma_i$ and ${\orient}':={\orient}|_{\Gamma'}$. We let $u_n$ and $z_n$ be the (possibly coinciding) vertices of $\Gamma_n$ such that $u_n,z_n\in V(\Gamma')$. Assume, without loss of generality, that $u_n=\sigma_{\orient}(e)$ for some $e\in E(\Gamma_n)$.  
    Notice that $(\Gamma_1,\dots,\Gamma_{n-1})$ is an ear decomposition of $\Gamma'$. Since ${\orient}'|_{\Gamma_i}={\orient}|_{\Gamma_i}$ is quasi-Eulerian for every $i=1,\dots, n-1$, by induction, the orientation ${\orient}'$  on $\Gamma'$ is strong. In particular, every vertex in $V(\Gamma')$ is strongly $\orient$-connected to $u_n$ and $z_n$.

    Consider $v\in V(\Gamma)\setminus V(\Gamma')$. It is clear that  $u_n$ is ${\orient}|_{\Gamma_n}$-connected to $v$, hence $u_n$ is ${\orient}$-connected to $v$. Similarly, $v$ is ${\orient}$-connected to $z_n$. On the other hand, $z_n$ is ${\orient}$-connected to $u_n$. Hence $v$ is ${\orient}$-connected to $u_n$ and we have proven that $v$ is strongly $\orient$-connected to $u_n$. This shows that every vertex in $V(\Gamma)$ is strongly $\orient$-connected to $u_n$, which shows that the orientation $\orient$ is strong.
\end{proof}

Combining the preceding two lemmas, we obtain the quasi-Eulerian orientation needed for the proof of Proposition~\ref{prop:q_orient}.

\begin{lemma}
\label{lem:comp-orient}
Let $\Gamma$ be a connected graph with no separating edges, and let
$a,b\in V(\Gamma)$. Then there exists a strong orientation $\orient_\Gamma$ of
$\Gamma$ with the following property. Let $\widehat{\Gamma}$ be obtained from
$\Gamma$ by adding one auxiliary edge $f$ connecting $a$ and $b$. If
$\widehat{\orient}_\Gamma$ is the orientation of $\widehat{\Gamma}$ extending
$\orient_\Gamma$ by orienting $f$ from $a$ to $b$, then $\widehat{\orient}_\Gamma$ is
quasi-Eulerian.
\end{lemma}

\begin{proof}
If $\Gamma$ has $1$ vertex, the statement is immediate. Assume then that $\Gamma$ has at least $2$ vertices. Since $\Gamma$ has no
separating edges, it is $2$-edge connected, so there is an ear decomposition
$(\Gamma_1,\ldots,\Gamma_n)$ of $\Gamma$. Thus
$\Gamma_1$ is the initial isolated vertex, and $\Gamma_j$ is an ear for every
$j\geq 2$.

We define an auxiliary graph $G$ as follows. Its
vertex set is $V(\Gamma)$, and for every $j\geq 2$ it has one edge $h_j$ joining
the two (possibly coinciding) end-vertices  of the ear $\Gamma_j$. The initial isolated vertex $\Gamma_1$ contributes no edge.

Add to $G$ one auxiliary edge $h$ from $a$ to $b$, and denote
the resulting graph by $\widehat{G}$. By
Lemma~\ref{lem:balanced-prescribed-edge}, there exists a quasi-Eulerian orientation
$\widehat{\orient}$ of $\widehat{G}$ such that $h$ is oriented
from $a$ to $b$. 

Replace each oriented edge $h_j$ by the corresponding ear $\Gamma_j$, oriented consistently if its end-vertices are distinct and cyclically otherwise. The resulting orientation $\orient_\Gamma$ is strong by Lemma~\ref{lem:strong-qE}. Replacing each $h_j$ by $\Gamma_j$, and $h$ by $f$, preserves at every vertex the difference between the numbers of incoming and outgoing edges. Thus $\Delta_{\widehat{\orient}_\Gamma}(v)=\Delta_{\widehat{\orient}}(v)$ for every $v\in V(\Gamma)$, so $\widehat{\orient}_\Gamma$ is quasi-Eulerian. \end{proof}

We now apply this auxiliary-edge construction to the connected components of the complement of a fixed $\Cone$-set.

\begin{proposition}
\label{prop:q_orient}
Let $\Gamma$ be a stable $2$-edge connected graph. For every $\Cone$-set
$S\subseteq E(\Gamma)$, there is an orientation $\orient$ on $\Gamma$ with
associated multidegree $\ud$ such that
\begin{enumerate}
\item[(i)] $\orient$ is strong and quasi-Eulerian;
\item[(ii)] for every $e\in S$, if $v=\sigma_{\orient}(e)$, then $d_v\leq k_v-1$;
\item[(iii)] for every $e\in S$, if $w=\tau_{\orient}(e)$, then $d_w\geq 1$;
\item[(iv)] if $\Gamma_1,\ldots,\Gamma_k$ are the connected components of
$\Gamma\setminus S$, then $\orient|_{\Gamma_i}$ is strong for every $i$.
\end{enumerate}
\end{proposition}

\begin{proof}
Let $\Gamma_1,\ldots,\Gamma_k$ be the connected components of
$\Gamma\setminus S$. Since $S$ is a $\Cone$-set,
$\Gamma/(E(\Gamma)\setminus S)$
is a cycle, and $\Gamma\setminus S$ has no separating edges. In particular, each $\Gamma_i$ has no separating edges.

The vertices of the quotient graph $\Gamma/(E(\Gamma)\setminus S)$ correspond to
the connected components $\Gamma_i$, and its edges are the elements of $S$ (see Figure \ref{fig:C_1-set}). Since
this quotient is a cycle, we may label the components and the edges of $S$
cyclically as
\(
S=\{e_1,\ldots,e_k\},
\)
so that $e_i$ joins $\Gamma_i$ to $\Gamma_{i+1}$, using indices modulo $k$ (so 
$\Gamma_{k+1}=\Gamma_1$ and $e_{0}=e_k$). In the case $k=1$, this means that the quotient cycle
has one vertex and one loop; in this case, the edge $e_1$ may be either a loop in $\Gamma$ or a
non-loop edge whose endpoints lie in the same component of $\Gamma\setminus S$.

Orient each $e_i$ from $\Gamma_i$ to $\Gamma_{i+1}$ (when $k=1$ and $e_1$ is a non-loop edge, choose either orientation for $e_1$), and let $v_i\in V(\Gamma_i)$ and $w_{i+1}\in V(\Gamma_{i+1})$ be the source and target of $e_i$.  

Apply Lemma~\ref{lem:comp-orient} to $\Gamma_i$
with $a=v_i$ and $b=w_i$. We obtain a strong orientation $\orient_i$ of $\Gamma_i$
such that, after adding an auxiliary edge
$f_i$ oriented from $v_i$ to $w_i$,
the extended orientation $\widehat{\orient}_i$ of $\widehat{\Gamma}_i$ (notation as in Lemma~\ref{lem:comp-orient}) is
quasi-Eulerian. Define $\orient$ on $\Gamma$ by setting $\orient|_{\Gamma_i}=\orient_i$, and by
keeping the cyclic orientation of the edges $e_1,\ldots,e_k$.

For every $x\in V(\Gamma_i)$, the contribution to $\Delta_{\orient}(x)$ from the two
$S$-incidences of the quotient at $\Gamma_i$, namely $e_{i-1}$ and $e_i$ counted
cyclically and with multiplicity, is exactly the contribution of the auxiliary edge
$f_i$ oriented from $v_i$ to $w_i$. Hence $\Delta_{\orient}(x)=\Delta_{\widehat{\orient}_i}(x)
\text{ for every } x\in V(\Gamma_i).$

Since each $\widehat{\orient}_i$ is quasi-Eulerian, $\orient$ is quasi-Eulerian. Moreover,
each $\orient_i$ is strong, and the edges $e_1,\ldots,e_k$ orient the quotient
$\Gamma/(E(\Gamma)\setminus S)$ as a directed cycle. It follows that $\orient$ is
strong on $\Gamma$. This proves (i) and (iv). 

It remains to prove (ii) and (iii). Recall that $d_x=g_x-1+\tau'_{\orient,x}$ and
$k_x=2g_x-2+\delta_x$ for every $x\in V(\Gamma)$. Since $\orient$ is quasi-Eulerian and
$\tau'_{\orient,x}+\sigma'_{\orient,x}=\delta_x$, we have
$\lfloor \delta_x/2\rfloor \le \tau'_{\orient,x}\le \lceil \delta_x/2\rceil$.

First suppose that  $e\in S$ is a loop at $v$. Then $S=\{e\}$. Loops contribute neither
to $\delta_v$ nor to $\tau'_{\orient,v}$. If $\delta_v=0$ (the graph with $1$ vertex), stability gives $g_v\ge 2$, hence
\[1\le d_v=g_v-1\le 2g_v-3=k_v-1.\] If $\delta_v>0$, then $\delta_v\ge 2$ because
$\Gamma$ has no separating edges, and $g_v\ge 1$ because $e$ is a loop at $v$. Thus \[1\le g_v-1+\lfloor\delta_v/2\rfloor\le d_v\le
g_v-1+\lceil\delta_v/2\rceil\le 2g_v-3+\delta_v=k_v-1\]
(the first and the last inequality use $g_v\ge 1$ and $\delta_v\ge 2$). Hence  (ii) and (iii) hold when $e$ is a loop.

Now let $e_i\in S$ be an edge that is  not a loop. Write $v_i=\sigma_{\orient}(e_i)\in V(\Gamma_i)$ and
$w_i=\tau_{\orient}(e_{i-1})\in V(\Gamma_i)$. Since $\Gamma_i$ has no separating edges, if
$v_i=w_i$, then either $\delta_{v_i}=2$ or $\delta_{v_i}\ge 4$; while if $v_i\neq w_i$,
then $\delta_{v_i}\ge3$ and $\delta_{w_i}\ge 3$. Also, if $x\in\{v_i,w_i\}$ and $\delta_x=3$,
then the two internal incidences of $x$ inside $\Gamma_i$ cannot be oriented in the same
direction, because $\orient|_{\Gamma_i}$ is strong. Thus exactly one of them enters $x$.

We now prove (ii). Let $v=v_i=\sigma_{\orient}(e_i)$.

If $v_i=w_i$ and $\delta_v=2$, then $\tau'_{\orient,v}=1$ and also $g_v\ge1$ by the stability of $\Gamma$. Hence
\[
d_v=g_v\le 2g_v-1=k_v-1.
\]

If $\delta_v\ge 4$, then, using quasi-Eulerianity,
\[
d_v\le g_v-1+\left\lceil\frac{\delta_v}{2}\right\rceil
\le 2g_v-3+\delta_v
= k_v-1.
\]

It remains to consider the case $v_i\ne w_i$ and $\delta_v=3$. Then $e_i$ is the unique edge of $S$ incident to $v$, and it is outgoing. By the previous paragraph, exactly one internal incidence of $v$ inside $\Gamma_i$ enters $v$. Hence $\tau'_{\orient,v}=1$, so
\[
d_v=g_v\le 2g_v=k_v-1.
\]

We now prove (iii). Fix $i$, and let us prove the inequality for the edge $e_{i-1}$, whose target is
$w_i=\tau_{\orient}(e_{i-1})$. If $v_i=w_i$ and $\delta_{w_i}=2$, then
$\tau'_{\orient,w_i}=1$, so $d_{w_i}=g_{w_i}\ge 1$ by the stability of $\Gamma$. If $\delta_{w_i}\ge 4$, then
$d_{w_i}\ge g_{w_i}-1+\lfloor\delta_{w_i}/2\rfloor\ge 1$. The only remaining case is
$v_i\neq w_i$ and $\delta_{w_i}=3$. Then $e_{i-1}$ is the unique edge in $S$ incident to $w_i$,
and it is incoming; by the previous paragraph, exactly one internal incidence inside $\Gamma_i$ enters $w_i$.
Hence $\tau'_{o,w_i}=2$, so $d_{w_i}=g_{w_i}+1\ge 1$.
This concludes our proof.
\end{proof}

We conclude this section with a numerical characterization of the extremal values of the multidegree associated to a strong orientation.

\begin{lemma}\label{lem:numerical}
    Let $\Gamma$ be a connected stable graph and let $\ud=\ud_{\orient}\in \Sigma(\Gamma)$. Then $0\le d_v\le k_v$ for all $v\in V(\Gamma)$. Moreover, $d_v=0$ (respectively, $d_v=k_v$)
    if and only if $g_v=0$, $t_v=1$, $\Delta_{\orient}(v)=-\delta_v+2$ (respectively, $\Delta_{\orient}(v)=\delta_v-2$). 
    
    If, in addition, $\ud\in \Sigma^{\qE}(\Gamma)$, then $d_v=0$ (respectively, $d_v=k_v$) if and only if $g_v=0$,  $t_v=1$, $\delta_v=3$, $\Delta_{\orient}(v)=-1$ (respectively, $\Delta_{\orient}(v)=1$).
\end{lemma}

\begin{proof}
Recall Equations~\eqref{eq:d-orient}, \eqref{eq:strong}, \eqref{eq:Delta}. The case $|V(\Gamma)|=1$ is immediate;  we assume $|V(\Gamma)|\geq 2$.
Applying Equation~\eqref{eq:strong} to $V=\{v\}$ gives
\begin{equation} \label{eq: ineq-dv}
g_v-1+t_v\le d_v\le g_v-1+\delta_v-t_v.
\end{equation}

If $g_v\ge 1$, then, by the stability of $\Gamma$, we have $g_v-1+t_v>0$. Since
\[
k_v-(g_v-1+\delta_v-t_v)=g_v-1+t_v,
\]
we get $0<d_v<k_v$.

Assume now that $g_v=0$. Then $w_\Gamma(v)=0$, so
$d_v=\tau_{\orient,v}-1$. We have 
\[
\Delta_{\orient}(v)=\tau_{\orient,v}-\sigma_{\orient,v}
=2\tau_{\orient,v}-\delta_v.
\]
Since $\Gamma$ is connected and stable, $\delta_v\ge 3$ and $t_v\ge 1$; hence \eqref{eq: ineq-dv} gives
\[
0\le d_v\le \delta_v-2=k_v.
\]
Moreover, equality on either side forces $t_v=1$. Therefore
\begin{equation}\label{eq: dv=0}
d_v=0
\iff t_v=1,\ \tau_{\orient,v}=1
\iff t_v=1,\ \Delta_{\orient}(v)=-\delta_v+2,
\end{equation}
and similarly
\begin{equation} \label{eq: dv=kv}
d_v=k_v
\iff t_v=1,\ \tau_{\orient,v}=\delta_v-1
\iff t_v=1,\ \Delta_{\orient}(v)=\delta_v-2.
\end{equation}

Finally, assume that $\ud\in\Sigma^{\qE}(\Gamma)$. Thus, the chosen
orientation $\orient$ satisfies $|\Delta_{\orient}(v)|\le 1$. If $d_v=0$, we have seen that $g_v=0$, $t_v=1$ and $|\Delta_{\orient}(v)|=\delta_v-2$. Hence, the stability of $\Gamma$ implies $\delta_v=3$ so, by \eqref{eq: dv=0} and \eqref{eq: dv=kv} we have $\Delta_{\orient}(v)=-1$ when $d_v=0$, and $\Delta_{\orient}(v)=1$ when $d_v=k_v$.

Conversely, if $g_v=0$, $t_v=1$, $\delta_v=3$, and
$\Delta_{\orient}(v)=-1$, then
$\Delta_{\orient}(v)=-\delta_v+2$, so \eqref{eq: dv=0} gives $d_v=0$. Similarly, if
$g_v=0$, $t_v=1$, $\delta_v=3$, and $\Delta_{\orient}(v)=1$, then
$\Delta_{\orient}(v)=\delta_v-2$, so \eqref{eq: dv=kv} gives $d_v=k_v$.
\end{proof}

\section{Gauss maps and their graphs}\label{sec:gauss-graph}

\subsection{The Gauss map of a stable semi-abelic pair}
\label{sub: gauss}

This section introduces the Gauss map for a  stable semi-abelic pair and studies its structure, with a focus on the case arising from the compactified Jacobian of a nodal curve. In particular, we describe the irreducible components of the Gauss graph and recall the computation of the tangent space to the theta divisor at a general invertible sheaf (Lemma \ref{lem:tangent_L_theta}). These constructions provide the intrinsic Gauss-theoretic framework that will be compared in Sections \ref{sec:branch-ramif} and \ref{sec:Torelli-theorem} with the auxiliary graph introduced in Section \ref{subsec:gaussgraph} in order to compute the branch locus of the projection from the Gauss graph.

\begin{definition} (\cite{alexeev}) \label{def: SSAP}
A \emph{stable semi-abelic pair (SSAP)} is a pair $A = (G\action P, \Theta)$ such that 
\begin{itemize}
    \item[(a)] $G$  is a semiabelian variety over $k$.
    \item[(b)] $P$ is a seminormal, connected, projective variety of pure dimension \footnote{We follow \cite[Definition 1.2.4(ii)]{Cap-Viv} in requiring $P$ to be pure dimensional; \cite[Definition 1.1.5]{alexeev} only requires $P$ and $G$ to have the same dimension.} equal to $\dim G$.
    \item[(c)] $G$ acts on $P$ with finitely many orbits, with connected and reduced stabilizers contained in the toric part of $G$.
    \item[(d)] $\Theta$ is an effective ample Cartier divisor on $P$, not containing any $G$-orbit.
\end{itemize}     
We let $P_0$  be the union of the maximal $G$-orbits and we set $\Theta_0 := \Theta\cap P_0$. 
\end{definition}

Given two SSAPs $A=(G\action P,\Theta)$ and
$A'=(G'\action P',\Theta')$, an isomorphism
$\alpha\colon A\cong A'$ is a pair
$\alpha=(\alpha_G,\alpha_P)$, where
$\alpha_G\colon G\cong G'$ and $\alpha_P\colon P\cong P'$
are isomorphisms such that
\[
\alpha_P(g\cdot p)=\alpha_G(g)\cdot\alpha_P(p)
\quad\text{and}\quad
\alpha_P^*\Theta'=\Theta.
\]
We denote the set of such isomorphisms by
$\Isom_{\SSAP}(A,A')$.

\begin{remark}\label{rem:connected-union-ssap}
If $A=(G\action P,\Theta)$ is a SSAP, and $P'\subseteq P$ is a  connected or empty union of irreducible components, then
$A'=(G\action P',\Theta|_{P'})$ is also a SSAP. 
\end{remark}

 Let $A = (G\action P, \Theta)$ be a SSAP. Assume that $\Theta$ is generically smooth. For each $p\in P$, consider the map (which we also denote by $p$)
\begin{align*}
    p\colon G &\to P\\
    g& \mapsto g\cdot p.
\end{align*}
If $p$ belongs to a maximal (open) stratum, then the map $p$ is an open embedding. 

\begin{definition}\label{def:GaussA}
    Let $A=(G\action P, \Theta)$ be a SSAP  and consider $\Theta^{\sm}_0=\Theta^{\sm} \cap \Theta_0$. 
The \emph{Gauss map} of $A$ is defined as
\begin{align*}
    \G\colon \Theta^{\sm}_0 & \to \mathbb{P}(T_0G)^{\vee}\\
         p&\mapsto \mathbb{P}(T_0p^{-1}\Theta).
\end{align*}
(Since $p\in \Theta^{\sm}_0$, we have that $0$ is a smooth point of $p^{-1}(\Theta)$.) The (closure of the) \emph{graph of the Gauss map} of $A$ is defined as
\[
\Upsilon = \Upsilon_A := \overline{\{(p, \G(p)), p\in \Theta^{\sm}_0\}}\subseteq \Theta\times \mathbb{P}(T_0G)^\vee.
\]
\end{definition}

We will denote the composition of the inclusion $\Upsilon\subseteq\Theta\times \mathbb{P}(T_0G)^{\vee}$ with the projections onto the first and second factor by
\begin{equation}\label{eq:zeta-gamma}
\zeta\colon \Upsilon\to \Theta
\; \text{ and } \; 
\gamma\colon \Upsilon\to \mathbb{P}(T_0G)^{\vee}
\end{equation}
 respectively. Given an irreducible component $\Theta_{\lambda}$ of $\Theta$, consider the restriction 
\begin{align*}
    \G_{\Theta_\lambda}\colon \Theta_{\lambda}\cap \Theta^{\sm}_0 & \to \mathbb{P}(T_0G)^{\vee}
\end{align*}
of $\G$ to $\Theta_{\lambda}$.  We obtain  a decomposition 
\begin{equation}\label{eq:upsilon-dec}
\Upsilon=\bigcup_{\lambda}\Upsilon_{\Theta_{\lambda}},
\end{equation}
where each $\Upsilon_{\Theta_{\lambda}}$ is irreducible and is defined as 
\[
\Upsilon_{\Theta_{\lambda}} = \overline{\{(p, \G_{\Theta_\lambda}(p)) :  p\in \Theta_\lambda\cap\Theta^{\sm}_0\}}\subseteq \Theta_\lambda\times \mathbb{P}(T_0G)^{\vee}.
\]
We denote the restriction of $\gamma$ to $\Upsilon_{\Theta_\lambda}$ by
\begin{equation}\label{eq:gamma-lambda}
    \gamma_{\Theta_\lambda}\colon \Upsilon_{\Theta_\lambda}\to \mathbb{P}(T_0G)^{\vee}.
\end{equation}

For consistency with the notation in the next sections, we will denote by $H$ a point in $\mathbb{P}(T_0G)^{\vee}$ and think of it as a hyperplane in $\mathbb{P}(T_0G)$.

Next we consider SSAPs arising from canonical compactified Jacobians of nodal curves in degree $g-1$, which we now define. 
Let $X$ be a nodal curve and consider a multidegree $\ud\in \mathbb Z^{V(\Gamma_X)}$. When $X$ is connected, recall the definition of the set $\Sigma(X)$ of stable multidegrees in Section \ref{sec:abstract_curves}. More generally, if $X$ is a not necessarily connected nodal curve, we say that a multidegree $\ud\in \mathbb Z^{V(\Gamma_X)}$ is \emph{stable} if its restriction to each connected component of $X$ is stable. 

 Given an invertible sheaf $I$ on a (not necessarily connected) curve $X$, we let $\ud^I$ be the multidegree such that $d^I_v=\deg I|_{X_v}$ for every $v\in V(\Gamma_X)$. We say that $I$ is stable if so is $\ud^I$. We let 
 \[
 P^{g_X-1}_X=\{[I] : I \text{ is a  stable invertible sheaf on } X\}.
 \]

Let $C$ be a connected nodal curve of genus $g$. Let $\Pic$ be the canonical compactified Jacobian in degree $g-1$ (see \cite{captheta}). We have a stratification  
\[
\Pic=\bigsqcup_{S \subseteq C^{\text{sing}}} P_{C_S}^{g_{C_S}-1}
\]
where, given a subset $S\subseteq C^{\sing}$, we let $C_S \to C$ be the partial normalization of $C$ at $S$ (see \cite[Fact~4.1.5]{captheta}). Thus, the elements of $\Pic$ are pairs $(S,[I])$, where $S\subseteq C^{\sing}$ and $I$ is a stable invertible sheaf on $C_S$. 

Moreover, the \emph{theta divisor} is defined as
\[
\Theta(C)= \{(S,[I])\in \Pic : h^0(C_S,I)> 0\}
\]
(see \cite[Lemma--Definition~4.2.1]{captheta} for more details).

When $C$ is a connected nodal curve of genus $g$, the pair 
\begin{equation} \label{eq:AC}
A(C) := (J(C)\action \Pic, \Theta(C)),
\end{equation}
where $J(C)$ is the generalized Jacobian of $C$, is a SSAP (see \cite[Fact 1.2.10]{Cap-Viv}). In particular, $\Theta(C)$  does not contain any stratum of $\Pic$. Moreover, $\Theta(C)$ is generically smooth. In fact, if $(S,[I])$ is a general point of an irreducible component of $\Theta$, we have that $S$ is the set of separating nodes of $C$ and $I$ is an invertible sheaf on $C_S$ with $h^0(C_S,I)=1$ (see \cite[Theorem~3.1.2]{captheta}).

\begin{remark} \label{rem:SSAPstab}
    If $g_C\geq 2$,  there is a canonical isomorphism $A(C) \to A(C^{\rm{st}})$ of SSAPs (see \cite[Remark~1.3.1]{Cap-Viv}).
\end{remark}

For every $\ud\in \Sigma(C)$, we let $\overline{P_C^{\ud}}$ be the closure inside $\Pic$ of the locus 
\[
P_C^{\ud}=\{(\emptyset,[I]) : I \text{ is a stable invertible sheaf on $C$ with } \ud^I=\ud\}.
\]

\begin{remark}\label{rem:bijection-irred}
If $C$ has no separating nodes, there is a bijection between the set of stabel multidegree $\Sigma(C)$ and the set of irreducible components of $\Pic$ mapping each $\ud\in \Sigma(C)$ to $\overline{P_C^{\ud}}$. 
\end{remark}

For every $\ud\in\Sigma(C)$, we denote by $\Theta_{\ud}:=\Theta(C)\cap \overline{P_C^{\ud}}$ the corresponding
irreducible component of $\Theta(C)$. 
 If $C$ has no separating nodes,
these are precisely all the irreducible components of $\Theta(C)$;
see \cite[Theorem~3.1.2 and Lemma--Definition~4.2.1]{captheta}. Moreover, $\Theta_0\cap \Theta_{\ud}$ (recall Definition \ref{def: SSAP}) is the locus of points in $\Theta_{\ud}$ corresponding to invertible sheaves on $C$. 

Next, consider the Gauss map of $A(C)$. For $\ud\in \Sigma(C)$, we set 
\begin{equation}\label{eq:upsilon-d}
\Upsilon_{\ud}:=\Upsilon_{\Theta_{\ud}}\stackrel{\gamma_{\ud}\;\;}{\longrightarrow} \Proj^\vee_C=\mathbb{P}(H^0(C,\omega_C)),
\end{equation}
where $\gamma_{\ud}:=\gamma_{\Theta_{\ud}}$ is the restriction of the projection to the second factor (see Equations \eqref{eq:upsilon-dec} and \eqref{eq:gamma-lambda}). In particular, if $C$ has no separating nodes, we have a decomposition into irreducible components $\Upsilon=\bigcup_{\ud\in\Sigma(C)}\Upsilon_{\ud}$.

Notice that $\gamma_{\ud}$ restricts to the Gauss map of $A(C)$ over $\Theta^{\sm}_0\cap\Theta_{\ud}$ (viewed as a dense open subset of $\Upsilon_{\ud}$). We also set
\begin{equation}\label{eq:zeta-d}
\zeta_{\ud}\colon \Upsilon_{\ud}\longrightarrow \Theta_{\ud},
\end{equation}
where $\zeta_{\ud}:=\zeta|_{\Upsilon_{\ud}}$ (see Equation \ref{eq:zeta-gamma}).

\begin{lemma}
\label{lem:tangent_L_theta}
Let $C$ be a connected nodal curve of genus  $g\geq 2$ with no separating nodes. Let $u=(\emptyset,[I])$ be a general element of an irreducible component of 
$\Theta(C)$. Then
       \[
    T_0 u^{-1}\Theta(C)=\Ima(H^0(C,\omega_C\otimes I ^{-1})\otimes H^0(C, I)\to H^0(C, \omega_C))^\perp \subseteq H^0(C, \omega_C)^{\vee}.
    \]
\end{lemma}

\begin{proof}
    Since $u$ is general, by \cite[Proposition 3.2.1]{captheta} we have that $h^0(C, I) = 1$, hence $h^0(C, \omega_C\otimes I^{-1})=1$ by Riemann-Roch. The result follows by standard deformation arguments, such as those in \cite[Chapter~3, Section~3.3]{sernesi}.
      
\end{proof}

\subsection{The Gauss graph of the symmetric product of a nodal curve}
\label{subsec:gaussgraph}
This section studies how certain divisors on a nodal curve relate to its canonical geometry. Starting from the canonical model of a  stable curve $C$ of genus $g$ with no separating nodes, and a multidegree $\ud$, we construct a parameter space $\Omega_{\ud}$, called the Gauss graph; see  Definition~\ref{def: omegad}. This auxiliary model is more explicit than the Gauss graph attached to the theta divisor and is therefore better suited to the ramification and branch-locus computations carried out in Section \ref{sec:branch-ramif}. We describe the structure of $\Omega_{\ud}$, particularly its projection to the space of hyperplanes. The main result of the section shows that, for $\ud\in \Sigma(C)$, the Gauss graph  is integral, Cohen–Macaulay, and has dimension $g-1$ (see Proposition \ref{prop:Omega}).

Throughout the section, $C$ will be a stable curve of genus $g\ge2$ with no separating nodes. We let $\ud=(d_v)\in \mathbb Z^{V(\Gamma_C)}$ be a multidegree of degree $|\ud|=g-1$, with $d_v\ge0$, for $v\in V(\Gamma_C)$. We will write $\varphi=\varphi_C$ for the canonical map of $C$.

As usual, we let $\widehat{C}\subseteq\Proj_C$ be the canonical model of $C$. Let $\nu\colon C^\nu\to C$ and $\widehat{\nu}\colon\widehat{C}^\nu\to \widehat{C}$ be the normalizations of $C$ and  $\widehat{C}$, respectively. There is a natural map $\widehat{\varphi}\colon C^\nu\to \widehat{C}^\nu$ giving rise to a commutative diagram
\begin{equation}\label{diag:normalizations}
\begin{tikzcd}
C^\nu \ar[d, "\widehat{\varphi}"] \ar[r, "\nu"]  & C \ar[d, "\varphi"] \\
\widehat{C}^\nu \ar[r, "\widehat{\nu}"] & \widehat{C}.
\end{tikzcd}
\end{equation}
We let $C_{\h}$ be the subcurve of $C$ that is the union of all honest hyperelliptic subcurves  of $C$ (see Definition~\ref{def: hh}), and $C_0$ be the complementary curve of $C_{\h}$ in $C$. Here and below, inverse images of subcurves are understood in the sense
of Convention~\ref{conv:curve-part}. We set 
\[
C^\nu_0 := \nu^{-1}(C_0), \;\;\; C_{\h}^{\nu} := \nu^{-1}(C_{\h}) , \;\;\; 
    \widehat{C}^\nu_0 := \widehat{\varphi}(C^\nu_0) 
, \;\;\; \widehat{C}^\nu_{\h} := \widehat{\varphi}(C^{\nu}_{\h}).
\]

    There is a natural involution $\iota_{\h}\colon C^\nu_{\h}\to C^\nu_{\h}$ induced by the involutions of the honest hyperelliptic subcurves (see Definition \ref{def:non-hyp}).  The restrictions of $\widehat{\varphi}$ to $C^\nu_0$ and $C^\nu_{\h}$ are, respectively,  an isomorphism $C^\nu_0\to \widehat{C}^\nu_0$ and the quotient $C^\nu_{\h}\to \widehat{C}^\nu_{\h}$ by the involution $\iota_{\h}$ (in particular, $\widehat{C}^\nu_{\h}$ is a disjoint union of rational curves). 
    
    We also write $C_{\h,\irr}\subseteq C_{\h}$ for the union of all irreducible honest hyperelliptic subcurves of $C$, and we let $C_{\h,\pair}$ be the complementary curve of $C_{\h,\irr}$ in $C_{\h}$ (i.e., the union  of all honest hyperelliptic pairs of $C$). We set 
\[
C^\nu_{\h,\irr}:=\nu^{-1}(C_{\h,\irr}), \;\;\; C^\nu_{\h,\pair}:=\nu^{-1}(C_{\h,\pair}),
\]
\[
\widehat{C}^\nu_{\h,\irr}:=\widehat{\varphi}(C^\nu_{\h,\irr}), \;\;\; \widehat{C}^\nu_{\h,\pair}:=\widehat{\varphi}(C^\nu_{\h,\pair}).
\]
Moreover, we set
\begin{equation}\label{eq:decomposition}
\begin{array}{lll}
   V_0:=V(\Gamma_{C_0}), & V_1:=V(\Gamma_{C_{\h,\irr}}), & 
   V_2:=V(\Gamma_{C_{\h,\pair}}),\\
   V^\nu_0:=V(\Gamma_{C_0^\nu}), & 
   V^\nu_1:=V(\Gamma_{C_{\h,\irr}^\nu}),& 
   V^\nu_2:=V(\Gamma_{C^\nu_{\h,\pair}}),\\
   \widehat{V}^\nu_0:=V(\Gamma_{
\widehat{C}_0^\nu}), & 
\widehat{V}^\nu_1:=V(\Gamma_{\widehat{C}_{\h,\irr}^\nu}),& 
\widehat{V}^\nu_2:=V(\Gamma_{\widehat{C}^\nu_{\h,\pair}}).
\end{array}
\end{equation}
Using this, we can write the following decompositions
\begin{equation}\label{eq:V-decomposition}
   V(\Gamma_C)=V_0\sqcup V_1\sqcup V_2,\quad
   V(\Gamma_{C^\nu})=V^\nu_0\sqcup V^\nu_1\sqcup V^\nu_2,\quad
V(\Gamma_{\widehat{C}^\nu})=\widehat{V}^\nu_0\sqcup \widehat{V}^\nu_1\sqcup \widehat{V}^\nu_2.
\end{equation}

Next, the maps $\nu\colon C^\nu\to C$ and $\widehat{\varphi}\colon C^\nu\to \widehat{C}^\nu$ induce a composed function
\begin{equation}\label{eq:fg}
f\colon V(\Gamma_C)\stackrel{g}{\to} V(\Gamma_{C^\nu})\stackrel{g'}{\to}V(\Gamma_{\widehat{C}^\nu}),
\end{equation}
where $g$ is a bijection (it restricts to bijections $V_i\to V_i^\nu$ for $i=0,1,2$) and $g'$ restricts to bijections $V_i^\nu\to \widehat{V}_i^\nu$ for $i=0,1$ and to a $2:1$ surjection $V_2^\nu\to \widehat{V}_2^\nu$. For every $v\in \widehat{V}^\nu_2$, we write $f^{-1}(v)=\{v',v''\}\subseteq V_2$.  

For $v\in V(\Gamma_C)$, recall that $C_v$ is the irreducible component of $C$ corresponding to $v$. Similarly, for $v\in V(\Gamma_{C^\nu})$ (respectively, $v\in V(\Gamma_{\widehat{C}^\nu})$) we denote by $C^\nu_v$ (respectively, $\widehat{C}^\nu_v$)  the component of $C^\nu$ (respectively,  $\widehat{C}^\nu)$ corresponding to $v$. For a multidegree $\ud=(d_v)\in \mathbb Z^{V(\Gamma_C)}$ such that $d_v\ge0$ for all $v$, we define  
\[
S_{\ud}(C^\nu) := \prod_{v\in V(\Gamma_C)}\Symm^{d_v}(C^\nu_{g(v)}).
\]
Similarly, we define
\[
S_{\ud}(\widehat{C}^\nu) 
:=\prod_{v\in V_0\cup V_1}\Symm^{d_v}(\widehat{C}^\nu_{f(v)})\times\prod_{v\in \widehat{V}^\nu_2}\Symm^{d_{v'}+d_{v''}}(\widehat{C}^\nu_v).
\]

The morphism $\widehat{\varphi}\colon C^\nu\to \widehat{C}^\nu$ gives rise to a natural finite morphism
\begin{equation}\label{eq:phi-hat-star}
\widehat{\varphi}_* \colon S_{\ud}(C^\nu) \to S_{\ud}(\widehat{C}^\nu)
\end{equation}
taking a divisor $D\in S_{\ud}(C^\nu)$ to its pushforward $\widehat{\varphi}_*(D)\in S_{\ud}(\widehat{C}^\nu)$. 
For any divisor $D\in S_{\ud}(C^\nu)$, let $D_0$ and $D_{\h}$ denote the restrictions of $D$ to $C^\nu_0$ and $C^\nu_{\h}$, so we have 
\begin{equation} \label{eq: defDh}
\widehat{\varphi}^*\widehat{\varphi}_*(D) = D_0 + D_{\h} + \iota_{\h}^*(D_{\h}).
\end{equation}

Next, there is a tautological inclusion 
\begin{equation}\label{eq:map1}
\mathcal O_{\mathbb P_C^\vee}(-1)
\hookrightarrow
H^0(\mathbb P_C,\mathcal O_{\mathbb P_C}(1))
\otimes \mathcal O_{\mathbb P_C^\vee},
\end{equation}
whose fiber over a point $H\in \mathbb P_C^\vee$ is the line
$\ell_H\subseteq H^0(C,\omega_C)=H^0(\mathbb P_C,\mathcal O_{\mathbb P_C}(1))$
corresponding to $H$. A non-zero section in $\ell_H$ cuts out the hyperplane
of $\mathbb P_C$ corresponding to $H$. 
There is also a universal divisor $\widehat{\mathcal{D}}\subseteq S_{\ud}(\widehat{C}^\nu)\times \widehat{C}^{\nu}$. 
Abusing notation, we write $\mathcal{O}_{\widehat{C}^\nu}(1)$ for the pullback of $\mathcal{O}_{\Proj_C}(1)\otimes \mathcal O_{\widehat{C}}$ via the map 
\[
S_{\ud}(\widehat{C}^\nu)\times \widehat{C}^\nu \to \widehat{C}^\nu\stackrel{\widehat{\nu}}{\to}\widehat{C},
\]
where the first map is the projection onto the second factor. By pushing-forward the natural map $\mathcal{O}_{\widehat{C}^\nu}(1)\to \mathcal{O}_{\widehat{C}^\nu}(1)|_{\widehat{\mc D}}$ via the projection $\pi_{\ud}\colon S_{\ud}(\widehat{C}^\nu)\times \widehat{C}^\nu\to S_{\ud}(\widehat{C}^\nu)$ onto the first factor, we obtain a map 
\begin{equation}\label{eq:map2}
H^0(\widehat{C}^\nu, \mathcal{O}_{\widehat{C}^\nu}(1))\otimes \mathcal{O}_{S_{\ud}(\widehat{C}^\nu)} = (\pi_{\ud})_*(\mathcal{O}_{\widehat{C}^\nu}(1))\to (\pi_{\ud})_*(\mathcal{O}_{\widehat{C}^\nu}(1)|_{\widehat{\mathcal{D}}}).
\end{equation}

We denote by
\[
\widehat{\pi}_1\colon S_{\ud}(\widehat{C}^\nu)\times \Proj^\vee_C\to S_{\ud}(\widehat{C}^\nu)
\quad\text{and}\quad 
\widehat{\pi}_2\colon S_{\ud}(\widehat{C}^\nu)\times \Proj^\vee_C\to \Proj^\vee_C
\]
the projections onto the first and second factors. Over $S_{\ud}(\widehat{C}^\nu)\times \Proj^\vee_C$, by pulling back via $\widehat \pi_2$ the map in Equation~\eqref{eq:map1}, composing with the restriction \[H^0(\mathbb P_C,\mathcal O_{\mathbb P_C}(1))
\to
H^0(\widehat C^\nu,\mathcal O_{\widehat C^\nu}(1)),\] and then composing with the pullback of the map in  Equation~\eqref{eq:map2} via $\widehat{\pi}_1$, we obtain the composite
\begin{equation}\label{eq:sigma}
\sigma_C\colon \widehat\pi_2^*\mathcal O_{\mathbb P_C^\vee}(-1)
\to
H^0(\widehat C^\nu,\mathcal O_{\widehat C^\nu}(1))
\otimes
\mathcal O_{S_{\underline d}(\widehat C^\nu)\times \mathbb P_C^\vee}
\to
\widehat\pi_1^*(\pi_{\ud})_*
\bigl(\mathcal O_{\widehat C^\nu}(1)|_{\widehat {\mathcal{D}}}\bigr).
\end{equation}
Equivalently, after tensoring by
$\widehat\pi_2^*\mathcal O_{\mathbb P_C^\vee}(1)$, we may regard
$\sigma_C$ as a section of a vector bundle of rank $|\underline d|=g-1$. Notice that $\widehat\pi_1^*(\pi_d)_*
\bigl(\mathcal O_{\widehat C^\nu}(1)|_{\widehat {\mathcal{D}}}\bigr)$ is indeed a vector bundle, because $\pi_{\ud}|_{\widehat{\mathcal D}}$ is a finite morphism of degree $g-1$, hence it is flat because $\widehat{\mathcal D}$ and $S_{\ud}(\widehat{C}^{\nu})$ are smooth. We write $Z(\sigma_C)$ for the zero locus of $\sigma_C$.

\begin{remark}\label{rem:obs-zero-locus}
Let $S$ be a scheme, let $W$ be a finite-dimensional vector space, and let
$e\colon W\otimes\mathcal O_S\to \mathcal E$ be a morphism to a vector bundle $\mathcal E$ on $S$.
Let $Z\subseteq S\times \mathbb P(W)$ be the zero locus of the composition
\[
\pi_2^*\mathcal O_{\mathbb P(W)}(-1)
\longrightarrow
W\otimes\mathcal O_{S\times \mathbb P(W)}
\longrightarrow
\pi_1^*\mathcal E,
\]
where $\pi_1$ and $\pi_2$ are the projections from $S\times \mathbb P(W)$ onto the first and second factor, respectively. 
For $D\in S$, the scheme-theoretic fiber $Z_D$ is naturally identified with
\[
\mathbb P(\ker(e_D))\subseteq \mathbb P(W).
\]

Moreover, if $T\subseteq S$ is an open subset such that
$\mathcal K:=\ker(e|_T)$ is a vector subbundle of $W \otimes \mathcal{O}_T$, then
\[
Z_T=\mathbb P_T(\mathcal K),
\]
where $Z_T$ denotes the inverse image of $T$ in $Z$.
\end{remark}

\begin{definition} \label{def: omegad}
Let $C$ be a connected stable curve without separating nodes of genus $g\ge2$. Let $\ud\in \mathbb Z^{V(\Gamma_C)}$ be a multidegree with  degree $|\ud|=g-1$ with $d_v\geq 0$ for all $v$. We define $\widehat{\Omega}_{\ud}\subseteq S_{\ud}(\widehat{C}^\nu)\times \Proj^\vee_C$ as the zero locus of $\sigma_C$, i.e., we set
\[
\widehat{\Omega}_{\ud}:= Z(\sigma_C) \subseteq S_{\ud}(\widehat{C}^\nu)\times \Proj^\vee_C.
\]
Consider the finite map $(\widehat{\varphi}_*,\id)\colon S_{\ud}(C^\nu)\times\Proj_C^\vee\to S_{\ud}(\widehat{C}^\nu)\times \Proj_C^\vee$ (see Equation~\ref{eq:phi-hat-star}). We also define
\[
\Omega_{\ud}:=(\widehat{\varphi}_*, \id)^{-1}(\widehat{\Omega}_{\ud})\subseteq S_{\ud}(C^\nu)\times\Proj_C^\vee.
\]
We call $\Omega_{\ud}$ the \emph{Gauss graph of the symmetric product of $C$}.
\end{definition}
 We will denote by 
\begin{equation} \label{eqdef:kappa}
\kappa_{\ud}\colon \Omega_{\ud}\to S_{\ud}(C^\nu)
\quad \text{ and } \quad
\quad \rho_{\ud}\colon \Omega_{\ud}\to \Proj^\vee_C
\end{equation}
the restrictions to $\Omega_{\ud}$ of the projections $S_{\ud}(C^\nu)\times \Proj^\vee_C\to S_{\ud}(C^\nu)$ and $S_{\ud}(C^\nu)\times \Proj^\vee_C\to \Proj^\vee_C$ onto the first and second factor, respectively.

The scheme $\Omega_{\ud}$ will serve in Section \ref{sec:branch-ramif} as the explicit model on which we compute ramification and branch loci, before comparing them in Section \ref{sec:Torelli-theorem} with the corresponding loci arising from the Gauss graph of the theta divisor.

Since $C$ has no separating nodes, we have  $\SC=\varphi(C^{\sing})\subseteq \widehat{C}$ (see Definition~\ref{def:canonical_model}). We define the open subsets of $\mathbb P_C^\vee$ 
 \begin{align}
\U_0 &:= \{H \in \Proj^\vee_C : H\cap \SC=\emptyset\}=\{H\in \Proj^\vee_C : \varphi^*(H)\subseteq C^{\sm}\}\label{eq:U0}\\
\U_1 &:= \{H\in \Proj_C^\vee : \dim(H\cap\widehat{C})=0\}\label{eq:U1}.
\end{align}
Notice that $\U_0\subseteq\U_1$. Moreover,  $\Proj^\vee_C\setminus \U_0=\cup_{p\in\SC} \Lambda^\vee_p$, so every irreducible component of $\Proj^\vee_C\setminus \U_0$ has codimension $1$ in $\Proj^\vee_C$.

\begin{remark}
\label{rem:Omega_d} 
Let $\ud\in \Sigma(C)$. The following properties are easy consequences of the definitions. 
\begin{itemize}
    \item[(i)] $\Omega_{\ud}$ is non-empty.
    \item[(ii)] We have (recall the notation of Section \ref{sec:schemes}, Diagram \ref{diag:normalizations} and \eqref{eq: defDh} for the definitions of $D_0$ and of $D_h$)
    \[
    \Omega_{\ud, \U_1} =\{(D, H) : \ D_0+D_{\h}+\iota_{\h}^*(D_{\h})\leq \nu^*\varphi^*(H)\}\subseteq S_{\ud}(C^\nu)\times \U_1.
    \]
    \item[(iii)] Consider $(D,H)\in \Omega_{\ud, \U_1}$. Since $2\deg(D|_{C_h^\nu}) = \deg(\nu^*\varphi^*(H)|_{C_{h}^\nu})$ (by Lemma~\ref{lem:dimdeg}(iv)), we must have  $D_h+\iota_{\h}^*(D_h) = \nu^*\varphi^*(H)|_{C_{h}^\nu}$. So, if $p\in C_h^\nu$ is a point, then  
    \[
    \mu_p(D) + \mu_{\iota_{\h}(p)}(D) = \mu_{p, H},
    \]
    where $\mu_{p, H} = \mu_p(\nu^*\varphi^*(H))$.   
    In particular, if  $0<\mu_{p,H}\cdot p\le D_{\h}$,
    then we cannot have $\iota^*_{\h}(p)\le D_{\h}$, for otherwise
    \[
    \mu_p(D) + \mu_{\iota_{\h}(p)}(D) \geq \mu_{p, H} + 1.
    \]
     
     \item[(iv)]   
     Given $H\in \Proj^\vee_C$, let $W_H\subseteq \widehat{C}$ be the maximal subcurve contained in $H$. Let  $Y^\nu_H$ be the complementary subcurve of $\nu^{-1}(\varphi^{-1}(W_H))$ in $C^\nu$. Then 
    \[
    \rho_{\ud}^{-1}(H)=\{(D,H) :\ (D_0+D_{\h}+\iota_{\h}^*(D_{\h}))|_{Y^\nu_H}\le \nu^*\varphi^*(H)|_{Y^\nu_H}\}.
    \]
    (Recall Equation \eqref{eqdef:kappa}). In particular, $\dim(\rho_{\ud}^{-1}(H))=d_{\varphi^{-1}(W_H)}$ and $\rho_{\ud}$ is dominant.
\end{itemize}
\end{remark}

\begin{lemma}
\label{lem:Irred_comp}
Let $f\colon X\to Y$ be a morphism of schemes with $Y$ irreducible. For every integer $i\ge0$, define 
\[
Y_i=\{y\in Y : \dim f^{-1}(y) = i\}.
\]
Assume that the following properties hold
\begin{itemize}
    \item[(i)] for every irreducible component $X'$ of $X$ we have $\dim X'\ge\dim Y$;
    \item[(ii)] the inequality $\codim_Y Y_i \geq  i +  1$ holds for every $i>0$.
\end{itemize}
Then every irreducible component of $X$ dominates $Y$.
\end{lemma}
\begin{proof}
    Let $X'$ be an irreducible component of $X$ that does not dominate $Y$. In particular $\dim f(X') < \dim Y$. By the theorem of the dimension of the fibers, we  have $\dim f(X') + \dim f^{-1}(y) \geq \dim(X') \geq \dim(Y)$ for a general point $y \in f(X')$. Set $i_0 = \dim f^{-1}(y)$. Notice that  $i_0>0$, since $\dim f(X') < \dim Y$ and $\dim(X')\ge \dim (Y)$. We have that $Y_{i_0}\cap f(X')$ is open in $f(X')$. In particular $\dim(Y_{i_0})\geq \dim f(X') \geq \dim(Y) - i_0$, hence $\codim_YY_{i_0}\leq i_0$, a contradiction.
\end{proof}

 The following proposition describes the finite locus of the map $\rho_{\ud}\colon \Omega_{\ud}\to \Proj_C^\vee$ (see Equation \eqref{eqdef:kappa}
 and establishes the basic geometric properties of $\Omega_{\ud}$.

\begin{proposition}\label{prop:Omega}
Let $C$ be a connected stable curve of genus $g\ge2$ with no separating nodes.
Let $\ud\in\Sigma(C)$.
The following properties hold.
\begin{itemize}

\item[(i)]
   The locus where the dominant map $\rho_{\ud}\colon \Omega_{\ud}\to \Proj_C^\vee$ is finite equals
   \[
   \U_{\ud} := \{H\in \Proj^\vee_C : \dim(H\cap W)=0, \forall\; W\in \Irr(\widehat{C})
   \text{ with } d_{\varphi^{-1}(W)}>0\}.
   \]
   Notice that $\U_{\ud}$ is open and it contains $\U_1$ (and hence $\U_0$).
\item[(ii)]
   $\Omega_{\ud}$ is generically smooth, Cohen-Macaulay, integral, and of dimension $g-1$.
\end{itemize}
\end{proposition}

\begin{proof}
First, we claim that each irreducible component of $\Omega_{\ud}$ has dimension at least $g-1$. Indeed, by the discussion after Equation \eqref{eq:sigma}, $\widehat\Omega_{\ud}$ is the zero locus of a
section of a rank $g-1$ vector bundle on
$S_{\ud}(\widehat C^\nu)\times\mathbb P_C^\vee$. Pulling this section back via
$(\widehat\varphi_*,\mathrm{id})$, we see that $\Omega_{\ud}$ is the zero locus of a
section of a rank $g-1$ vector bundle on
$S_{\ud}(C^\nu)\times\mathbb P_C^\vee$. This proves the claim. 

Let us prove (i). By Remark
\ref{rem:Omega_d} (iv), $\rho_{\ud}^{-1}(H)$ is finite if and only if $H$ contains
no irreducible component $W$ of $\widehat C$ with $d_{\varphi^{-1}(W)}>0$, i.e., if
and only if $H\in \U_{\ud}$. Since $\rho_{\ud}$ is proper, this ends the proof of (i).

Let us prove (ii). Set $S_{\ud}:=S_{\ud}(C^\nu)$, and let $\kappa_{\ud}\colon \Omega_{\ud}\to S_{\ud}$ be the map defined in Equation \eqref{eqdef:kappa}. Let $\mathcal E_{\ud}$ be the pullback to $S_{\ud}$ of
$\widehat{\varphi}^*((\pi_{\ud})_*
( \mathcal O_{\widehat C^\nu}(1)|_{ \widehat{\mathcal D}}))$, and consider the
evaluation map
\begin{equation} \label{eq: evaluation}
e_{\ud}\colon H^0(C,\omega_C)\otimes \mathcal O_{S_{\ud}}\longrightarrow \mathcal E_{\ud}.
\end{equation}
For $D\in S_{\ud}$, taking the fiber over $D$ gives the evaluation map
\[
e_{\ud,D}\colon H^0(C,\omega_C)\longrightarrow
H^0(\widehat\varphi_*D,\mathcal O_{\widehat\varphi_*D}(1)).
\]

We define
\[
\V_{\ud}:=\{D\in S_{\ud}:\operatorname{rk}(e_{\ud,D})=g-1\}
=\{D\in S_{\ud}:\dim(\spann(\widehat{\varphi}_*(D)))=g-2\}.
\]
This is an open subset of $S_{\ud}$. Since $S_{\ud}$ is a product of symmetric powers of smooth
irreducible curves, $S_{\ud}$ is smooth and irreducible; hence so is $\V_{\ud}$, if nonempty.

Over $\V_{\ud}$, the map $e_{\ud}$ is surjective. Set
$\mathcal L:=\ker(e_{\ud}|_{\V_{\ud}})$. Then $\mathcal L$ is a line bundle, and Remark
\ref{rem:obs-zero-locus} gives
\begin{equation}\label{eq:kappaV}
\kappa_{\ud}^{-1}(\V_{\ud})=\mathbb P_{\V_{\ud}}(\mathcal L)\cong \V_{\ud}.
\end{equation}
Under this isomorphism, a divisor $D\in\V_{\ud}$ corresponds to the point
$(D,\spann(\widehat{\varphi}_*(D)))\in\Omega_{\ud}$. In particular,
$\kappa_{\ud}^{-1}(\V_{\ud})$ is smooth and irreducible.

We now prove that $\Omega_{\ud,\U_1}$ is irreducible. Let $G\subseteq \U_1$ be
a dense open subset such that, for every $H\in G$, the divisor $H\cap\widehat C$
is reduced and Theorem \ref{thm:general_position} applies. The fiber $\rho_{\ud}^{-1}(H)$ is nonempty for $H\in G$: choose the points
componentwise, using $0\leq d_v\leq k_v$ (see Lemma~\ref{lem:numerical}) on the non-hyperelliptic components
and Lemma~\ref{lem:dimdeg} (ii),(iv) on the honest hyperelliptic components
and pairs.

Let $(D,H)\in \rho_{\ud}^{-1}(G)$. By Remark \ref{rem:Omega_d} (iii),
Theorem \ref{thm:general_position}, and Lemma \ref{lem:dimdeg} (iii), we have
$\spann(\widehat\varphi_*D)=H$. Hence $D\in\V_{\ud}$, and therefore $\rho_{\ud}^{-1}(G)\subseteq \kappa_{\ud}^{-1}(\V_{\ud})$. Since $\rho_{\ud}^{-1}(G)$ is a nonempty open subset of the irreducible scheme
$\kappa_{\ud}^{-1}(\V_{\ud})$, it is irreducible. By (i) and the claim at the beginning of the proof, every irreducible component of $\Omega_{\ud,\U_1}$ has dimension
$g-1$ and dominates $\U_1$. Hence $\rho_{\ud}^{-1}(G)$ is dense in
$\Omega_{\ud,\U_1}$. Since $\rho_{\ud}^{-1}(G)$ is irreducible, so is
$\Omega_{\ud,\U_1}$.

Now we prove that $\Omega_{\ud}$ is irreducible. For $i\ge0$, set
\[\mathbb P^\vee_{C,i}:=\{H\in \Proj^\vee_C:\dim(\rho_{\ud}^{-1}(H))=i\}.\]
By Remark \ref{rem:Omega_d} (iv), if $i>0$, then
$\mathbb P^\vee_{C,i}$ is contained in the union of the linear spaces $\Lambda_W^\vee$ such that
$d_{\varphi^{-1}(W)}=i$. Hence, using Lemma \ref{lem:dimdeg} (iii),
\[
\dim \mathbb P^\vee_{C,i}\leq
\max_W(\dim(\Proj^\vee_C)-N_W-1)<g-1-i.
\]
Thus $\codim_{\Proj_C^\vee}\mathbb P^\vee_{C,i}\geq i+1$ for every $i>0$.
By the claim at the beginning of the proof and Lemma \ref{lem:Irred_comp}, every irreducible component of
$\Omega_{\ud}$ dominates $\Proj_C^\vee$. Since $\Omega_{\ud,\U_1}$ is irreducible,
there is only one such component. Hence $\Omega_{\ud}$ is irreducible.

The equality $\dim\Omega_{\ud}=g-1$ follows from (i) and
$\dim(\Proj^\vee_C)=g-1$. Moreover, $\Omega_{\ud}$ is Cohen-Macaulay by
\cite[Proposition 7.1]{fulton}: it has the expected dimension, it is the zero locus
of a section of a rank $g-1$ vector bundle on the Cohen-Macaulay scheme
$S_{\ud}\times \Proj^\vee_C$.

Finally, $\kappa_{\ud}^{-1}(\V_{\ud})\cong \V_{\ud}$ (see Equation \eqref{eq:kappaV}) is a nonempty smooth open subset of
$\Omega_{\ud}$. Since $\Omega_{\ud}$ is irreducible, this open subset is dense.
Thus $\Omega_{\ud}$ is generically smooth, hence generically reduced. Since a
Cohen-Macaulay scheme has no embedded associated points, generically reduced
implies reduced. Therefore $\Omega_{\ud}$ is reduced. Being irreducible and
reduced, $\Omega_{\ud}$ is integral.
\end{proof}

\section{Ramification and branch loci}\label{sec:branch-ramif}

In this section, we study the geometry of the generically finite and dominant morphism $\rho_{\ud}\colon \Omega_{\ud}\to \Proj^\vee_C$, associated with the Gauss graph defined in Equation~\eqref{eqdef:kappa}. Our goal is to analyze its ramification and branch loci, since these will later be compared in Section~\ref{sec:Torelli-theorem} with the corresponding loci arising from the intrinsic Gauss graph of the theta divisor. The main results are Proposition \ref{prop:rho-branch}, which characterizes the branch locus of $\rho_{\ud}$ in terms of (the dual of) certain subcurves of $\widehat{C}$ and (the dual of) the points in $\mathcal B_{\widehat{C}}$, and Proposition \ref{prop:rho-ramif}, which identifies enough components of the ramification of $\rho_{\ud}$ to eventually be able to describe the branch locus of the Gauss graph. 
  
Throughout this section, $C$ will be a connected stable curve of genus $g\ge2$ with no separating nodes. As usual, we let $\nu\colon C^\nu \to C$ be the normalization of $C$ and we will write $\varphi$ and $\widehat{C}$ for the canonical map and canonical model of $C$ respectively. 

Recall  the definition of the set $\mathcal B_{\widehat{C}}$ in  Definition~\ref{def:canonical_model}.  Recall also the definitions of the open subsets $\U_0 \subseteq \U_1 \subseteq \U_{\ud} \subseteq \mathbb{P}_C^{\vee}$, for a multidegree $\ud$ with $|\ud|=g-1$ in Equations \eqref{eq:U0}, \eqref{eq:U1} and Proposition \ref{prop:Omega}. For $W \in \Irr(\widehat{C})$, define
\[
\Irr_W(C)= \{Z \in \Irr(C): \overline{\varphi(Z)}=W \}.
\] 

\begin{proposition}\label{prop:rho-branch}
Let $C$ be a connected stable curve of genus $g\ge2$ with no separating nodes. Let $\ud\in \Sigma(C)$. Consider
\[
\mathcal B_{\ud}=\{W\in \Irr_{\ge2}(\widehat{C}) : 0<d_Z<k_Z, \text{ for every } Z\in \Irr_W(C)\}.
\]
The branch locus $\Br(\rho_{\ud})$ of the map $\rho_{\ud}\colon \Omega_{\ud}\to \Proj_C^\vee$ is 
    \[
    \Br(\rho_{\ud}) = \left(\bigcup_{W\in \mathcal B_{\ud}} W^\vee\right) 
    \cup\left(\bigcup_{p\in \mathcal B_{\widehat{C}}}\Lambda_p^\vee\right).
    \]
    Moreover, $\Br(\rho_{\ud}) = \Br^1(\rho_{\ud})$ and $\Br(\rho_{\ud})$ is the closure of $\Br(\rho_{\ud,\U_0})$ in $\Proj^\vee_C$. 
\end{proposition}

\begin{proof}
Recall that the morphism $\rho_{\ud}$ is finite and dominant on $\U_{\ud}\subseteq \Proj^\vee_C$   (see Proposition~\ref{prop:Omega}~(i)). By Proposition \ref{prop:Omega} (ii), we can compute $\Br(\rho_{\ud,\mathcal U_{\ud}})$ using Proposition \ref{prop:branch_fiber_number}.  Also, recall that  $C_v$ is the irreducible component of $C$ corresponding to $v\in V(\Gamma_C)$. Finally, recall Equations~\eqref{eq:decomposition} and \eqref{eq:fg} and the notation  
$\{v',v''\}=f^{-1}(v)$ for every $v\in \widehat{V}^\nu_2$. 

First, we claim that the degree of $\rho_{\ud, \U_{\ud}}$ is given by
\begin{equation}\label{eq:degree}
\deg(\rho_{\ud, \U_{\ud}})=\prod_{v\in V_0} \binom{k_v}{d_v} \prod_{v\in V_1} 2^{d_v} \prod_{v\in \widehat{V}^\nu_2}
\binom{d_{v'}+d_{v''}}{d_{v'}}.
\end{equation}
Indeed, the degree of $\rho_{\ud, \U_{\ud}}$ is the number of elements in $\rho_{\ud}^{-1}(H)$ for a general $H\in \U_{\ud}$. By Remark~\ref{rem:Omega_d} (ii), this is the number of divisors $D\in S_{\ud}(C^\nu)$ such that $D_0+D_h+\iota_{\h}^*(D_{\h})\le \nu^*\varphi^*(H)$. Let us consider this relation over the preimage of each component of $\widehat{C}$. 

If $v\in V_0$, then $\varphi|_{C_v}\colon C_v\to \varphi(C_v)$ is birational and $\nu^*\varphi^*(H)|_{C_v}$ is a reduced effective divisor of degree $k_v$. We then have to choose a degree $d_v$ divisor $D|_{C_v}\leq \nu^*\varphi^*(H)|_{C_v}$, which gives the contribution of $\binom{k_v}{d_v}$ to Formula \eqref{eq:degree}.

If $v\in V_1$, then $\varphi|_{C_v}\colon C_v\to \varphi(C_v)$ is a $2:1$ map. By Remark~\ref{rem:Omega_d} (iii), we cannot choose points in $\nu^*\varphi^*(H)|_{C_v}$ in the same preimage over a point in $H\cap \varphi(C_v)$.  By Lemma~\ref{lem:dimdeg}~(ii) and (iv) we have that $d_v = \deg(\varphi(C_v))$, so $H$ cuts $\varphi(C_v)$ in $d_v$ distinct points. For each of these $d_v$ points we have to choose exactly one preimage in $C_v$, which gives the contribution of $2^{d_v}$ to Formula~\eqref{eq:degree}. 

Finally, for $v\in \widehat{V}^\nu_2$, set $W_v:=\varphi(C_{v'}\cup C_{v''})(=\varphi(C_{v'})=\varphi(C_{v''}))$. Then $\varphi|_{C_{v'}\cup C_{v''}} \colon C_{v'}\cup C_{v''}\to W_v$ is a $2:1$ map (it restricts to isomorphisms $C_{v'}\to W_v$ and $C_{v''}\to W_v$), where $C_{v'},C_{v''},W_v$ are smooth rational curves. Again by Remark~\ref{rem:Omega_d} (iii), we cannot choose points in $\nu^*\varphi^*(H)|_{C_{v'}\cup C_{v''}}$ in the same preimage over a point in $H\cap W_v$. By Lemma \ref{lem:dimdeg} (ii) and (iv), we have $d_{v'}+d_{v''}  = \deg(W_v)$, hence $H$ cuts $W_v$ in $d_{v'}+d_{v''}$ distinct points. Inside the set of preimages of these points in $C_{v'}\cup C_{v''}$ via $\varphi\circ\nu$, we have to choose $d_{v'}$ points in $C_{v'}$ and $d_{v''}$ points in $C_{v''}$, and any two of these points have different images in $W_v$. After choosing the $d_{v'}$ points in $C_{v'}$, we have only one choice for the $d_{v''}$ points in $C_{v''}$, so we get the contribution of $\binom{d_{v'}+d_{v''}}{d_{v'}}$ to Formula~\eqref{eq:degree}. This concludes the proof of the claim.

Next, recall that $W_H\subseteq \widehat{C}$ is defined as the maximal subcurve contained in $H$, for every $H\in \Proj^\vee_C$ (see Remark \ref{rem:Omega_d} (iv)). We now compute the number of preimages of some fixed $H\in \U_{\ud}$ via $\rho_{\ud}$, by taking the product of the following terms.
\begin{itemize}
    \item[(a)] If $v\in V(\Gamma_C)$ and $\varphi(C_v)\subseteq W_H$, then $d_v=0$ (recall that $H\in \U_{\ud}$), and there is no choice of divisors on $C_v$ to be made. In particular the term corresponding to $v$ is equal to $1 = \binom{k_v}{d_v} = \binom{k_v}{0}$.
    
    \item[(b)] If $v\in V_0$ and $\varphi(C_v)\not\subseteq W_H$, write $\nu^*\varphi^*(H)|_{C_v} = \sum_{i=1}^k a_ip_i$ with $p_i\neq p_j$ for all $i \neq j$. Then the term is the number of ways to  write
    \[
    \sum_{i=1}^k x_i = d_v, \text{ for } 0\leq x_i\leq a_i.
    \]
    This term is smaller than $\binom{k_v}{d_v}$ if and only if at least one $a_i\geq 2$ and $0 < d_v < k_v$. 
    
    \item[(c)]  If $v \in V_1$ and $\varphi(C_v)\not\subseteq W_H$, write 
    \[
    \nu^*\varphi^*(H)|_{C_v} = \sum_{i=1}^k a_i(p_i +q_i) + 2\sum_{j=1}^{k'} b_j r_j,
    \]
    where $p_i,q_i$ are pairwise distinct conjugate pairs and $r_j$ are distinct ramification points of the $2:1$ map $C_v\to \varphi(C_v)$. The term is equal to $\prod_{i=1}^{k}(a_i + 1)$, which is strictly smaller than $2^{d_v}$ if and only if $k'\neq 0$ or some $a_i\geq 2$. Notice a special case: if $\varphi(C_v)$ is a line, we never have $a_i \geq 2$.
    
    \item[(d)] If $v\in\widehat{V}^\nu_2$ and $\varphi(C_{v'}\cup C_{v''})\not\subseteq W_H$, write $\nu^*\varphi^*(H)|_{C_{v'}} = \sum a_ip_i'$. In particular $\nu^*\varphi^*(H)|_{C_{v''}} = \sum a_ip_i''$, where each $p_i''$ is the image of $p_i'$ via the involution $\iota \colon C_{v'}\to C_{v''}$. Then, we have to choose a divisor $D_{v'}\leq \sum a_ip_i'$ on $C_{v'}$ and this forces the choice of the divisor $D_{v''}=\iota_*(\sum a_ip'_i - D_{v'})$ on $C_{v''}$. This term equals the number of ways to write
    \[
    \sum_{i=1}^k x_i = d_{v'}, \text { for } 0\leq x_i\leq a_i.
    \]
    This term is strictly smaller than $\binom{\sum a_i}{d_{v'}} = \binom{d_{v'}+d_{v''}}{d_{v'}}$ if and only if there is at least one $a_i\geq 2$, and also $d_{v'}>0$ and $d_{v''}>0$. Notice that $d_{v'}+d_{v''}=k_{v'}=k_{v''}$, so we also have $d_{v'}<k_{v'}$ and $d_{v''}<k_{v''}$.
\end{itemize}

Summing up, $H \in\Br(\rho_{\ud, \U_{\ud}})$ if and only if either $H$ is tangent to some $W\in \Irr(\widehat{C})$ (and $W\not\subseteq H$) with $0 < d_Z < k_Z$, for every $Z\in \Irr(\varphi^{-1}(W))$, or $H$ contains a point $p\in \mathcal B_{\widehat{C}}$. Since the components of the canonical model of a nodal curve are nodal (as it follows from \cite[Remark~3.8]{catanese}), the first condition is equivalent to $H\in \mathcal \bigcup_{W\in \mathcal B_{\ud}}W^\vee$ by Proposition \ref{prop:dual}. The second condition is equivalent to $H\in \Lambda^\vee_p$, for $p\in \mathcal B_{\widehat{C}}$. This yields the formula for $ \Br(\rho_{\ud})$ in the statement. In particular, we see that $\Br(\rho_{\ud})= \Br^1(\rho_{\ud})$. Finally, every irreducible component of $\Br(\rho_{\ud})$ intersects $\U_0$, hence $\Br(\rho_{\ud})$ is the closure of $\Br(\rho_{\ud, \U_0})$. This concludes the proof.
\end{proof}

Recall the description of the restriction of $\Omega_{\ud}$ over $\U_1$ from Remark~\ref{rem:Omega_d} (ii). For $v\in V(\Gamma_C)$, we let $\Ram_{\ud,v}\subseteq \Omega_{\ud}$ be the closure in $\Omega_{\ud}$ of  
\begin{align}\label{eq:Ramv}
\Ram^0_{\ud,v} &= 
\left\{(D,H)\in \Omega_{\ud, \U_1}; 
\begin{array}{l} \text{there is} \; p\in (C^{\sm}\cap C_v)\setminus \varphi^{-1}(\mathcal B_{\widehat{C}}) \\ \text{such that } 1\le \mu_p(D) < \mu_p(\varphi^*(H))
\end{array}\right\}.
\end{align}
Given a point $p\in C^{\sm}$, we also define $\Ram_{\ud, p}\subseteq \Omega_{\ud}$ as the closure in $\Omega_{\ud}$ of
\begin{equation}\label{eq:Ramp}
\Ram^0_{\ud, p} := \{(D,H) \in \Omega_{\ud, \U_1}: 1\le \mu_p(D)\le \mu_p(\varphi^*(H))\}.
\end{equation}

The next proposition identifies explicit loci contained in $\Ram(\rho_{\ud})$. For the applications in Section \ref{sec:Torelli-theorem}, only the inclusion established here is needed, so we do not pursue the converse. Recall the definition of $\kappa_{\ud}$ in Equation \eqref{eqdef:kappa}.

\begin{proposition}\label{prop:rho-ramif}
Let $C$ be a connected stable curve of genus $g\ge2$ with no separating nodes. Let $\ud\in \Sigma(C)$. The following properties hold.
\begin{itemize}
    \item[(i)] 
      The ramification locus $\Ram(\rho_{\ud}) \subseteq\Omega_{\ud}$ of the map $\rho_{\ud}\colon \Omega_{\ud}\to \Proj_C^\vee$ satisfies 
    \[
     \left(\underset{0<d_v<k_v}{\bigcup_{v\in V(\Gamma_C)}}\Ram_{\ud,v}\right)\cup\left(\underset{\varphi(p)\in\mathcal B_{\widehat{C}}}{\bigcup_{p\in C}} \Ram_{\ud, p}\right)\subseteq \Ram(\rho_{\ud}).
    \]
    \item[(ii)] For every $v\in V(\Gamma_C)$ with $0<d_v<k_v$ let   $W_v=\varphi(C_v)$. If $\deg(W_v) >1$ then 
     $\codim_{S_{\ud}(C^\nu)}(\kappa_{\ud}(\Ram_{\ud,v}))=1$ and $\rho_{\ud}(\Ram_{\ud, v}) = W_v^{\vee}$.
    \item[(iii)] For every $p\in C$ with $\varphi(p)\in\mathcal B_{\widehat{C}}$, we have $\codim_{S_{\ud}(C^\nu)}(\kappa_{\ud}(\Ram_{\ud, p}))=1$ and $\rho_{\ud}(\Ram_{\ud, p}) = \Lambda_{\varphi(p)}^{\vee}$.
\end{itemize}
\end{proposition}

\begin{proof}
We prove (i). Since $\Ram(\rho_{\ud})$ is closed, it suffices to prove the inclusion 
    \[
     \left(\underset{0<d_v<k_v}{\bigcup_{v\in V(\Gamma_C)}}\Ram_{\ud,v}^0\right)\cup\left(\underset{\varphi(p)\in\mathcal B_{\widehat{C}}}{\bigcup_{p\in C}} \Ram^0_{\ud, p}\right)\subseteq \Ram(\rho_{\ud}).
    \]
Since $\mathcal U_1\subseteq\mathcal U_{\underline d}$,
Proposition~\ref{prop:Omega}(i) shows that the restriction
\(
\rho_{\underline d,\mathcal U_1}\colon
\Omega_{\underline d,\mathcal U_1}\to \mathcal U_1
\)
is finite and dominant, with integral Cohen--Macaulay source and smooth target. So we can use  Proposition~\ref{prop:branch_fiber_number}. We shall use that, after
base change along a morphism $\Spec(R)\to\mathcal U_1$ from the spectrum of a DVR, two
distinct points of the generic fiber cannot specialize to the same
unramified point.

We start by proving the inclusion $\Ram^0_{\ud, v} \subseteq \Ram(\rho_{\ud})$ for all $v \in V(\Gamma_C)$. Recall the decomposition $V(\Gamma_C) = V_0 \sqcup V_1 \sqcup V_2$ in Equation \eqref{eq:V-decomposition}. In what follows, given a hyperplane $H$ in $\Proj_C^\vee$, we will consider a family $\mathcal H\to \Spec(R)$ of hyperplanes, where $R$ is a DVR, with generic point $\eta$, whose fiber over $\eta$ is equal to a general hyperplane $H_\eta$ and whose fiber over the special point is equal to $H$.

First, fix $v\in V_0$ with $0<d_v<k_v$ and let $(D, H) \in \Ram_{\ud, v}^0$. Then there is a point $p_0 \in C^{\sm}\cap C_v$, and an integer $a_0\geq 1$ such that 
\[
\nu^*\varphi^*(H) = (a_0 + 1)p_0 + E,
\]
for some effective divisor $E$ on $C^\nu$, and such that $D = a_0p_0 + D'$, for some $D'$ satisfying $\mu_{p_0}(D') = 0$ (Note that we do not impose that $\mu_{p_0}(E)=0$). Then 
\[
\nu^*\varphi^*(H_\eta) = \sum_{i=1}^{a_0+1}p_{i,\eta} + E_\eta,
\]
where $p_{i,\eta}$ specializes to $p_0$, for every $i=1, \ldots, a_0+1$, and where $E_\eta$ specializes to $E$. For $j=1, \ldots, a_0+1$, we define
\[
D_{j,\eta} := \left(\sum_{i=1}^{a_0+1}p_{i,\eta}\right) - p_{j,\eta} +  D'_\eta,
\]
where $D'_{\eta}$ is any divisor specializing to $D'$ and satisfying (recall Remark \ref{rem:Omega_d} (ii)) 
\[
(D'_{\eta})_0 + (D'_{\eta})_{\h}+ \iota_{\h}^*((D'_{\eta})_{\h}) \leq E_\eta.
\] 
Therefore $(D_{j, \eta}, H_\eta)\in \Omega_{\ud}$ and $(D_{j,\eta}, H_\eta)$ specializes to $(D, H)$ for each $j=1,\ldots, a_0+1$. Since $a_0\geq 1$, at least two points  $(D_{j,\eta}, H_\eta)$ specialize to $(D,H)$, so $(D, H)\in \Ram(\rho_{\ud})$. 

Next fix $v\in V_1\sqcup  V_2$ with $0<d_v<k_v$, let
$(D,H)\in\Ram^0_{\ud,v}$, and choose $p_0$ as in Equation~\eqref{eq:Ramv}.  Set $q_0 := \iota_{\h}(p_0)\neq p_0$ (because $p_0\notin \varphi^{-1}(\mathcal{B}_{C})$). Then, by Remark~\ref{rem:Omega_d} (iii), there exists an integer $a_0\geq 1$ such that
\[
\nu^*\varphi^*(H) = (a_0+ 1)(p_0+q_0) + E,
\]
and such that $D = a_0p_0 + q_0 + D'$ with $\mu_{p_0}(D') = 0$. Then 
\[
\nu^*\varphi^*(H_\eta) = \sum_{i=1}^{a_0 + 1}(p_{i,\eta} + q_{i,\eta}) + E_\eta,
\]
where $\iota_{\h}(p_{i, \eta}) = q_{i, \eta}$,  $p_{i,\eta}$ specializes to $p_0$, and $q_{i, \eta}$ specializes to $q_0$, for every $i=1, \ldots, a_0+1$, and where $E_\eta$ specializes to $E$.  For $j=1, \ldots, a_0+1$, we define
\[
D_{j,\eta} := \left(\sum_{i=1}^{a_0+1}p_{i,\eta}\right) - p_{j,\eta} +  q_{j, \eta} + D'_\eta,
\]
where $D'_\eta$ is any divisor specializing to $D'$ and satisfying (recall Remark \ref{rem:Omega_d} (ii))
\[
(D'_{\eta})_0 + (D'_{\eta})_{\h} +\iota_{\h}^*((D'_{\eta})_{\h})\leq E_\eta.
\]
In particular, $(D_{j, \eta}, H_\eta)\in \Omega_{\ud}$ and $(D_{j,\eta}, H_\eta)$ specializes to $(D, H)$ for every $j=1,\ldots, a_0+1$. Since $a_0\geq 1$, at least two points  $(D_{j,\eta}, H_\eta)$ specialize to $(D,H)$, so $(D, H)\in \Ram(\rho_{\ud})$. 

We now show the inclusion $\Ram^0_{\ud, p}\subseteq \Ram(\rho_{\ud})$ whenever $\varphi(p)\in \mathcal{B}_{\widehat{C}}$. In this case, given $(D,H)\in \Ram^0_{\ud, p}$, we can write 
\[
\nu^*\varphi^*(H) = 2p + E,
\]
for some effective divisor $E$ on $C^\nu$,
and $D = p + D'$ for some $D'$. Since $(D,H)\in\Omega_{\ud}$, we have
\(
D'_0+D'_{\h}+\iota_{\h}^*(D'_{\h})\leq E
\). The generic hyperplane $H_{\eta}$ now satisfies
\[
\nu^*\varphi^*(H_{\eta}) = p_{\eta} + q_{\eta} + E_{\eta},
\]
where $q_\eta=\iota_{\h}(p_\eta)$ and $p_\eta,q_\eta$ specialize to $p$.
Choose a divisor $D'_\eta$ specializing to $D'$ and satisfying
\[
(D'_{\eta})_0+ (D'_{\eta})_{\h}+\iota_{\h}^*((D'_{\eta})_{\h})\leq E_\eta.
\]
Set $D_{1,\eta} := p_{\eta}+D'_{\eta}$ and $D_{2, \eta} := q_{\eta}+D'_{\eta}$. Both pairs $(D_{1,\eta}, H_{\eta})$, $(D_{2, \eta}, H_{\eta})$ belong to $\Omega_{\ud}$ and specialize to $(D, H)$, so $(D, H)\in \Ram(\rho_d)$. This ends the proof of (i).

Before proving (ii) and (iii), notice that $\rho_{\ud}$ is generically \'etale (since it is dominant, generically finite and  $\operatorname{char}(k)=0$). Hence $\Ram(\rho_{\ud})$ has positive codimension. By (i) and the
birationality of $\kappa_{\ud}$ (see Equation \eqref{eq:kappaV}), the same is true of the two loci inside $S_{\ud}(C^\nu)$ in (ii) and (iii). It remains to prove that their codimensions are at most $1$.

We prove (ii). If $g=2$ we have nothing to prove since $\deg W_v=1$ for every $v\in V(\Gamma_C)$. Assume that $g\ge 3$. The fact that $\rho_{\ud}(\Ram_{\ud, v}) = W_v^{\vee}$, where $W_v=\varphi(C_v)$ is clear. Fix $v\in V(\Gamma_C)$ with $0 < d_v < k_v$ and $\deg(W_v)>1$. Choose $v'$ as follows. If $v\in V_1$, set
$v'=v$. If $v\in V_2$, let $v'$ be the other vertex of the corresponding honest hyperelliptic pair. If $v\in V_0$ and $d_v\geq2$,
set $v'=v$; otherwise choose $v'\neq v$ with $d_{v'}>0$.

Let $\ud'$ be the multidegree obtained from $\ud$ by subtracting $1$ at the entries corresponding to $v$ and $v'$ (so $|\ud'|=g-3$). For a subcurve $Z\subseteq C$, set \[\dot{Z}:=(Z\cap C^{\sm})\setminus \varphi^{-1}(\mathcal B_{\widehat{C}}).\] Define  $F_1\subseteq \dot{C}_v\times S_{\ud'}(C^\nu)\times \Proj^\vee_C$ and $F_2\subseteq \dot{C}_v\times \dot{C}_{v'}\times S_{\ud'}(C^\nu)\times\Proj^\vee_C$ as 
\begin{align*}
    F_1 & := \left\{(p, D', H) : \begin{array}{l} \exists q\in \dot{C}_{v'} \text{ such that }(p+q+D',H)\in \Omega_{\ud, \U_1},\\ 
    \text{with }1\leq \mu_p(p+q+D') < \mu_p(\nu^*\varphi^*(H))
    \end{array}\right\}; \\ \\
    F_2 & := \left\{(p, q, D', H) : \begin{array}{l} (p+q+D', H)\in \Omega_{\ud, \U_1}, \text{ with}\\ 1\leq \mu_p(p+q+D') < \mu_p(\nu^*\varphi^*(H))
    \end{array}\right\}.
\end{align*}
We have the following commutative diagram
\[
\begin{tikzcd}
F_1 \ar[d, "\zeta_{1, S}"] & F_2 \ar[l, "{\zeta_{2,1}}"'] \ar[d, "{\zeta_{2, S}}"] \ar[r, "{\zeta_{2, R}}"] & \Ram_{\ud,v}\ar[d, "\kappa_{\ud}"] \\
\dot{C}_v\times S_{\ud'}(C^\nu)  & \ar[l, "\eta_S"'] \dot{C_v}\times \dot{C}_{v'}\times S_{\ud'}(C^\nu) \ar[r, "\pi_S"] & S_{\ud}(C^\nu)
\end{tikzcd}
\]
where
\begin{enumerate}
    \item the map $\zeta_{2, 1}$ is the  forgetful map $\zeta_{2, 1}(p, q, D', H) = (p, D', H)$;
    \item the map $\zeta_{2, R}$ is defined by $\zeta_{2, R}(p, q, D', H) = (p+q+D', H)$;
    
\item the map $\zeta_{1, S}$ is the forgetful map $\zeta_{1, S}(p, D', H) = (p, D')$;

\item the map $\zeta_{2,S}$ is the forgetful map $\zeta_{2, S}(p, q, D', H) = (p, q, D')$;

\item the map $\eta_S$ is the projection onto the first and third factor;
\item $\pi_S$ is defined by $\pi_S(p,q,D')=p+q+D'$.
\end{enumerate}
Notice that $\zeta_{2, 1}$ 
is a dominant map.

We claim that $\zeta_{1, S}$ is dominant. By the definition of $F_1$, it is enough to prove that, for a general pair $(p,D')\in \dot{C}_v\times  S_{\ud'}(C^\nu)$, there is $H\in \Proj^\vee_C$ such that $2p + D'_0 + D'_{\h}+\iota_{\h}^*(D'_{\h})\leq \nu^*\varphi^*(H)$ and, if $C_v$ is contained in an honest hyperelliptic subcurve of $C$, then $2p +\iota_{\h}^*(2p) + D'_0 + D'_{\h} + \iota_{\h}^*(D'_{\h})\leq \nu^*\varphi^*(H)$. This is equivalent to saying that there is $H\in \Proj^\vee_C$ tangent to $\widehat{C}$ at $\varphi(\nu(p))$ such that $\varphi(\nu(D'))\subseteq H$.   Since $\deg(D') = g-3$, there is a hyperplane $H\in \Proj^\vee_C$ containing $\varphi(\nu(D'))$ and the tangent line to $\widehat{C}$ at $\varphi(\nu(p))$. Finally, since $\deg(\nu^*\varphi^*(H)|_{C_v^{\nu}})=k_v> d_v$ and $\deg(\nu^*\varphi^*(H)|_{C^{\nu}_{v'}})=k_{v'}\geq d_{v'}$, there exists $q\in\dot C_{v'}$ such that
$p+q+D'\leq \nu^*\varphi^*(H)$.
Thus $(p,q,D',H)\in F_1$, proving the claim.

By the claim we get that $\zeta_{1,S}\circ \zeta_{2,1}$ is dominant, which means that $\eta_S\circ \zeta_{2, S}$ is dominant. In particular, $\codim_{C_v\times C_{v'}\times S_{\ud'}(C^\nu)}(\zeta_{2,S}(F_2)) \leq 1$. Since the map $\pi_S$ is dominant and quasi-finite, we have $\codim_{S_{\ud}(C^\nu)}(\pi_S(\zeta_{2, S}(F_2))) \leq 1$.
Because the above diagram is commutative, we deduce 
\[
(\pi_S\circ \zeta_{2,S})(F_2) = (\kappa_{\ud}\circ\zeta_{2,R})(F_2)\subseteq \kappa_{\ud}(\Ram_{\ud,v}),
\]
which concludes the proof of (ii).

We prove (iii). The fact that $\rho_{\ud}(\Ram_{\ud, p}) = \Lambda_{\varphi(p)}^{\vee}$ is again clear. Fix $p\in C$ such that $\varphi(p)\in \mathcal B_{\widehat{C}}$. Let $v\in V(\Gamma_C)$ be such that $p\in C_v$.  Let $\ud'$ be the multidegree obtained from $\ud$ by subtracting $1$ at the entry corresponding to $v$ (so $|\ud'|=g-2$). Define $F\subseteq S_{\ud'}(C^\nu)\times \Proj^\vee_C$ as
\[
F := \{(D', H);\ (p + D', H)\in \Omega_{\ud, \U_1}\}.
\]
Since $\deg(D')= g-2$, for a general $D'$, there always exists a hyperplane containing $\varphi(p)$ and  $\varphi(\nu(D'))$. Hence the following map is dominant
\[
\zeta_S\colon F\to S_{\ud'}(C^\nu)\times \Proj^\vee_C\to S_{\ud'}(C^\nu),
\]
where the first map is the natural inclusion and the second is the projection onto the first factor. Also define the map $\zeta_R\colon F\to \Ram_{\ud,p}$ by $\zeta_R(D',H)=(p+D', H)$. We have the following commutative diagram
\[
\begin{tikzcd}
F \ar[d, "\zeta_S"] \ar[r, "\zeta_R"]  & \Ram_{\ud,p} \ar[d, "\kappa_{\ud}"] \\
 S_{\ud'}(C^\nu) \ar[r, "s_p"] & S_{\ud}(C^\nu)
\end{tikzcd}
\]
where $s_p\colon S_{\ud'}(C^\nu)\to S_{\ud}(C^\nu)$ is defined by $s_p(D')=p + D'$. Because the above diagram is commutative, we have 
\[
(s_p\circ \zeta_S)(F) = (\kappa_{\ud}\circ \zeta_R)(F) \subseteq \kappa_{\ud}(\Ram_{\ud,p}).
\]
Since $\zeta_S$ is dominant, we obtain
\[
\codim_{S_{\ud}(C^\nu)}(s_p\circ \zeta_S(F)) = \codim_{S_{\ud}(C^\nu)}(s_p(S_{\ud'}(C^\nu))) = 1,
\]
which concludes the proof of (iii).
\end{proof}

 We now illustrate the ramification locus of $\rho_{\ud}$ in a simple example.
 
\begin{example}\label{ex:acgh}
    Consider a smooth non-hyperelliptic curve $C$ of genus $g\geq 3$. Then the canonical map $\varphi \colon C\to\widehat{C} \subseteq \mathbb{P}_C$ is an isomorphism.
    
    Let $H \subseteq \mathbb{P}_C$ be a general tangent hyperplane to $\widehat{C}$, and write $\varphi^*(H) = 2p + q_1 + \dots +q_{2g-4}$. Let $H_t$ be a one-parameter family of hyperplanes limiting to $H$ for $t \to 0$, with $\varphi^*(H_t) = p'_t+p''_t + q_{t,1} + \dots + q_{t,2g-4}$. The number of effective divisors $D_t$ of degree $g-1$ such that $D_t\leq \varphi^*(H_t)$ is $\binom{2g-2}{g-1}$. On the other hand, the number of effective divisors $D$ of degree $g-1$ such that $D\leq\varphi^*(H)$ is $\binom{2g-4}{g-3} + \binom{2g-4}{g-2} + \binom{2g-4}{g-1}$. More precisely, these effective divisors are 
    \[
    2p + \sum_{i\in S_1}q_{i}, \;\;\;\;\;  p+\sum_{i\in S_2}q_{i}, \;\;\;\;\;  \sum_{i\in S_3}q_{i}
    \]
    where $S_1, S_2, S_3\subseteq \{1, \ldots, 2g-4\}$ are any subsets with $|S_1| = g-3$, $|S_2| = g-2$, and $|S_3| = g-1$. It is clear that
   \[
   \lim_{t\to 0}\Bigg(p'_t+\sum_{i\in S_2}q_{t,i}\Bigg) = p +  \sum_{i\in S_2}q_i = \lim_{t\to 0} \Bigg(p''_t+\sum_{i\in S_2}q_{t, i}\Bigg).
   \]
   This means that $(p + \sum_{i\in S_2}q_i, H)$ is a ramification point  for every $S_2\subseteq\{1, \ldots, 2g-4\}$ such that $|S_2|=g-2$, while the points $(2p + \sum_{i\in S_1}q_i, H)$ and $(\sum_{i\in S_3}q_i, H)$ are unramified. 
\end{example}

\section{A Torelli theorem via canonical triples}\label{sec:Torelli-theorem}

\subsection{Recovering components and distinguished points}

In this section, we let $C$  be a connected stable curve of genus $g\ge2$ with no separating nodes. We denote by $\nu\colon C^{\nu} \to C$ the normalization, and by $\varphi\colon C\to \mathbb P_C$ the canonical map. 

Fix  $\ud\in \Sigma(C)$, and recall the definitions  of $\Upsilon_{\ud}$ from Section~\ref{sub: gauss} and of $\Omega_{\ud}$ and $\U_0 \subseteq \mathbb{P}_C^{\vee}$ from Section~\ref{subsec:gaussgraph}. In this section we compare the auxiliary Gauss graph $\Omega_{\ud}$ and the intrinsic Gauss component $\Upsilon_{\ud}$. This transfers the branch computation of Section~\ref{sec:branch-ramif} to $\Upsilon_{\ud}$ and it also allows us to recover the set $\mathcal{S}_{\widehat{C}}$.

\begin{notation}\label{not:partial-normalization}
Let $C$ be a nodal curve. For $S\subseteq C^{\sing}$, let
$C_S\to C$ be the partial normalization at $S$, and let
$\mu_S\colon C^\nu\to C_S$ be the induced map. A divisor $D$ on
$C^\nu$ supported on $\nu^{-1}(C^{\sm})$ will also be regarded,
via $\mu_S$, as a divisor on $C_S$; we write
$\mathcal O_{C_S}(D)$ accordingly.
\end{notation}

Consider $(D,H) \in \Omega_{\ud,\U_0}$. By  Remark~\ref{rem:Omega_d} we have 
\[
\nu^*\varphi^*(H)\subseteq \nu^{-1}(C^{\sm})
\; \text{ and } \; 
D_0+D_{\h}+\iota^*_{\h}(D_{\h})\leq \nu^*\varphi^*(H).
\]
In particular, $D$ is supported on $\nu^{-1}(C^{\sm})$, so $D$  is an effective Cartier divisor on $C$ supported on $C^{\sm}$ and $\mathcal{O}_{C}(D)\in \Theta_{\ud}$. This allows us to define a $\U_0$-morphism
\begin{equation}\label{eq:betaU0_def}
\begin{split}
   \beta_{\ud,\U_0}\colon \Omega_{\ud,\U_0} & \to \Theta_{\ud}\times \U_0 \\
      (D, H) &\mapsto ([\mathcal{O}_{C}(D)], H).
\end{split}
\end{equation}

We can extend this morphism on a larger open subset as follows. 

\begin{definition}\label{def:T}
   For each $p\in C^{\sing}$, set $\{p',p''\} := \nu^{-1}(p)$. We let $\mathbb T_{\ud}\subseteq S_{\ud}(C^\nu)$ be the open subset of divisors $D$ satisfying the following properties
\begin{enumerate}
\item for each $p\in C^{\sing}$, we have $\mu_{p'}(D) + \mu_{p''}(D)\leq 1$.

\item if $S_D\subseteq C^{\sing}$ denotes the set of $p\in C^{\sing}$ such that
$\mu_{p'}(D)+\mu_{p''}(D)=1$, and we set
$D^{\rm sm}:=D\cap \nu^{-1}(C^{\rm sm})$, then
$\mathcal O_{C_{S_D}}(D^{\rm sm})$ is stable on $C_{S_D}$.
\end{enumerate}
\end{definition}

By the definition of $\mathbb T_{\ud}$, there is a morphism $\alpha_{\ud}\colon \mathbb T_{\ud}\to \Pic$ taking $D$  to $(S_D, [\mathcal{O}_{C_{S_D}}(D^{\sm})])$. We define the open subset $\Omega_{\mathbb T_{\ud}}$ of $\Omega_{\ud}$ as
    \[
    \Omega_{\mathbb T_{\ud}} = \{(D, H)\in \Omega_{\ud} :  D\in \mathbb T_{\ud}\}\subseteq \Omega_{\ud}.
    \]
We consider the natural morphism to the product
\begin{equation}\label{eq:beta_def}
\begin{aligned}
\beta_{\mathbb T_{\ud}}\colon \Omega_{\mathbb T_{\ud}} & \to \Pic\times \Proj_C^\vee \\
(D, H) & \mapsto (\alpha_{\ud}(D), H).
\end{aligned}.
\end{equation}

Recall $\gamma_{\ud}\colon \Upsilon_{\ud}\to \Proj^\vee_C$ and  $\rho_{\ud}\colon \Omega_{\ud}\to \Proj^\vee_C$ from Equations \eqref{eq:upsilon-d} and~\eqref{eqdef:kappa}.
    
\begin{lemma}\label{lem:factorize}
Let $C$ be a connected stable curve of genus $g\ge2$ with no separating nodes. Let $\ud\in \Sigma(C)$. The following properties hold.
    \begin{enumerate}
        \item[(i)] The morphism $\beta_{\ud,\U_0}$ factors through $\Upsilon_{\ud,\U_0}$, yielding  a finite birational $\U_0$-morphism $\beta_{\ud,\U_0}\colon \Omega_{\ud,\U_0}\to \Upsilon_{\ud,\U_0}$.
        \item[(ii)] We have an inclusion $\Omega_{\ud, \U_0}\subseteq \Omega_{\mathbb T_{\ud}}$ and $\beta_{\mathbb T_{\ud}}|_{\Omega_{\ud, \U_0}} = \beta_{\ud, \U_0}$.
        \item[(iii)] The morphism $\beta _{\mathbb T_{\ud}}$ factors through $\Upsilon_{\ud}$.
    \end{enumerate}
\end{lemma}

Since $\Omega_{\ud}$ is irreducible by Proposition~\ref{prop:Omega}(ii), we obtain a birational map $\beta_{\ud} \colon \Omega_{\ud} \dashrightarrow \Upsilon_{\ud}$.

\begin{proof}
Let $V\subseteq \Omega_{\ud,\U_0}$ be the open subset consisting of pairs $(D,H)$
such that $h^0(C,\mathcal O_C(D))=1$. The general element of
$\Theta_{\ud}$ has exactly one section by \cite[Proposition 3.2.1]{captheta};
hence $V$ is dense in $\Omega_{\ud,\U_0}$.

Take $(D,H)\in V$ and set $I:=\mathcal O_C(D)$. Let
$\ell_H\subseteq H^0(C,\omega_C)$ be the line  corresponding to
$H\in\mathbb P_C^\vee$. Since $(D,H)\in\Omega_{d,\U_0}$, the sections in
$\ell_H$ vanish along $\nu(D)$; hence
\[
\ell_H\subseteq
H^0(C,\omega_C\otimes I^{-1})
=
H^0(C,\omega_C(-D)).
\]
Since $h^0(C,I)=1$, Riemann--Roch gives
$h^0(C,\omega_C\otimes I^{-1})=1$. Therefore
\[
H=
\mathbb P\bigl(H^0(C,\omega_C\otimes I^{-1})\bigr)
=
\mathbb P\bigl(H^0(C,\omega_C(-D))\bigr)
\in \mathbb P_C^\vee.
\]
Equivalently, the hyperplane of
$\mathbb P_C=\mathbb P(H^0(C,\omega_C)^\vee)$ corresponding to $H$ is
\[
\mathbb P\bigl(H^0(C,\omega_C(-D))^\perp\bigr)\subseteq \mathbb P_C.
\]
By Lemma~\ref{lem:tangent_L_theta}, this hyperplane is $\mathbb P(T_0u^{-1}\Theta(C))$, where $u=[I]$, so $([I],H)\in\Upsilon_{d,\U_0}$. Hence, since $V$ is dense in $\Omega_{\ud,\U_0}$, the morphism $\beta_{\ud,\U_0}$ factors through $\Upsilon_{\ud,\U_0}$. 

 As noted at the beginning of the proof, the general element of $\Theta_{\ud}$ has only one section, hence the morphism $\beta_{\ud,\U_0}$ is birational. We prove that $\beta_{\ud,\U_0}$ is finite by showing it is proper and quasi-finite. Since $\gamma_{\ud,\U_0}$ and $\rho_{\ud, \U_0}$ are proper, $\beta_{\ud,\U_0}$ is also proper. Moreover, $\beta_{\ud, \U_0}$ is quasi-finite, because it is a $\U_0$-morphism, so every fiber of $\beta_{\ud, \U_0}$ is contained in a fiber of $\rho_{\ud, \U_0}$, and the latter is finite by  Proposition~\ref{prop:Omega}(i). This ends the proof of (i).

Notice that $\Omega_{\ud,\U_0}\subseteq \Omega_{\mathbb{T}_{\ud}}$. Indeed, if
$(D,H)\in\Omega_{\ud,\U_0}$, then $\supp(D)\subseteq\nu^{-1}(C^{\rm sm})$, so
$S_D=\emptyset$ and $\mathcal O_C(D)$ has multidegree $\ud\in\Sigma(C)$. Hence $D\in \mathbb{T}_{\ud}$. Moreover, by Equations \eqref{eq:betaU0_def} and \eqref{eq:beta_def}, we have
$\beta_{\mathbb{T}_{\ud}}(D,H)=\beta_{d,U_0}(D,H)$. This proves (ii).

To prove (iii), it is enough to notice that $\Omega_{\ud}$ is irreducible, so the image of $\beta_{\mathbb T_{\ud}}$ lies in the closure of the image of $\beta_{\ud, \U_0}$, which is contained in $\Upsilon_{\ud}$. 
\end{proof}

Next, given $\ud\in \Sigma(C)$, we consider the normalization $\varphi\colon\widehat{\Upsilon}_{\ud}\to\Upsilon_{\ud}$ of $\Upsilon_{\ud}$. We let 
\[
\widehat{\gamma}_{\ud}\colon\widehat{\Upsilon}_{\ud}\stackrel{\varphi}{\lto}\Upsilon_{\ud}\stackrel{\gamma_{\ud}}{\lto} \Proj^\vee_C
\]
be the composition. Then $\widehat{\gamma}_{\ud}$ is generically finite and dominant, since so is $\gamma_{\ud}$. 

\begin{proposition}\label{prop:equal}
    Let $C$ be a connected stable curve of genus $g\ge2$ with no separating nodes. Let $\ud\in \Sigma(C)$. Then $\Br(\widehat{\gamma}_{\ud})=\Br(\rho_{\ud})$ (see Equation \eqref{eqdef:kappa}). 
    \end{proposition}
    
\begin{proof}
The proof of the inclusion $\Br(\widehat{\gamma}_{\ud})\subseteq\Br(\rho_{\ud})$ relies on the following

{\bf Claim:} There exists an open subset $\U^\dagger \subseteq \mathbb{P}^{\vee}_C$ 
 such that:
\begin{enumerate}
   \item[(1)] The inequality $\codim_{\Proj^\vee_C}(\Proj^\vee_C\setminus \U^\dagger)\geq 2$ holds.
   \item[(2)] The rational map $\beta_{\ud}$ in Lemma \ref{lem:factorize} restricts to a well-defined $\U^\dagger$-morphism $\Omega_{\ud, \U^\dagger}\to \Upsilon_{\ud, \U^\dagger}$ that, by abusing notation, we still call  $\beta_{\ud}$.
   \item[(3)] The morphism $\beta_{\ud}\colon\Omega_{\ud, \U^\dagger}\to \Upsilon_{\ud, \U^\dagger}$ is finite birational. 
   \item[(4)] The scheme $\widehat{\Upsilon}_{\ud, \U^\dagger}$ is smooth.
\end{enumerate}

Let us first prove the claim. By Lemma \ref{lem:factorize} (i) we have a rational map $\beta_{\ud}\colon\Omega_{\ud}\dasharrow\Upsilon_{\ud}$ such that $\beta_{\ud}$ restricts to the morphism $\beta_{\ud,\U_0}$ defined in Equation \eqref{eq:betaU0_def} over $\Omega_{\ud, \U_0}$. Let $\mathcal{D}_{\ud}\subseteq\Omega_{\ud}$ be the open subset where $\beta_{\ud}$ is a morphism, so $\Omega_{\ud, \U_0}\subseteq \mathcal{D}_{\ud}$. 

Recall the definition of the open subset $\U_1\subset \Proj_C^\vee$ from Section~\ref{sec:gauss-graph}. Also recall that we have an  inclusion $\U_0 \subseteq \U_1$, and that the morphism  $\rho_{\ud,\U_1}\colon \Omega_{\ud, \U_1}\to \U_1$ is finite and dominant (by Proposition~\ref{prop:Omega}(i)). Then define 
\begin{align*}
\mathcal Y &:= \left(\bigcup_{W\in \Irr_{\geq2}(\widehat{C})} W^\vee\right)\cup \left(\bigcup_{p\in \mathcal{B}_{\widehat{C}}} \Lambda^\vee_{p}\right),\\
\U' &:= \U_0\cup (\U_1\setminus \mathcal Y),\\
\mathcal F&:= \Omega_{\ud,\U'}\setminus \Omega_{\ud,\U_0}.
\end{align*}

First, $\U'$ is an open subset of $\Proj^\vee_C$ with $\U'\subseteq\U_1$. 
Also, $\Proj_C^{\vee}\setminus \mathcal U_0 = \bigcup_{p\in \mathcal{S}_{\widehat{C}}}\Lambda_p^\vee$ and $(\U_1\setminus \mathcal{Y})\cap \Lambda_p^{\vee} \neq \emptyset$ for every $p\in \mathcal{S}_{\widehat{C}}$, therefore $\codim_{\Proj^\vee_C}(\Proj^\vee_C\setminus\U') \geq 2$.

Second, $\mathcal{F}$ is pure of dimension $g-2$. Indeed, this follows from the fact that  $\rho_{\ud, \U'}\colon \Omega_{\ud, \U'}\to \U'$ is finite and dominant, that $\mathcal F=\rho_{\ud, \U'}^{-1}(\U'\setminus \U_0)$, and that 
\[
\U'\setminus \U_0 = \bigcup_{p \in \mathcal{S}_{\widehat{C}}}(\Lambda_p^\vee \cap (\U_1\setminus \mathcal{Y}))
\]
is pure of dimension $g-2$. Moreover, since $\U'\subseteq \U_1$, we have that $\mathcal F\subseteq \Omega_{\ud, \U_1}$.

Denote by $\Omega_{\ud}^{\sm}$ the smooth locus of $\Omega_{\ud}$. By Proposition~\ref{prop:branch_fiber_number},
$\rho_{\ud,\U^1}$ is flat. By Proposition~\ref{prop:rho-branch},
it is unramified along $\F$, hence étale along $\F$. Since $\U^1$
is smooth, we have $\F\subseteq\Omega_{\ud}^{\rm sm}$.
Because the indeterminacy locus of a rational map on a smooth source has codimension at least two, we have $\codim_{\Omega_{\ud}^{\sm}}(\Omega_{\ud}^{\sm}\setminus \mathcal{D}_{\ud}) \geq 2$. Thus every irreducible component of $\mathcal F$ has non-empty intersection with $\mathcal{D}_{\ud}$, so $\codim_{\Omega_{\ud, \U'}}(\Omega_{\ud, \U'}\setminus \mathcal{D}_{\ud})\geq 2$. Since $\rho_{\ud,\U'}\colon\Omega_{\ud, \U'}\to \U'$ is finite and dominant, we deduce that  $\codim_{\U'}(\rho_{\ud}(\Omega_{\ud, \U'}\setminus \mathcal{D}_{\ud})) \geq 2$. 

We now define the open set $\U'' := \U'\setminus \rho_{\ud}(\Omega_{\ud, \U'}\setminus \mathcal{D}_{\ud})$. Because  $\Omega_{\ud, \U''}$ is a subset of $\mathcal{D}_{\ud}$, the rational map $\beta_{\ud}$ is a well-defined morphism $\Omega_{\ud, \U''}\to \Upsilon_{\ud, \U''}$. 
It is clear that $\Omega_{\ud, \U''}\to \Upsilon_{\ud, \U''}$ is finite

and by Lemma \ref{lem:factorize} (i) it is also birational. Finally, set $\U^{\dagger} := \U''\setminus \widehat{\gamma}_{\ud}(\widehat{\Upsilon}^{\sing}_{\ud})$. Since $\widehat{\Upsilon}_{\ud}$ is normal, we have that $\dim \widehat{\Upsilon}^{\sing}_{\ud} \leq g-3$, hence $\dim (\widehat{\gamma}_{\ud}(\widehat{\Upsilon}^{\sing}_{\ud})) \leq g-3$ and therefore $\codim_{\Proj_C^{\vee}}(\Proj_C^{\vee}\setminus \U^{\dagger})\geq 2$. The latter statement proves that Item (1) holds, and since we are removing the image of the singular locus, it is clear that Item (4) holds. Items (2) and (3) follow from the fact that $\U^{\dagger}\subseteq \U''$ and they hold for $\U''$. This concludes the proof of the Claim.

We now prove the inclusion $\Br(\widehat{\gamma}_{\ud})\subseteq \Br(\rho_{\ud})$. By Theorem~\ref{thm:purity} and Proposition~\ref{prop:rho-branch}, both $\Br(\widehat{\gamma}_{\ud})$ and $\Br(\rho_{\ud})$ are pure of codimension~$1$. Hence, by Item~(1) of the Claim, it is enough to prove the inclusion after restricting to $\U^\dagger$, since $\mathbb{P}_C^\vee\setminus \U^\dagger$ has codimension at least $2$. By Items~(2) and (3) of the Claim and the universal property of the normalization, there is a finite birational morphism $\Psi\colon \widehat{\Upsilon}_{\ud,\U^\dagger}\to \Omega_{\ud,\U^\dagger}$; hence $\widehat{\gamma}_{\ud,U^\dagger}
=\rho_{\ud,U^\dagger}\circ\Psi$ is finite and dominant, and
$\deg(\widehat{\gamma}_{\ud})=\deg(\rho_{\ud})$. Now let $H\in \U^\dagger\setminus \Br(\rho_{\ud})$. Since $\Omega_{\ud,\U^\dagger}$ is Cohen--Macaulay by Proposition~\ref{prop:Omega}~(ii), and since $\mathbb{P}_C^\vee$ is smooth, Proposition~\ref{prop:branch_fiber_number} applied to $\rho_{\ud,\U^\dagger}$ gives $|\rho_{\ud}^{-1}(H)|=\deg(\rho_{\ud})$. Moreover, $\Psi$ is surjective on fibers, so $|\widehat{\gamma}_{\ud}^{-1}(H)|\ge |\rho_{\ud}^{-1}(H)|$. On the other hand, by Item~(4) of the Claim, $\widehat{\Upsilon}_{\ud,\U^\dagger}$ is smooth, hence Cohen--Macaulay, and therefore Proposition~\ref{prop:branch_fiber_number} also applies to $\widehat{\gamma}_{\ud,\U^\dagger}$, yielding $|\widehat{\gamma}_{\ud}^{-1}(H)|\le \deg(\widehat{\gamma}_{\ud})$. Since $\deg(\widehat{\gamma}_{\ud})=\deg(\rho_{\ud})$, we conclude that $|\widehat{\gamma}_{\ud}^{-1}(H)|=\deg(\widehat{\gamma}_{\ud})$, and thus $H\notin \Br(\widehat{\gamma}_{\ud})$. This proves the inclusion of branch loci over $\U^\dagger$, and hence globally.

We now prove the opposite inclusion
$\Br(\widehat{\gamma}_{\ud})\supseteq \Br(\rho_{\ud})$.
By the last assertion of Proposition~\ref{prop:rho-branch},
$\Br(\rho_{\ud})$ is the closure of $\Br(\rho_{\ud,\U_0})$ in
$\Proj_C^\vee$. Since $\Br(\widehat{\gamma}_{\ud})$ contains the
closure of $\Br(\widehat{\gamma}_{\ud,\U_0})$, it is enough to prove the inclusion $\Br(\rho_{\ud,\U_0})
\subseteq
\Br(\widehat{\gamma}_{\ud,\U_0})$.

Recall the evaluation map $e_{\ud}$ from Equation~\eqref{eq: evaluation}, and the open subset
\[
\V_{\ud}:=\{D\in S_{\ud}(C^\nu):  \operatorname{rk}(e_{\ud,D})=g-1\} \subseteq  S_{\ud}(C^\nu) \] introduced in the proof of Proposition~\ref{prop:Omega}, and set
\[S_{\ud,0}:=\kappa_{\ud}(\Omega_{\ud,U_0}),\qquad
\V_{\ud,0}:=\V_{\ud}\cap S_{\ud,0}.
\]
By the proof of Proposition~\ref{prop:Omega}(ii),
$\kappa_{\ud}^{-1}(\V_{\ud,0})\cong\V_{\ud,0}$ is a nonempty smooth
open subset of $\Omega_{\ud,U_0}$.

Set $S_{\ud,0}^1:=S_{\ud,0}\setminus\V_{\ud,0}$.
Its inverse image in $\Omega_{\ud,U_0}$ is a proper closed subset,
hence has dimension at most $g-2$. Since every fiber over
$S_{\ud,0}^1$ has positive dimension, $\dim S_{\ud,0}^1\leq g-3$.
As $\dim S_{\ud,0}=g-1$, we obtain
\begin{equation}\label{eq:cod2}
\codim_{S_{\ud,0}}S_{\ud,0}^1\geq2.
\end{equation}

By
Lemma~\ref{lem:factorize}(i) and the universal property of the
normalization, there is a finite birational morphism
\[
\Psi_0\colon
\widehat{\Upsilon}_{\ud,\U_0}
\longrightarrow
\Omega_{\ud,\U_0}.
\]
Since $\kappa_{\ud}^{-1}(\V_{\ud,0})$ is smooth, and hence normal,
$\Psi_0$ restricts to an isomorphism
\begin{equation}\label{eq:phi-restriction}
\Psi_0^{-1}\bigl(\kappa_{\ud}^{-1}(\V_{\ud, 0})\bigr)
\stackrel{\cong}{\longrightarrow}
\kappa_{\ud}^{-1}(\V_{\ud,0}).
\end{equation}

Now let $v\in V(\Gamma_C)$ be such that $0<d_v<k_v$ and $W=\varphi(C_v)$ has degree at least $2$, and let $p\in C$ be such that $\varphi(p)\in\mathcal B_{\widehat{C}}$. By Proposition~\ref{prop:rho-ramif}(ii) and (iii), we may choose irreducible components $R_v\subseteq\Ram_{\ud,v}$ and $R_p\subseteq\Ram_{\ud,p}$ whose images under $\kappa_{\ud}$ have codimension $1$ in $S_{\ud}(C^\nu)$. Since $\rho_{\ud,\U_1}$ is finite, these components dominate $W^\vee$ and $\Lambda_{\varphi(p)}^\vee$, respectively. Hence their intersections with $\Omega_{\ud,\U_0}$ are dense, and Equation~\eqref{eq:cod2} implies that $R_v\cap\kappa_{\ud}^{-1}(\V_{\ud, 0})$ and $R_p\cap\kappa_{\ud}^{-1}(\V_{\ud, 0})$ are nonempty, and thus dense in $R_v$ and $R_p$.

By Proposition~\ref{prop:rho-ramif}(i) and Equation~\eqref{eq:phi-restriction}, their inverse images under $\Psi_0$ are contained in the ramification locus of $\widehat{\gamma}_{\ud,\U_0}$. Their images are dense in $W^\vee\cap\U_0$ and $\Lambda_{\varphi(p)}^\vee\cap\U_0$, respectively; since the branch locus is closed in $\U_0$, we obtain the inclusions $W^\vee\cap\U_0\subseteq\Br(\widehat{\gamma}_{\ud,\U_0})$ and $\Lambda_{\varphi(p)}^\vee\cap\U_0\subseteq\Br(\widehat{\gamma}_{\ud,\U_0})$. As $v$ and $p$ vary, Proposition~\ref{prop:rho-branch} gives $\Br(\rho_{\ud,\U_0})\subseteq\Br(\widehat{\gamma}_{\ud,\U_0})$, concluding the proof.
\end{proof}

By combining the previous proposition with earlier results, we obtain:

\begin{theorem}\label{thm:branch-gamma}
     Let $C$ be a connected nodal curve of genus $g\ge2$ with no separating nodes. Let $\ud\in\Sigma(C)$. The branch locus $\Br(\widehat{\gamma}_{\ud})$ of the map $\widehat{\gamma}_{\ud}\colon \widehat{\Upsilon}_{\ud}\to \Proj_C^\vee$ is 
    \[
    \Br(\widehat{\gamma}_{\ud}) = \left(\bigcup_{W\in \mathcal B_{\ud}} W^\vee\right) 
    \cup\left(\bigcup_{p\in \mathcal B_{\widehat{C}}}\Lambda_p^\vee\right),
    \]
    where
    \begin{equation} \label{eq:Bd}
\mathcal B_{\ud}=\{W\in \Irr_{\ge2}(\widehat{C}) : 0<d_Z<k_Z, \text{ for every } Z\in \Irr_W(C)\}.
\end{equation}
Moreover, if $\ud\in\Sigma^{\qE}(C)$, then 
\[    \Br(\widehat{\gamma}_{\ud})= \Br(\widehat{\gamma})=\left(\bigcup_{W\in \Irr_{\ge2}(\widehat{C})} 
    W^\vee \right)\cup\left(\bigcup_{p\in\mathcal B_{\widehat C}}\Lambda_p^\vee\right).
    \]    
\end{theorem}

\begin{proof}
     We can assume that $C$ is stable. Indeed, $A(C)$ is naturally isomorphic to $A(C^{\st})$, the set $\Sigma(C)$ is naturally identified with $\Sigma(C^{\st})$, we have a natural isomorphism $\widehat{C}\cong \widehat{C}^{\st}$, and $\mathcal{B}_{\widehat{C}}=\mathcal{B}_{\widehat{C}^{\st}}$. Moreover, under these identifications, $\Irr_W(C)$ is naturally
identified with $\Irr_W(C^{\rm st})$ for every
$W\in\Irr(\widehat C)$. The first part follows immediately from Propositions \ref{prop:rho-branch} and \ref{prop:equal}.
     
     Next, by Lemma \ref{lem:linear-classification}, for $Z\in \Irr(C)$ with $g_Z=0$ and $\delta_Z=3$ we find that $\varphi(Z)$ is a linear component of $\widehat{C}$. Hence, according to Lemma \ref{lem:numerical}, 
     for every $\ud\in\Sigma^{\qE}(C)$ and every component $W\in \Irr_{\ge2}(\widehat{C})$ we have $0<d_Z<k_Z$ for each $Z\in \Irr_W(C)$. Therefore, the second part  is derived again from Propositions \ref{prop:rho-branch} and \ref{prop:equal}.
\end{proof}

\begin{remark}\label{rem:degree-Gauss}
    Keep the same notation as in the statement of Theorem \ref{thm:branch-gamma}. Since there is a birational map from $\Omega_{\ud}$ to $\Theta_{\ud}$ commuting with $\rho_{\ud}\colon \Omega_{\ud}\to \Proj_C^\vee$ and the Gauss map $\mathcal G_{\Theta_{\ud}}\colon\Theta_{\ud}\dashrightarrow \Proj_C^\vee$, Equation \eqref{eq:degree} shows that, for every $\ud\in \Sigma(C)$, the Gauss map $\mathcal G_{\Theta_{\ud}}$ is  generically finite  of degree 
    \[
    \deg(\mathcal G_{\Theta_{\ud}})=\prod_{v\in V_0} \binom{k_v}{d_v} \prod_{v\in V_1} 2^{d_v} \prod_{v\in \widehat{V}^\nu_2}
\binom{d_{v'}+d_{v''}}{d_{v'}}.
    \]

\end{remark}

\begin{corollary} The restriction of $\mathcal{G}$ to every irreducible component of $\Theta(C)$
is birational if and only if $V_1=\emptyset$ and every irreducible
component of $\widehat C$ is linear.
\end{corollary}
\begin{proof}
    For a fixed $\ud\in\Sigma(C)$, the degree formula shows that
$\mathcal{G}_{\Theta_{\ud}}$ is birational if and only if
\[
V_1=\emptyset,\qquad
\ud_v\in\{0,k_v\}\ \text{for every }v\in V_0,\qquad
\ud_{v'}\ud_{v''}=0\ \text{for every }v\in\widehat V_2^\nu.
\]
If every component of $\widehat C$ is linear and $V_1=\emptyset$,
then Lemmas~\ref{lem:linear-classification} and~\ref{lem:dimdeg} show
that all these conditions hold for every $\ud\in\Sigma(C)$.

Conversely, assume that every componentwise Gauss map is birational. Choose $\ud\in\Sigma^{\qE}(C)$, whose existence
follows from Proposition~\ref{prop:q_orient}. If $v\in V_1$, then Lemma~\ref{lem:dimdeg}(iv) gives
$\ud_v=k_v/2\geq1$, so the factor $2^{\ud_v}$ is greater than
one, a contradiction. Thus $V_1=\emptyset$. 

If $W\in\Irr(\widehat C)$ is nonlinear, then Lemmas
\ref{lem:linear-classification} and \ref{lem:numerical} give
\[
0<\ud_Z<k_Z
\qquad
\text{for every }Z\in\Irr_W(C).
\]
The factor of the degree formula corresponding to $W$ is therefore
greater than one, again a contradiction. Hence every component of
$\widehat C$ is linear.
\end{proof}

\begin{remark}\label{rem:contra-example-Andreotti}
As mentioned in the introduction, Andreotti's computation of the branch locus of the Gauss map relies on \cite[Lemma 3]{andreotti}, which is not correct already in dimension $d=2$. Here we give a counterexample to the lemma. 

Let $Z$ be a projective integral scheme of dimension $d$ and $h\colon Z\dashrightarrow\mathbb P^d$ be a  dominant rational map and set $n:=[k(Z):k(\mathbb P^d)]$. Let  $U \subseteq Z$ be the open subset where $h$ is a morphism.  Following Andreotti, define
\[
\begin{aligned}
\Phi_h &:= \overline{\{(z,h(z)) : z \in U\}} \subseteq Z \times \mathbb P^d,
& h^{-1}(q) &:= \pi_Z\bigl(\Phi_h \cap (Z \times \{q\})\bigr), \\
A_h &:= \{q \in \mathbb P^d : |h^{-1}(q)| = n\},
& D_h &:= \overline{\mathbb P^d \setminus A_h}, \\
\operatorname{Br}_{\mathrm A}(h)
&:= \bigcup_{\substack{T \in \operatorname{Irr}(D_h) \\ \dim T = d-1}} T.
\end{aligned}
\]
($\pi_Z\colon \Phi_h\to Z$ is the restriction of the projection onto the first factor.)
Thus $\operatorname{Br}_{\mathrm A}(h)$ is defined set-theoretically, without multiplicities. 

Now, \cite[Lemma~3]{andreotti} implies that, if $Z$ and $Z'$ are $d$-dimensional projective irreducible smooth varieties, $\varphi\colon Z\dashrightarrow Z'$ is birational, and $h\colon Z\dashrightarrow\mathbb P^d$ and $h'\colon Z'\dashrightarrow\mathbb P^d$ are separable dominant rational maps satisfying $h'\circ\varphi=h$, then  $\operatorname{Br}_{\mathrm A}(h)=\operatorname{Br}_{\mathrm A}(h')$.

Let $f\colon X\to\mathbb P^2$ be the double cover branched along a smooth conic $C$, and let $L$ be tangent to $C$. Then $X\cong\mathbb P^1\times\mathbb P^1$ and $f^{-1}(L)=E_1\cup E_2$, where $E_1^2=E_2^2=0$ and $E_1\cdot E_2=1$. Let $\pi \colon Y \to X$ be the blow up of two points of $E_1\setminus E_2$ and of  one point of $E_2\setminus E_1$. The strict transforms $B_1,B_2$ of $E_1,E_2$ satisfy $B_1^2=-2$, $B_2^2=-1$, and $B_1\cdot B_2=1$. Contracting $B_2$ and then the image of $B_1$, whose self-intersection becomes $-2+1=-1$, gives a birational morphism $c\colon Y\to X'$ to a smooth projective surface, contracting $B_1\cup B_2$ to some point $r$.

Set $g:=f\circ\pi$ and $f':=g\circ c^{-1}\colon X'\dashrightarrow\mathbb P^2$. By properness and density, $\Phi_{f'}=(c,g)(Y)$. For general $q\in L$, the two points of $g^{-1}(q)$ lie respectively on $B_1$ and $B_2$, and both are mapped by $c$ to $r$. Hence $(f')^{-1}(q)=\{r\}$ in Andreotti's sense. Since $f'$ has degree $2$, the line $L$ is a component of $\operatorname{Br}_{\mathrm A}(f')$, whereas $\operatorname{Br}_{\mathrm A}(f)=C$. This contradicts \cite[Lemma~3]{andreotti}.
\end{remark}
\begin{remark}\label{rem:correction-ACGH}
 In \cite[p.~247]{acgh}, the authors claim that, for a smooth non-hyperelliptic curve, the ramification of the map $\rho_{\ud}\colon \Omega_{\ud}\to \Proj_C^\vee$ (referred to as $\alpha'$ in \cite[p.~247]{acgh}) occurs at the pairs $(D, H)$ where the divisor $D$ is not reduced. Example~\ref{ex:acgh} shows that this claim is incorrect. Using the notation in the example, pairs of the form $(p+\sum_{i\in S_2} q_i, H)$ are in the ramification locus while the divisor is generically reduced; on the other hand, pairs of the form $(2p + \sum_{i \in S_1} q_i,H)$ are not in the ramification locus even though the divisor is nonreduced. 

Thus, the argument given in \cite[p.~245]{acgh} for the proof of Torelli's theorem for smooth curves via Andreotti's argument is incomplete. It is worth noting that the final description given in \cite{acgh} of the branch locus of $\widehat{\gamma}_{\ud}$ is correct (as it is in \cite{andreotti}), as we establish in Theorem~\ref{thm:branch-gamma} in the more general case of stable curves with no separating nodes.

A crucial step in our proof of Proposition~\ref{prop:equal} involves identifying specific components of the ramification locus (the loci $\Ram_{\ud,v}$ and $\Ram_{\ud,p}$) that dominate the branch locus of $\widehat{\gamma}_{\ud}$, and whose images in the symmetric product have codimension~$1$. This step, which seems essential for this argument, is missing from \cite{acgh}. Specializing Proposition~\ref{prop:rho-ramif} to the case of a smooth curve provides a correction to the proof of the classical Torelli's theorem via Andreotti's method given in \cite{acgh}.
\end{remark}
   
Next, let \(A=(G\curvearrowright P,\Theta)\) be a smoothable principally polarized
stable semiabelic pair. Then \(A\) carries a canonical involution
\(
\iota_A\colon P\longrightarrow P
\)
characterized by
\[
\iota_A(g\cdot p)=g^{-1}\cdot\iota_A(p),
\qquad
\iota_A^*\Theta=\Theta.
\]
On the locus of PPAVs, this is the canonical reflection centred at the theta divisor (see
\cite[Corollary~3.0.7, Remark~3.0.11 and Theorem~5.7.1]{alexeev}
and \cite[Section~4]{hulek} for more details). Here we only need that every SSAP \(A(C)\) considered below is principally polarized and
smoothable.  A $G$-orbit of $P$ that is fixed by this involution will be called a \emph{fixed orbit}. If $A= A(C)=(J(C)\action \overline{P_C^{g-1}},\Theta(C))$ for some stable curve $C$, the involution $\iota_A$ takes a pair $(S,[I])$ to $(S, [\omega_{C_S}\otimes I^\vee])$. 

\begin{definition}\label{def:Eulerian-locus}
    The \emph{Eulerian locus} of  $\overline{P_C^{g-1}}$ is the closed subset $\overline{P^E_C}$ of $\overline{P_C^{g-1}}$ which is the union of the closures of all fixed orbits of maximal dimension. We let $\Theta_E(C)$ be the restriction of $\Theta(C)$ to $\overline{P^E_C}$. 
\end{definition}

\begin{remark}\label{rem:S-definition}
Notice that the Eulerian locus $\overline{P^E_C}$ of $\overline{P_C^{g-1}}$ is an irreducible component of $\overline{P_C^{g-1}}$ if and only if the dual graph of $C$ is Eulerian. In this case $\overline{P^E_C}=\overline{P^{\ud}_C}$, where $\ud$ is the unique stable Eulerian multidegree in $\Sigma(C)$. If instead the dual graph of $C$ is not Eulerian, then $\overline{P^E_C}$ is empty. 

By Remark~\ref{rem:connected-union-ssap} the triple $A^E(C) := (J(C)\action \overline{P_C^E},\Theta_E(C))$ is again a SSAP, and it is a non-empty SSAP if and only if the dual graph of $C$ is Eulerian. 
\end{remark}

\begin{definition}
\label{def:theta_0}
   Let $A=(G\action P, \Theta)$ be a SSAP.  Recall Equation \eqref{eq:zeta-gamma}. For $H\in \mathbb{P}(T_0G)^{\vee}$, we set $\Theta_H := \zeta(\gamma^{-1}(H))\subseteq \Theta$. 
We also define
\[
\SA := \{q\in \mathbb{P}(T_0G) : \Theta_H\not\subseteq \Theta_0 \text{ for general $H\in\Lambda^\vee_q$}\} \subseteq  \mathbb{P}(T_0G).
\]
(Recall the definition of $\Theta_0 \subseteq \Theta$ from Definition~\ref{def: SSAP}).
\end{definition}

When $A=A(C)=(J(C)\action \Pic,\Theta(C))$ for a stable curve $C$, then 
\[
\SA = \mathcal{S}_{A(C)} = \{q\in \mathbb{P}_C : \Theta_H\not\subseteq \Theta_0 \text{ for  general $H\in\Lambda^\vee_q$}\} \subseteq  \mathbb{P}_C.
\]

For the next result, recall Definition~\ref{def:canonical_model}.

\begin{proposition}\label{prop:singularities-recovering}
Let $C$ be a connected nodal curve of genus $g\ge2$ with no separating nodes. The following properties hold.
\begin{enumerate}
    \item[(i)] $\SC=\mathcal{S}_{A(C)}$, where $A(C)=(J(C)\action \Pic,\Theta(C))$.
    \item[(ii)] Assume further that the dual graph of $C$ is Eulerian. Then $\mathcal S_{\widehat{C}} = \mathcal S_{A^E(C)}$.
\end{enumerate}
\end{proposition}

\begin{proof} 
We can reduce to the case where $C$ is stable. Indeed, by Remark~\ref{rem:canonical_model_sep_stab_isomorphic} we have $\mathcal S_{\widehat{C}} = \mathcal S_{\widehat{C}^{\st}}$ , by Remark~\ref{rem:SSAPstab} we have $A(C) = A(C^{\st})$, and the dual graph of $C$ is Eulerian if and only if the dual graph of $C^{\st}$ is Eulerian.

We prove (i). Let $q\notin \mathcal{S}_{\widehat C}$, and let $H\in\Lambda_q^\vee$ be general. Then
$H\cap\mathcal{S}_{\widehat C}=\emptyset$, hence $H\in \mathcal U_0$. For every
$\ud\in\Sigma(C)$, Lemma~\ref{lem:factorize}(i) gives a finite birational morphism
$\beta_{\ud,\mathcal U_0}\colon \Omega_{\ud,\mathcal U_0}\to \Upsilon_{\ud,\mathcal U_0}$.
In particular, this morphism is surjective. Given $(D',H')\in\Omega_{\ud,\mathcal U_0}$, we have $\supp(D')\subseteq C^{\mathrm{sm}}$, and therefore the image of $\beta_{\ud,U_0}(D',H')$ via $\zeta\colon \Upsilon_{A(C)}\to \Theta(C)$ corresponds to an invertible sheaf on $C$. Thus $\zeta(\gamma^{-1}(H))\subseteq \Theta_0$, and hence $\Theta_H\subseteq\Theta_0$. We deduce that $q\notin\mathcal{S}_{A(C)}$, proving the inclusion $\mathcal{S}_{A(C)}\subseteq\mathcal{S}_{\widehat C}$.

To prove the reverse inclusion, fix $q\in\mathcal{S}_{\widehat C}$. By
Corollary~\ref{cor:C_1-set}, the nodes in $\varphi^{-1}(q)$ correspond to a
$\Cone$-set $E_q\subseteq E(\Gamma)$,
where $\Gamma=\Gamma_C$. Following Definition~\ref{def:Cq}, we denote by $C_q$ the partial normalization of $C$ at the set $\varphi^{-1}(q)\subseteq C^{\sing}$.

Choose an orientation $\orient$ as in Proposition~\ref{prop:q_orient}, with
$S=E_q$, and set
\[
\ud=\ud_{\orient}\in\Sigma^{\qE}(C).
\]
 In particular, the edges of $E_q$ induce a
directed cycle on $\Gamma/(E(\Gamma)\setminus E_q)$. Consequently, each connected component of $\Gamma\setminus E_q$ has exactly one incoming and one outgoing $E_q$-incidence, with a loop contributing one of each.

Let $H\in\Lambda_q^\vee$ be general. We may assume that
\[
H\in \mathcal U_1,\qquad
H\cap\mathcal{S}_{\widehat C}=\{q\},\qquad
H\cap\mathcal{B}_{\widehat C}=\emptyset,
\]
and that the divisor
$A:=\nu^*\varphi^*(H)$ on $C^\nu$
is reduced. In particular, $A$ contains the two branches over each node
corresponding to $E_q$, and it contains no branch over any other node.

For every $e\in E_q$, denote the two points over the corresponding node by
$p_e^{\mathrm{out}}$ and $p_e^{\mathrm{in}}$, where
$p_e^{\mathrm{out}}$ lies over the source vertex $\sigma_{\orient}(e)$ and
$p_e^{\mathrm{in}}$ lies over the target vertex $\tau_{\orient}(e)$. If $e$ is a loop,
the orientation does not distinguish its two branches; choose either branch as
$p_e^{\mathrm{out}}$ and call the other one $p_e^{\mathrm{in}}$.

For every $v\in V(\Gamma)$, set
\[
P_v:=\sum_{p_e^{\mathrm{in}}\in C_v^\nu}p_e^{\mathrm{in}},\qquad
Q_v:=\sum_{p_e^{\mathrm{out}}\in C_v^\nu}p_e^{\mathrm{out}},\qquad
A_v:=A|_{C_v^\nu}.
\]
The cyclic orientation of $E_q$ implies that
\[
\deg P_v\leq 1\; \text{ and } \; \deg Q_v\leq 1.
\]

Next, we will construct an effective divisor
\begin{equation}\label{eq:D-definition}
D=\sum_{v\in V(\Gamma)}D_v\in S_{\ud}(C^\nu)
\end{equation}
with the following properties:
\begin{enumerate}
\item[(a)] $\deg D_v=d_v$ for every $v\in V(\Gamma)$;
\item[(b)] $D_0+D_h+\iota_h^*(D_h)\leq A$;
\item[(c)] for every $e\in E_q$, the divisor $D$ contains
$p_e^{\mathrm{in}}$ but not $p_e^{\mathrm{out}}$, and $D$ contains no point
lying over a node not corresponding to an edge of $E_q$;
\item[(d)] if we set $D^{\sm}:=D\cap \nu^{-1}(C^{\sm})$, then $\mathcal{O}_{C_q}(D^{\sm})$ is stable on every connected
component of $C_q$.
\end{enumerate}

We first choose the divisors $D_v$ componentwise so as to obtain
Properties (a), (b), (c). Recall the decomposition $V(\Gamma_C)=V_0 \sqcup V_1 \sqcup V_2$ in Equation \eqref{eq:decomposition}.

Assume first that $v\in V_0$. By Lemma~\ref{lem:dimdeg}(ii), we have
$\deg A_v=k_v$. We set
\[
D_v=P_v+E_v,
\]
for some divisor $E_v$ such that
\[
0 \leq E_v\leq A_v-P_v-Q_v
\;
\text{ and }
\;
\deg E_v=d_v-\deg P_v.
\]
In fact, such an effective divisor $E_v$ exists. Indeed, if $P_v\neq 0$, then
Proposition~\ref{prop:q_orient}(iii) gives $d_v\geq 1$, whereas if
$P_v=0$, Lemma~\ref{lem:numerical} gives $d_v\geq 0$. Similarly, if
$Q_v\neq 0$, then Proposition~\ref{prop:q_orient}(ii) gives
$d_v\leq k_v-1$, whereas if $Q_v=0$, Lemma~\ref{lem:numerical} gives
$d_v\leq k_v$. Therefore
\[
0\leq d_v-\deg P_v
\leq k_v-\deg P_v-\deg Q_v
=\deg(A_v-P_v-Q_v)
\]
showing the existence of such an effective divisor $E_v$. 

Thus $D_v$ has degree $d_v$, contains the incoming point when one is present, and contains no outgoing point.

Assume now that $v\in V_1$. The divisor $A_v$ is invariant under $\iota_h$ and, by the generality of $H$, is a disjoint union $\iota_h$-orbits (each consisting of two points). Moreover, Lemma~\ref{lem:dimdeg}(iv) gives $d_v=\frac{k_v}{2}$. If $C_v^\nu$ contains points lying over the nodes corresponding to $E_q$, then
$P_v+Q_v$ is one of these $\iota_h$-orbits. Indeed, its two points are
conjugate by Proposition~\ref{prop:catanese_hh}(iv) combined with the
description of the involution for a node internal to an honest hyperelliptic
subcurve. We choose one point from each two-point orbit of $A_v$, choosing the
point in $P_v$ from the distinguished orbit $P_v+Q_v$. This gives a divisor
$D_v$ satisfying
\[
\deg D_v=d_v,\qquad
P_v\leq D_v,\qquad
\operatorname{Supp}(D_v)\cap\operatorname{Supp}(Q_v)=\emptyset,
\qquad
D_v+\iota_h^*D_v=A_v.
\]

Finally, let $v',v''\in V_2$ correspond to an honest hyperelliptic pair $Z$. The
involution $\iota_{\h}$ identifies $C_{v'}^\nu$ with $C_{v''}^\nu$, and we have
\[
A_{v''}=\iota_h^*A_{v'}
\;
\text{ and }
\;
d_{v'}+d_{v''}=k_{v'}
\]
by Lemma~\ref{lem:dimdeg}(iv). If $Z$ contains points lying over the nodes
corresponding to $E_q$, relabel $v'$ and $v''$ so that 
$p^{\mathrm{in}}$ lies on $C_{v'}^\nu$ and 
$p^{\mathrm{out}}$ lies on $C_{v''}^\nu$. These two points are exchanged by
$\iota_{\h}$ by Proposition~\ref{prop:catanese_hh}(iv) and the description of the involution of an honest hyperelliptic pair (see Definition~\ref{def:non-hyp}). We choose
\[
0 \leq D_{v'}\leq A_{v'},\qquad \deg D_{v'}=d_{v'},
\]
requiring that $p^{\mathrm{in}}\leq D_{v'}$ when the incoming point is present.
This is possible because Proposition~\ref{prop:q_orient}(iii) gives
$d_{v'}\geq 1$ in that case, while Lemma~\ref{lem:numerical} gives
$d_{v'}\leq k_{v'}=\deg A_{v'}$. We define
\[
D_{v''}:=\iota_h^*(A_{v'}-D_{v'}).
\]
Then
\[
\deg D_{v''}
=k_{v'}-d_{v'}
=d_{v''}
\;
\text{ and }
\;
D_{v'}+\iota_h^*D_{v''}=A_{v'}.
\]
Since $p^{\mathrm{in}}\leq D_{v'}$ and
$\iota_h(p^{\mathrm{out}})=p^{\mathrm{in}}$, the divisor $D_{v''}$ does not
contain $p^{\mathrm{out}}$.

Form the divisor $D$ as in Equation \eqref{eq:D-definition}. The preceding construction gives
\[
D_0+D_h+\iota_h^*(D_h)\leq A=\nu^*\varphi^*(H).
\]
Furthermore, for every $e\in E_q$, the divisor $D$ contains exactly the point
$p_e^{\mathrm{in}}$ lying over the corresponding node. Since $A$ contains no
branch over any other node, $D$ contains no point lying over a node outside
$E_q$. Consequently,
$S_D=E_q$ (recall Definition \ref{def:T} for the definition of $S_D$), and Properties (a) (b) (c) hold.

It remains to prove Property (d). Let $\Gamma_1,\ldots,\Gamma_k$ be the connected components of $\Gamma\setminus E_q$, which are the dual
graphs of the connected components of $C_q$. For $v\in V(\Gamma_i)$, set $r_v:=\deg P_v$. Since the points of $D$ lying over the normalized nodes are precisely the
incoming points, the degree of $\mathcal{O}_{C_q}(D^{\mathrm{sm}})$ on $C_v$ is
\[
\deg_{C_v}\mathcal{O}_{C_q}(D^{\mathrm{sm}})
=d_v-r_v.
\]
On the other hand,  $d_v=w_\Gamma(v)-1+\tau_{o,v}$, 
and restricting $\orient$ to $\Gamma_i$ removes precisely the $r_v$ incoming
$E_q$-incidences at $v$. Hence
\[
d_v-r_v
=w_\Gamma(v)-1+\tau_{\orient|_{\Gamma_i},v}
=d_{\orient|_{\Gamma_i},v}.
\]
By Proposition~\ref{prop:q_orient}(iv), the orientation
$\orient|_{\Gamma_i}$ is strong. Therefore the restriction of the multidegree of
$\mathcal{O}_{C_q}(D^{\mathrm{sm}})$ to each connected component of $C_q$ is
stable. This proves Property (d). Thus, all four properties hold. Notice that it also follows that $D\in \mathbb{T}_{\ud}$ (recall 
Definition~\ref{def:T}).

We will now use $D$ and its properties to produce a point $x$ of $\Theta_H$ lying outside
$\Theta_0$. Indeed,  Properties (a) and (b), together with
Remark~\ref{rem:Omega_d}(ii), give $(D,H)\in\Omega_{\ud}$,  while Properties
(c) and (d) give $S_D=E_q$ and
$D\in \mathbb{T}_{\ud}$, hence
$(D,H)\in\Omega_{\mathbb{T}_{\ud}}$. Set
\[
x:=\alpha_{\ud}(D)
=
\bigl(E_q,[\mathcal{O}_{C_q}(D^{\mathrm{sm}})]\bigr).
\]
By Lemma~\ref{lem:factorize}(iii), the morphism $\beta_{\mathbb{T}_{\ud}}$ factors through
$\Upsilon_{\ud}$, so Equation~\eqref{eq:beta_def} gives
\[
\beta_{\mathbb{T}_{\ud}}(D,H)=(x,H)\in\Upsilon_{\ud}.
\]
Thus $(x,H)\in\gamma^{-1}(H)$ and hence $x=\zeta(x,H)\in\Theta_H$.
On the other hand, $S_D=E_q\neq\emptyset$, so $x$ belongs to the boundary
stratum of $\overline{P_C^{g-1}}$ indexed by $E_q$. In particular,
$x\in\Theta_H\setminus\Theta_0$,
and hence $\Theta_H\not\subseteq\Theta_0$. Since this holds for a general
$H\in\Lambda_q^\vee$, Definition~\ref{def:theta_0} gives
$q\in\mathcal{S}_{A(C)}$. This concludes the proof of the other inclusion $\mathcal{S}_{\widehat C}\subseteq\mathcal{S}_{A(C)}$, and hence of~(i).

We now prove (ii). The inclusion
\[
\mathcal{S}_{A^E(C)}\subseteq\mathcal{S}_{A(C)}
=\mathcal{S}_{\widehat C}
\]
is immediate. Conversely, assume that $\Gamma$ is Eulerian. Then every
quasi-Eulerian orientation is Eulerian, and
\[
\Sigma^{\qE}(C)=\Sigma^{\mathrm{Eul}}(C)=\{\ud_E\},
\]
where $\ud_E$ is the unique stable Eulerian multidegree; see
Remark~\ref{rem:S-definition}. Thus, although the orientation used in the proof of (i) above may
depend on $q$, its associated multidegree is always $\ud_E$. The construction
therefore gives, for every $q\in\mathcal{S}_{\widehat C}$ and general
$H\in\Lambda_q^\vee$, a point of 
$\zeta_{\ud_E}(\gamma^{-1}_{\ud_E}(H))\setminus \Theta_0$.
Hence $q\in\mathcal{S}_{A^E(C)}$, and thus
$\mathcal{S}_{\widehat C}=\mathcal{S}_{A^E(C)}$.
\end{proof}

\subsection{Proof of the main results}

In this section we prove the main results of the paper. We first need some preliminary results on curves with separating nodes.

Let $C$ be a connected
nodal curve of genus $g\ge 1$. Let $\widetilde C$ be the normalization of $C$ at the set of all separating nodes of $C$. As in Notation \ref{not:dagger}, write the decomposition into connected components
\[
\widetilde C=\coprod_{1\le i\le r} C_i \sqcup \coprod_{1\le j\le s} R_j,
\]
where $g_{C_i}>0$ and $g_{R_j}=0$.

\begin{remark}\label{rem:product-separating}
There is a canonical isomorphism of SSAPs
\[
A(C)\cong \prod_{i=1}^r A(C_i).
\]
Under this isomorphism, the theta divisor is the sum of the pullbacks
of the theta divisors of the factors.
\end{remark}


In the notation of Remark~\ref{rem:product-separating}, write
\[
A(C)=(G\action P,\Theta),\qquad
A(C_i)=(G_i\action P_i,\Theta_i).
\]
Thus, after identifying $A(C)$ with $\prod_{1\le i\le r} A(C_i)$, we have
\[
G=\prod_{1\le i\le r} G_i,\qquad
P=\prod_{1\le i\le r} P_i,\qquad
\Theta=\sum_{1\le i\le r} \pi_i^*\Theta_i,
\]
where $\pi_i \colon  P=\prod_{1\le j\le r} P_j \to P_i$ is the projection onto the $i$-th
factor. Set
\[
V_i:=T_0G_i,\qquad
V:=T_0G=\bigoplus_{1\le i\le r} V_i,\qquad
\Proj^\vee_{C_i}:=\Proj(V_i^\vee)\subseteq \Proj(V^\vee).
\]
Recall the definition of the Gauss map  
\[
\mathcal{G}\colon \Theta^{\sm}_0\to \Proj(T_0G)^\vee=\Proj_C^\vee
\]
of $A(C)$ and the closure of its graph $\Upsilon=\Upsilon_{A(C)}$ in Definition \ref{def:GaussA}. Also, recall that the map $\gamma\colon \Upsilon\to \Proj_C^\vee$ is the restriction to $\Upsilon$ of the projection onto the second factor (see Equation~\eqref{eq:zeta-gamma}). Analogously, let $\mathcal G_i$ and
\(
\gamma_i\colon\Upsilon_{A(C_i)}\to\mathbb P_{C_i}^{\vee}
\)
denote the Gauss map and the projection associated with $A(C_i)$.

\begin{lemma}\label{lem:gauss-blocks-separating}
The image of the Gauss map of $A(C)$ is
\[
 \gamma(\Upsilon)=\coprod_{1\le i\le r} \Proj_{C_i}^{\vee}.
\]
 If $Y_i\subseteq\Upsilon$ is the union of the irreducible components mapping
to $\Proj^\vee_{C_i}$, then
\[
Y_i=\Upsilon_{A(C_i)}\times\prod_{j\ne i}P_j
\qquad
\text{ and }
\qquad
\gamma|_{Y_i}=\gamma_i\circ\pi_i,
\]
where
\(\pi_i\colon Y_i\to\Upsilon_{A(C_i)}\)  is the projection onto the first factor.   Hence $\gamma|_{Y_i}$ and $\gamma_i$ have the
same finite factor in their Stein factorizations over $\Proj^\vee_{C_i}$.

Moreover, each $\Proj^\vee_{C_i}\subset \Proj_C^\vee$, together with  
$A(C)$, intrinsically determines  $A(C_i)$.
\end{lemma}

\begin{proof}
Let $p=(p_1,\ldots,p_r)\in\Theta_0^{\rm sm}$. Since
$\Theta=\sum_{1\le i\le r}\pi_i^*\Theta_i$, there is a unique $i$ such that
$p_i\in\Theta_i$, and locally at $p$ we have
$\Theta=\pi_i^{-1}(\Theta_i)$. In particular
\[
\mathcal{G}(p) = T_0p_i^{-1}\Theta_i\oplus\bigoplus_{j\ne i}V_j \subseteq V,
\]
so $\mathcal{G}(p)$ lies in $\Proj^\vee_{C_i}$ and is identified with $\mathcal{G}_i(p_i)\in \Proj_{C_i}^\vee$. The maps $\gamma_i$ are dominant:  this is immediate when
$g_{C_i}\le2$; when $g_{C_i}\ge3$ this follows from Remark \ref{rem:SSAPstab},  Proposition~\ref{prop:Omega}(i) and Lemma~\ref{lem:factorize}(i). Taking closures gives
\[
Y_i=\Upsilon_{A(C_i)}\times\prod_{j\ne i}P_j
\qquad
\text{ and } 
\qquad
\gamma|_{Y_i}=\gamma_i\circ\pi_i,
\]
so $\gamma(\Upsilon)=\coprod_{1\le i\le r}\Proj^\vee_{C_i}$. In particular, since
$\pi_i$ has connected fibers, the Stein finite factor of $\gamma|_{Y_i}$ and $\gamma_i$ is unchanged.

Let $\Theta^{(i)}$ be the sum of the components $\Theta_\lambda$ of $\Theta$
such that $\Upsilon_{\Theta_\lambda}\subseteq Y_i$. Then
$\Theta^{(i)}=\pi_i^*\Theta_i$. Let
$K_i\subseteq G$ be the identity component of $\operatorname{Stab}_G(\Theta^{(i)})$. Then $\prod_{j\neq i}G_j\subseteq K_i$. For every
$p\in\Theta^{(i)}\cap\Theta_0^{\sm}$, the inclusion
$K_i\cdot p\subseteq\Theta^{(i)}$ gives
$T_0K_i\subseteq T_0p^{-1}\Theta^{(i)}$; the latter is the
hyperplane represented by $\mathcal G(p)$. As $p$ varies, the
classes of these hyperplanes are dense in $\mathbb P_{C_i}^{\vee}$.
Hence
\[
T_0K_i\subseteq\bigcap_{H\in\mathbb P_{C_i}^{\vee}}H
=\bigoplus_{j\neq i}V_j.
\]

Since $K_i=\prod_{j\neq i}G_j$, the quotient group $G/K_i$ is
isomorphic to $G_i$. Moreover, the projection
\[
\pi_i\colon P=\prod_{j=1}^r P_j\longrightarrow P_i
\]
is the categorical quotient of the variety $P$ by $K_i$ (the
categorical quotient $P_j/G_j$ is a point for every $j\neq i$).
Finally, $\Theta^{(i)}=\pi_i^*\Theta_i$, so $\Theta_i$ is the unique
divisor on $P_i$ whose pullback is $\Theta^{(i)}$. Thus $K_i$,
$\Theta^{(i)}$, and the induced action of $G/K_i$ on $P_i$ recover $A(C_i)$, proving the final assertion.
\end{proof}

We can now state and prove the main result of the paper.

\begin{theorem}\label{thm:main1}
Let $C$ and $C'$ be connected nodal curves. There is a natural function
\begin{align*}
\rho_{C,C'}\colon\Isom_{\SSAP}(A(C),A(C')) &\to \Isom_{\CT}(\Can(C), \Can(C'))\\
\alpha = (\alpha_G, \alpha_P) &\mapsto \Proj(d_e\alpha_G)
\end{align*}
\end{theorem}

\begin{proof}
By dimension considerations, we can assume that $C$ and $C'$ have the same genus $g$, otherwise $\Isom_{\SSAP}(A(C),A(C'))=\emptyset$ and the result is trivial. If $g=0$, then $\Isom_{\SSAP}(A(C),A(C'))=\{\id\}$ and $\Can(C)=\Can(C')=\emptyset$ so in this case $\Isom_{\CT}(\Can(C), \Can(C'))=\{\id\}$ and we are done.

Fix an SSAP isomorphism 
\(\alpha=(\alpha_G,\alpha_P)\colon A(C)\xrightarrow{\sim}A(C')\).
Using \(T_eJ(C)\cong H^0(C,\omega_C)^\vee\), we define the isomorphism
\[
\xi_\alpha:=\mathbb P(d_e\alpha_G)\colon
\mathbb P_C\xrightarrow{\cong}\mathbb P_{C'}.
\]
Let us show that \(\xi_\alpha\) is an isomorphism in $\Isom_{\CT}(\Can(C), \Can(C'))$.

Let \(\xi_\alpha^\vee=\mathbb P(((d_e\alpha_G)^{-1})^\vee)\) be the
dual projectivity. The equivariance of \(\alpha_P\) and preservation of the
theta divisor give
\(\mathcal G_{A(C')}\circ\alpha_P
=\xi_\alpha^\vee\circ\mathcal G_{A(C)}\).
So \(\alpha_P\times\xi_\alpha^\vee\) identifies $\Upsilon_{A(C)}$ and  $\Upsilon_{A(C')}$ and
commutes with their projections. Let \(C^\dagger=\coprod_{1\le i\le r} C_i\) and \(C'^\dagger=\coprod_{1\le j\le r'} C'_j\) be the decompositions into connected components of the curves obtained by normalizing the separating nodes and discarding the
genus-zero connected components. We use the canonical identifications of canonical triples
\[
\Can(C) = \Can(C^{\dagger}) = \bigcup_{1\le i\le r} \Can(C_i), \quad
\Can(C') = \Can(C'^{\dagger}) = \bigcup_{1\le j\le r'} \Can(C_j')
\]
given by Remarks~\ref{rem:1-scheme-property} and \ref{rem:canonical_model_sep_stab_isomorphic}. By
Remark~\ref{rem:product-separating} and
Lemma~\ref{lem:gauss-blocks-separating}, the irreducible components of the image of the Gauss maps of $A(C)$ and $A(C')$ are \(\mathbb P_{C_i}^\vee\) and
\(\mathbb P_{C'_j}^\vee\), respectively. Hence \(r=r'\), and there is a unique
permutation \(\sigma\) such that
\[
\xi_\alpha^\vee(\mathbb P_{C_i}^\vee)
=\mathbb P_{C'_{\sigma(i)}}^\vee.
\]
By
Lemma~\ref{lem:gauss-blocks-separating} we have isomorphisms
\(\alpha_i\colon A(C_i)\xrightarrow{\cong}A(C'_{\sigma(i)})\), whose
projectivized differentials are the restrictions
\(\xi_{\alpha,i}:=\xi_\alpha|_{\mathbb P_{C_i}}\).

Suppose that \(g_{C_i}\geq2\). Let
$\widehat\gamma_i\colon\widehat\Upsilon_{A(C_i)}
\to\mathbb P_{C_i}^{\vee}$
be the morphism induced by $\gamma_i$ on the normalization of the
Gauss graph, and set $R_i:=\Br(\widehat\gamma_i)$.
The isomorphism $\alpha_i$ identifies these morphisms, and therefore
$\xi_{\alpha,i}^{\vee}(R_i)=R'_{\sigma(i)}$. By
Theorem~\ref{thm:branch-gamma}, the linear components of \(R_i\) recover
\(\mathcal B_{\widehat C_i}\), while the nonlinear components of \(R_i\) recover the nonlinear part of the curve
\(X_{\widehat C_i}\) (see Equation~\ref{eq:X-def}) by biduality (Proposition~\ref{prop:Kleiman_dual}). Hence \(\xi_{\alpha,i}\) carries
these two sets to their counterparts for \(\widehat C'_{\sigma(i)}\).  Finally, Proposition~\ref{prop:singularities-recovering}(i)
identifies $\mathcal{S}_{\widehat C_i}$ with the intrinsically defined set
$\mathcal{S}_{A(C_i)}$. Hence
\[
\xi_{\alpha,i}(\mathcal{S}_{\widehat C_i})
=\mathcal{S}_{\widehat C'_{\sigma(i)}}.
\]
Proposition~\ref{prop:curves-reconstruct-linear} now shows that
\(\xi_{\alpha,i}\) is an isomorphism  in $\Isom_{\CT}(\Can(C_i),\Can(C'_{\sigma(i)}))$.

If \(g_{C_i}=1\), then $\Proj_{C_i}$ is a point. By definition, this point belongs to \(\mathcal S_{\widehat{C_i}}\) if and only if \(C_i\) is singular, and this fact 
is detected by the torus rank of \(J(C_i)\). Since \(\alpha_i\)
preserves torus rank, \(\xi_{\alpha,i}\) is again an isomorphism in $\Isom_{\CT}(\Can(C_i),\Can(C'_{\sigma(i)}))$.

The upshot is that \(\xi_\alpha\) is 
an isomorphism in
\(\Isom_{\CT}(\Can(C),\Can(C'))\). We define
\(\rho_{C,C'}(\alpha):=\xi_\alpha\). Compatibility of differentials
and projectivizations with composition shows that \(\rho\) is natural.
\end{proof}

\begin{remark}\label{rem:CV-gap}
At the beginning of the proof of the necessary direction of
\cite[Theorem~2.1.7]{Cap-Viv}, the reduction to curves without separating
nodes is deduced from \cite[Corollary~1.3.3(ii)]{Cap-Viv} and
\cite[Remark~1.3.1]{Cap-Viv}. These results give product decompositions of
the associated SSAPs, but an isomorphism between the products does not by
itself identify their factors. Such an identification is required because the $\Cone$-equivalence relation
of \cite[Definition~2.1.5]{Cap-Viv} is defined componentwise for
disconnected curves. The cited statements alone do not justify the required matching of factors;  Lemma~\ref{lem:gauss-blocks-separating}
fills this gap by recovering, in the case when $A=A(C)$ is the compactified Jacobian of a curve with separating nodes, each factor from the ambient SSAP and the
corresponding component of its Gauss image.
\end{remark}

For the next result, recall Definition \ref{def:Eulerian-locus}.

\begin{theorem}\label{thm:main-Euler}
Let $C$ and $C'$ be connected nodal curves. Assume that the dual graph of $C$ is Eulerian. Then there is a natural function
\[
\rho^E_{C,C'}\colon\Isom_{\SSAP}(A^E(C),A^E(C')) \to \Isom_{\CT}(\Can(C), \Can(C')).
\]
\end{theorem}

\begin{proof}
As in the proof of Theorem~\ref{thm:main1}, we may assume that
$C$ and $C'$ have the same genus $g \geq 2$ and that they have no separating nodes.  Since $\Gamma_C$ is Eulerian,  $\overline{P^E_C}$ is not empty. If $\overline{P^E_{C'}}$ is empty, then $\Isom_{\SSAP}(A^E(C),A^E(C'))$ is empty and the statement is trivial. 

Assume therefore $\overline{P^E_{C'}}$ is not empty. By  Remark~\ref{rem:S-definition} the dual graph of $C'$ is Eulerian. 
 Let $\ud$ and $\ud'$ be the unique Eulerian multidegrees in $\Sigma(C)$ and $\Sigma(C')$, respectively.  
The proof now follows the same argument as Theorem \ref{thm:main1}, considering $\Upsilon_{\ud}$ and $\Upsilon_{\ud'}$ instead of $\Upsilon_{A(C)}$ and $\Upsilon_{A(C')}$, and invoking Part~(ii) instead of Part~(i) of Proposition~\ref{prop:singularities-recovering} to identify $\mathcal{S}_{\widehat{C}}$ with $\mathcal{S}_{\widehat{C}'}$.
\end{proof}

\subsection{Comparison results and automorphisms} 

In this final section, all curves are connected nodal curves of genus at
least two without separating nodes; stability will be imposed when needed.
First, in Theorem~\ref{thm:T-equivalence}, we relate the isomorphisms of normalizations occurring in the Torelli
theorem of \cite{Cap-Viv} to isomorphisms of canonical triples. Then in Theorem~\ref{thm:main-aut} we 
compare the two resulting constructions and apply this comparison to
automorphism groups.

Recall that $\nu_C\colon C^\nu\to C$ and
$\varphi_C\colon C\to\mathbb P_C$ denote the normalization and the
canonical map, respectively.

\begin{definition}\label{def:Torelli-eq}
Let $C$ and $C'$ be connected stable curves without separating nodes. For every
$\Cone$-set $S\subseteq E(\Gamma_C)$, let
\begin{equation}\label{def:DS}
D_S:=\bigcup_{e\in S}\nu_C^{-1}(p_e)\subseteq C^\nu,
\end{equation}
where $p_e$ is the node corresponding to $e$, and define $D_{S'}$ similarly
for $C'$.

A \emph{$\TT$-isomorphism} from $C$ to $C'$ is an isomorphism
$f\colon C^\nu\xrightarrow{\sim}C'^\nu$ which carries the collection of
divisors $D_S$ bijectively onto the corresponding collection for $C'$. We
denote the set of $\TT$-isomorphisms by $\operatorname{Isom}_{\TT}(C,C')$.

We say that
$C$ and $C'$ are \emph{Torelli equivalent} if
$\operatorname{Isom}_{\TT}(C,C')\neq\emptyset$.
\end{definition}
The composite and inverse of $\TT$-isomorphisms are again $\TT$-isomorphisms.
Thus connected stable curves without separating nodes, with $\TT$-isomorphisms as morphisms, form a groupoid.

To pass from a $T$-isomorphism to a projective isomorphism of canonical
models, we use the residue description of canonical differentials. We
first introduce the corresponding residue vectors and record two
elementary properties.

\begin{definition}
    Let $C$ be a nodal curve with normalization $\nu_C\colon C^\nu\to C$. Given a subset of edges $S\subseteq E(\Gamma_C)$, let $N_S\subseteq C^{\sing}$ be the subset corresponding to $S$ and set $H(C,S):=\nu_C^{-1}(N_S)$. The elements of $H(C,S)$ are called \emph{branches}. If $S=E(\Gamma_C)$, we simply write $H(C)$.  
\end{definition}

There is a natural involution without fixed points $\iota_C\colon H(C)\to H(C)$ such that $\nu_C(\iota_C(q)) = \nu_C(q)$. Moreover, there is an incidence map $\sigma_C \colon H(C)\to V(\Gamma_C)$ and a natural map $H(C)\to E(\Gamma_C)$ (which is the quotient by $\iota_C$). So, given $e\in E(\Gamma_C)$, we can uniquely write $e=\{h,\iota_C(h)\}$ for some $h\in H(C)$. We define the 'set of possible residue vectors' as
\[
R(C):=
\left\{(r_h)\in k^{H(C)}:
r_h+r_{\iota_C(h)}=0 \text{ and }
\sum_{\sigma_C(h) = v}r_h=0\ \text{for every }v\in V(\Gamma_C)
\right\}.
\]

\begin{lemma}\label{lem:RC-prop}
With the above notation, the following properties hold.
    \begin{enumerate}
        \item[(i)] 
    For every $(r_h)\in R(C)$ and every $V\subseteq V(\Gamma_C)$, we have
    \[
\sum_{\substack{\sigma_C(h)\in V\\ \sigma_C(\iota_C(h))\notin V}} r_h = 0.
    \]
    \item[(ii)] Let $S$ be a $\Cone$-set of $\Gamma_C$ and let $h, h'\in H(C,S)=D_S$ be branches (see Equation \eqref{def:DS}). Then there exists $\epsilon_{h,h'}\in \{\pm1\}$ such that $r_h = \epsilon_{h,h'} r_{h'}$ for every $r \in R(C)$.
    \end{enumerate}
\end{lemma}

\begin{proof}
    Part (i) follows from 
    \begin{align*}
    \sum_{\substack{\sigma_C(h)\in V\\ \sigma_C(\iota_C(h))\notin V}}r_h =& \sum_{\sigma_C(h)\in V} r_h -  \sum_{\sigma_C(h), \sigma_C(\iota_C(h))\in V} r_h\\
    =&\sum_{v\in V}\sum_{\sigma_C(h) = v} r_h - \sum_{\substack{e\in E(V, V)\\ e=\{h, \iota_C(h)\}}} (r_h+r_{\iota_C(h)}) = 0-0=0.
    \end{align*}

    For (ii), if $|S|=1$, the assertion follows immediately from
$r_h+r_{\iota_C(h)}=0$. Assume that $|S|=k\geq2$. Let
$V_1,\ldots,V_k$ be the vertex sets of the connected components of
$\Gamma_C\setminus S$. Since the associated quotient graph is a cycle,
we may label them and write $S=\{e_1,\ldots,e_k\}$ so that, with indices
taken modulo $k$, the edge $e_i$ joins $V_i$ to $V_{i+1}$. Then
\[
E(V_i,V_i^c)=\{e_{i-1},e_i\}.
\]
Let $h_i$ be the branch of $e_i$ incident to $V_i$. Applying (i) to
$V_i$ gives
\[
r_{h_i}+r_{\iota_C(h_{i-1})}=0,
\]
and hence $r_{h_i}=r_{h_{i-1}}$. Thus all the $r_{h_i}$ are equal,
whereas $r_{\iota_C(h_i)}=-r_{h_i}$, proving the assertion.
\end{proof}

\begin{lemma}
\label{lem:beta_unique}
    Let $C$ and $C'$ be connected stable curves without separating nodes. Given $f\in \Isom_T(C, C')$, there exists a unique linear isomorphism $\beta_f\colon \Proj_C\to \Proj_{C'}$ such that $\beta_f\circ\varphi_C\circ \nu_C = \varphi_{C'}\circ \nu_{C'}\circ f$. 
\end{lemma}
\begin{proof}
If $\beta$ and $\beta'$ both satisfy the required equality, then
$(\beta')^{-1}\circ\beta$ fixes $\widehat C$ pointwise. Since
$\widehat C$ is connected and spans $\mathbb P_C$, it follows that
$\beta=\beta'$.

We denote by $\Gamma,\Gamma'$ the dual graphs of $C,C'$ respectively.
The isomorphism $f$ induces bijections $\eta\colon V(\Gamma)\to V(\Gamma')$ and $\tau\colon H(C)\to H(C')$ such that:
\begin{itemize}
    \item[(1)] Given a $\Cone$-set $S$ of $\Gamma$, we have $\tau(D_S) = D_{S'}$ for some $\Cone$-set $S'$ of $\Gamma'$ (see Equation~\eqref{def:DS}).
    \item[(2)] $\sigma_{C'}\circ  \tau = \eta\circ \sigma_C$.
\end{itemize}

If $e=\{h,\iota_C(h)\}\in E(\Gamma)$ for some $h\in H(C)$, then $\tau(h)$ and $\tau(\iota_C(h))$ belong to the same $D_{S'}$ where $S'$ is a $\Cone$-set of $\Gamma'$. By Lemma~\ref{lem:RC-prop}, there exists $\epsilon_e\in\{\pm1\}$ such that, for every $r\in R(C')$, we have $r_{\tau(\iota_C(h))}=\epsilon_e r_{\tau(h)}$.

We claim that
\[
\prod_{e\in E(Q)}(-\epsilon_e)=1
\]
for every cycle $Q\subseteq\Gamma$. Indeed, if a $\Cone$-set $S$ meets $Q$, then the image of $Q$ in the
quotient cycle
\(
\Gamma/(E(\Gamma)\setminus S)
\)
is a nonempty Eulerian subgraph, and hence is the whole cycle. Therefore
$S\subseteq E(Q)$. Hence, by Item (1) above,  there exists $S'\subseteq E(\Gamma')$ such that $\tau(H(C, E(Q))) = H(C', S')$. Moreover, Item (2) implies that $S'=E(Q')$ where $Q'$ is a subgraph of $\Gamma'$ such that each vertex has degree $2$; in particular, $Q'$ is a disjoint union of cycles.

Construct an element $r\in R(C')$ nonzero on $H(C', E(Q'))$ as follows. Orient each connected component of this $2$-regular subgraph, put $1$ on the branch associated with the source of every oriented edge and $-1$ on the branch associated with the target, and extend by zero elsewhere. This defines an element of $R(C')$ that is nonzero on $H(C', E(Q'))$.

Write $Q=(v_1,e_1,\ldots,v_m,e_m,v_1)$ 
cyclically, and write $e_i = \{h_i, \iota_C(h_i)\}$ with $\sigma_C(h_i) = v_i$. Set $q_i:=r_{\tau(h_i)}$. Since $\sigma_{C'}(\tau(h_i)) = \sigma_{C'}(\tau(\iota_C(h_{i-1})))$ (by Item (2)), we have that $r_{\tau(h_i)} = - r_{\tau(\iota_C(h_{i-1}))} = - \epsilon_{e_{i-1}}r_{\tau(h_{i-1})}$. Hence, $q_i=-\epsilon_{e_{i-1}}q_{i-1}$. This means that $\prod q_i = \prod (-\epsilon_{e_{i-1}}) \prod q_{i-1}$. Since $q_i\neq 0$ for every $i$, this proves the claim.

Fix a vertex $v_0\in V(\Gamma_C)$. Given any vertex $v\in V(\Gamma)$ define
\[
s_v:=\prod_{e\in P}(-\epsilon_e),
\]
where $P$ is any trail from $v_0$ to $v$. The claim makes this
independent of $P$. Moreover, for every edge $e$ incident to vertices $u$ and $v$, we have $s_u+\epsilon_e s_v=0$.  Using the usual residue description of sections of the dualizing
sheaf \cite[pp.~76--77]{catanese}, we define
\begin{equation}
\label{eq:Lf}
L_f\colon H^0(C',\omega_{C'})\to H^0(C,\omega_C)
\end{equation}
by
\[
(L_f\xi)|_{C_v^\nu}
:=
s_v(f|_{C_v^\nu})^*
\bigl((\nu_{C'}^*\xi)|_{f(C_v^\nu)}\bigr).
\]
For $\xi\in H^0(C',\omega_{C'})$, let
$r=(r_{h'})\in R(C')$ be the residue vector of $\nu_{C'}^*\xi$.
At the node corresponding to an edge
$e=\{h,\iota_C(h)\}$ incident to vertices $u$ and $v$, the two
residues of $L_f\xi$ have sum
\[
s_u r_{\tau(h)}
+s_v r_{\tau(\iota_C(h))}
=
(s_u+\epsilon_e s_v)r_{\tau(h)}
=
0.
\]
Thus $L_f$ is well defined. Moreover, $L_f$ is injective and hence an isomorphism,
since $C$ and $C'$ have the same  genus. Set
\[
\beta:=\mathbb P(L_f^\vee)\colon
\mathbb P_C\xrightarrow{\cong}\mathbb P_{C'}.
\]
After identifying the connected components of $C^\nu$ and
$(C')^\nu$ via $f$, the map $L_f$ acts on each component by
multiplication by a sign. Hence this sign disappears after
projectivization, and
\(
\beta\circ\varphi_C\circ\nu_C
=
\varphi_{C'}\circ\nu_{C'}\circ f,
\)
as required.
\end{proof}

Given a connected stable curve $C$ with no
separating nodes, set 
\[
N_C:=\{h\in\operatorname{Isom}_{\TT}(C,C):\varphi_C\circ\nu_C\circ h=\varphi_C\circ\nu_C\}.
\]
Note that we have
 \(
N_C=\prod_Z\langle\iota_Z\rangle,
\)
where $Z$ ranges over the honest hyperelliptic subcurves of $C$ and
$\iota_Z$ is their involution, extended by the identity on the remaining
components of $C^\nu$.

\begin{theorem}\label{thm:T-equivalence}
Let $C$ and $C'$ be connected stable curves without
separating nodes. There is a  natural surjection
\begin{align*}
\pi_{C,C'}\colon \operatorname{Isom}_{\TT}(C,C')&\longrightarrow
\operatorname{Isom}_{\CT}(\Can(C),\Can(C'))\\
 f&\longmapsto \beta_f.
\end{align*}
Moreover, every fiber of $\pi_{C,C'}$ is a right $N_C$-torsor  and a left $N_{C'}$-torsor. 
In particular, $\pi_{C,C}$ is injective if and only if
$C$ has no honest hyperelliptic subcurves.
\end{theorem}

\begin{proof}
We first construct $\pi_{C,C'}$. Let
$f\in\operatorname{Isom}_{\TT}(C,C')$ and let $\beta_f$ be the unique linear isomorphism $\beta_f\colon \Proj_C\to \Proj_{C'}$ given by the existence statement in Lemma \ref{lem:beta_unique}.
Since $f$ carries every $D_S$ (recall Equation \eqref{def:DS}) onto the corresponding divisor for $C'$, Lemma \ref{lem:beta_unique} shows that $\beta_f$ identifies
$\mathcal{S}_{\widehat C}$ with $\mathcal S_{\widehat C'}$. Away from these sets, it
identifies the branch loci, and hence also $\mathcal{B}_{\widehat C}$ with
$\mathcal{B}_{\widehat C'}$. Therefore $\beta_f$ is an isomorphism of canonical
triples. By the uniqueness in Lemma \ref{lem:beta_unique},
$\beta_{g\circ f}=\beta_g\circ\beta_f$ for composable
$T$-isomorphisms $f$ and $g$, and $\beta_{\mathrm{id}}=\mathrm{id}$. We may
therefore define $\pi_{C,C'}(f):=\beta_f$, and these maps are natural.

We prove surjectivity. Let
$\beta\in \Isom_{\CT}(\Can(C),\Can(C'))$, and use $\beta$ to
identify the two canonical triples with some canonical triple $(F,\mathcal{S}_F,\mathcal{B}_F)$. For
$q\in \mathcal{S}_F$, Lemmas~\ref{lem:recover_YCq}(ii) and
\ref{lem:reconstruct-type-three} recover $\widehat C_q$, and in particular its connected components. Projection from $q$ induces a morphism $\rho_q\colon F^\nu\to\widehat C_q$, where $\nu_F\colon F^\nu\to F$ is the normalization of $F$. For $x,y\in\nu_F^{-1}(q)$, write $x\sim_q y$ if $\rho_q(x)$ and $\rho_q(y)$ belong to the same connected component of $\widehat C_q$.

The quotient graph associated with the $\Cone$-set represented by $q$ is
a cycle (recall Definition~\ref{d: C1set}, Remark~\ref{rem:propC1set}, and see Figure~\ref{fig:C_1-set}). Hence, for each connected component $Z$ of the partial normalization $C_q$, exactly two points
$q_Z^+,q_Z^-\in C^\nu$ satisfy the conditions
\begin{itemize}
    \item[(1)] $q_Z^+, q_Z^-$ map to $Z$ via $C^\nu\to C_q$;
    \item[(2)] $\varphi_C\circ\nu_C(q_Z^+) = \varphi_C\circ\nu_C(q_Z^-) = q$.
\end{itemize}

We have an induced morphism $\overline\varphi_C\colon C^\nu\to F^\nu$. 
Let $W\in \Irr(F)$. The degree of $\overline{\varphi}_C$ over $W^\nu\subseteq F^\nu$ is the degree of the restriction \begin{equation} \label{eq:defmap}\overline{\varphi}_C|_{\overline{\varphi}_C^{-1}(W^{\nu})}\colon \overline{\varphi}_C^{-1}(W^{\nu})\to W^\nu.\end{equation} Using Proposition~\ref{prop:catanese_hh}, we have that  $\overline{\varphi}_C(q_Z^+) = \overline{\varphi}_C(q_Z^-)$ if and only if there exists $W\in \Irr(F)$ such that $\overline{\varphi}_C$ has degree two over $W^{\nu}$ and $\overline{\varphi}_C(q_Z^+)\in W^{\nu}$.
 It follows that, if $W\in\operatorname{Irr}(F)$ meets $\mathcal S_F$ and if $x\in W^\nu\cap\nu_F^{-1}(q)$ with $q\in \mathcal S_F$, then $e_W:=2/|[x]_{\sim_q}|$ is the degree of the map in Equation~\eqref{eq:defmap} and so is independent of $q$ and $x$. If $W\cap \mathcal S_F=\emptyset$, then $F=W$ and $C$ is smooth; in this case $e_W=2$ precisely when $\mathcal B_F\neq\emptyset$.

When $e_W=2$, let $Z_W\subseteq C$ be the union of the irreducible
components whose normalizations map onto $W^\nu$ under
$\overline\varphi_C$. Then $Z_W$ is an honest hyperelliptic subcurve,
so Proposition~\ref{prop:catanese_hh} gives $W^\nu\cong \mathbb P^1$. By Definition~\ref{def:canonical_model} and
Remark~\ref{rem:branchminusS}, the branch divisor of the induced double cover is $(\nu_F|_{W^\nu})^{-1}(\mathcal B_F\cap W)$. Thus the canonical triple determines the cover over $W^\nu$: it is $W^\nu$ when $e_W=1$, the double cover with this branch divisor when the divisor is nonempty, and the trivial cover $W^\nu\sqcup W^\nu\to W^\nu$ otherwise. Taking the disjoint union over the components of $F$ reconstructs
$\varphi_C\circ\nu_C\colon C^\nu\to F$ up to isomorphism. It also reconstructs the sets $D_S$ (recall Definition \ref{def:Torelli-eq}), since every $D_S$ is the fiber over the corresponding point of $\mathcal S_F$.

Applying the same reconstruction to $C'$ gives an isomorphism
$f\colon C^\nu\xrightarrow{\sim}(C')^\nu$ such that
\[
\beta\circ\varphi_C\circ\nu_C
=
\varphi_{C'}\circ\nu_{C'}\circ f.
\]
Let $q\in \mathcal{S}_F$ correspond to the $\Cone$-set $S$ of $\Gamma_C$, and let
$S'$ be the $\Cone$-set of $\Gamma_{C'}$ corresponding to the point
$\beta(q)\in \mathcal S_{\widehat C'}$. Then
\[
f(D_S)
=
f\bigl((\varphi_C\circ\nu_C)^{-1}(q)\bigr)
=
(\varphi_{C'}\circ\nu_{C'})^{-1}(\beta(q))
=
D_{S'}.
\]
Thus $f\in\Isom_{\TT}(C,C')$. By Lemma~\ref{lem:beta_unique}, we have
$\beta=\beta_f$, and hence $\pi_{C,C'}$ is surjective.

Suppose now that $\pi_{C,C'}(f_1)=\pi_{C,C'}(f_2)$ and put $h=f_1^{-1}\circ f_2$. Naturality gives $\beta_h=\mathrm{id}$, and
Lemma \ref{lem:beta_unique} gives
$\varphi_C\circ\nu_C\circ h=\varphi_C\circ\nu_C$. Hence $h\in N_C$ and $f_2=f_1\circ h$.
Conversely, if $h\in N_C$, the same lemma
gives $\beta_h=\mathrm{id}$. Thus every fiber is a right
$N_C$-torsor  (it is clear that the action of $N_{C}$ is free). Similarly, every fiber is a left $N_{C'}$-torsor. This concludes the proof.
\end{proof}

\begin{corollary}\label{cor:CV}
Let $C$ and $C'$ be connected stable curves with no separating nodes. Assume that there is an isomorphism of SSAPs $A(C)\cong A(C')$.
Then $C$ and $C'$ are Torelli equivalent. 
\end{corollary}

\begin{proof}
 Just combine Theorems \ref{thm:main1} and \ref{thm:T-equivalence}.
\end{proof}

We now compare $\pi_{C,C'}$ with the isomorphism of SSAPs associated in
\cite{Cap-Viv} with a $T$-isomorphism. Let $C$ and $C'$ be connected stable curves without separating nodes, and let $f\in\Isom_{\TT}(C,C')$. By \cite[Definition~3.1.3(c), Proposition~3.2.1, and
Section~4.4]{Cap-Viv} we have an isomorphism
\[
\alpha_f=(\alpha_{G,f},\alpha_{P,f})
\in\Isom_{\SSAP}(A(C),A(C'))
\]
such that $(d_e\alpha_{G,f})^\vee=\pm L_f$ (recall Equation \eqref{eq:Lf}). 

\begin{lemma}
\label{lem:CV-compatibility}
Let $C$ and $C'$ be connected stable curves without
separating nodes. If 
$f\in\Isom_{\TT}(C,C')$, then $\rho_{C,C'}(\alpha_f)=\pi_{C,C'}(f)$.
\end{lemma}

\begin{proof}
By construction, $(d_e\alpha_{G,f})^\vee=\pm L_f$. Passing to the projectivization, we have
\[
\rho_{C,C'}(\alpha_f)
 =\mathbb P(d_e\alpha_{G,f})
 =\mathbb P(L_f^\vee)
 =\beta_f
 =\pi_{C,C'}(f),
\]
where the first equality follows from Theorem~\ref{thm:main1} and the last from Theorem \ref{thm:T-equivalence}. 
\end{proof}
We next turn to automorphisms. We aim to prove the following rigidity lemma, showing that an
automorphism of $A(C)$ is determined by its action on the generalized
Jacobian $J(C)$.

\begin{lemma}\label{lem:forgetful-aut}
Let \(C\) be a connected nodal curve of genus \(g\ge2\) without separating nodes. The
forgetful homomorphism
\begin{align*}
u_C\colon \Aut_{\SSAP}(A(C))
& \longrightarrow \Aut_{\mathrm{Group}}(J(C))\\
 (\alpha_G,\alpha_P) &\longmapsto \alpha_G
\end{align*}
(where the target is automorphisms as an algebraic group) is injective.
\end{lemma}

We start with some notation and an auxiliary lemma. Set $\Gamma:=\Gamma_C$, $E:=E(\Gamma)$ and
$T:=\ker(J(C)\to J(C^\nu))$. Label the branches over each $e \in E$ by
$p_{e,0},p_{e,1}$. For $z=(z_e)\in\widetilde T:=(\mathbb G_m)^E$, let $M_z$ be the line bundle on $C$
obtained by gluing the trivial bundles on all $C_v^\nu$ with
relations $s(p_{e,0})=z_e s(p_{e,1})$. Thus $\nu_C^*M_z\cong\mathcal O_{C^\nu}$, and $z\mapsto[M_z]$
defines a surjective map $\widetilde T\to T$.
Write $z^m:=\prod_{e\in E}z_e^{m_e}$ for $m=(m_e)\in\mathbb Z^E$.
For $\ud\in\Sigma(C)$, set $n_v:=d_v-w_\Gamma(v)+1$ (in particular $\sum_v n_v=|E|$) and
\[
 S_{\ud}:=\left\{\epsilon\in\{0,1\}^E:
 \#\{e:p_{e,\epsilon_e}\in C_v^\nu\}=n_v\text{ for every }v\right\}.
\]

We claim that $S_{\ud}\neq \emptyset$ for every $\ud\in \Sigma(C)$. Indeed, for a strong orientation $\orient$ inducing $\ud$, define $\epsilon^{\orient} = (\epsilon^{\orient}_e)$ as follows. If $e$ is not a loop, then we require that $p_{e, \epsilon^{\orient}_e}\in C^{\nu}_{\tau_{\orient}(e)}$.  If $e$ is a loop, we choose any $\epsilon^{\orient}_e\in \{0,1\} $. By Equation \eqref{eq:d-orient}, we have that  $\epsilon^{\orient}$ is an element of $S_{\ud}$.

For every $e\in E$, both values $0$ and $1$ occur as the
$e$-th coordinate of elements of $S_{\ud}$: reverse a directed cycle through the corresponding edge,
or switch branches if it is a loop.

\begin{lemma}\label{lem:theta-gluing-polynomial}
Assume  $C$ is connected nodal without separating nodes and $g(C) \geq 2$,
and use the notation above. For $\ud\in\Sigma(C)$ and general
$I\in P_C^{\ud}$, there exists a unique (up to multiplication by an element of $k^*$) square-free polynomial
\begin{equation}\label{eq:forgetful-det}
 F_I(z)=\sum_{\epsilon\in S_{\ud}}c_\epsilon z^\epsilon,
 \qquad c_\epsilon\in k^*,
\end{equation}
such that, for every $z\in\widetilde T$,
\[
 F_I(z)=0
 \Longleftrightarrow H^0(C,I\otimes M_z)\ne0
 \Longleftrightarrow I\otimes M_z\in\Theta(C).
\]
If $\ud'\in\Sigma(C)$ and general $I'\in P_C^{\ud'}$ satisfy
\[
 F_I(z)=0\Longleftrightarrow F_{I'}(z)=0
 \qquad\text{for every }z\in\widetilde T,
\]
then $\ud'=\ud$ and $F_{I'}$ is proportional to $F_I$.
Moreover, for every $\lambda\in\widetilde T$, the polynomial
$F_{I\otimes M_\lambda}(z)$ is proportional to $F_I(\lambda z)$.
\end{lemma}
\begin{proof}
For general $(\emptyset, I)\in P_C^{\ud}$, set $L_v:=(\nu_C^*I)|_{C_v^\nu}$.
For any $n_v$ distinct branches $q_1,\ldots,q_{n_v}$ on $C_v^\nu$,
the evaluation map
\[
H^0(C_v^\nu,L_v)\longrightarrow
 \bigoplus_{j=1}^{n_v}L_v|_{q_j},
 \qquad s\longmapsto (s(q_j))_{j=1}^{n_v},
\]
is an isomorphism. Indeed, since $I$ is general, so is $L_v$. The kernel of the evaluation map is $H^0(C_v^\nu,L_v(-q_1-\cdots-q_{n_v}))$ and the line bundle $L_v(-q_1-\cdots-q_{n_v})$ has degree
$w_\Gamma(v)-1$, this kernel vanishes. Evaluation is therefore injective into an $n_v$-dimensional space, and Riemann--Roch gives $h^0(C_v^\nu,L_v)=n_v$.

Choose bases of $H^0(C_v^\nu,L_v)$ and trivializations of the
fibers $(\nu_C^*I)|_{p_{e,j}}$ for $j=0,1$. Let $\alpha_e\in k^*$ represent,
in these trivializations, the identification
\[
 (\nu_C^*I)|_{p_{e,1}}\xrightarrow{\cong}
 (\nu_C^*I)|_{p_{e,0}}
\]
induced by $I$. Thus a section $s$ of $\nu_C^*I$ descends to $I$
exactly when $s(p_{e,0})=\alpha_e s(p_{e,1})$ for every $e$. The normalization sequence identifies
$H^0(C,I\otimes M_z)$ with the kernel of the  map
\[
 A_I(z):\bigoplus_v H^0(C_v^\nu,L_v)\longrightarrow k^E,
 \qquad s=(s_v)_v\longmapsto
 \bigl(s(p_{e,0})-\alpha_ez_e s(p_{e,1})\bigr)_e.
\]

Fix the basis $B$ of $\bigoplus H^0(C_v^\nu, L_v) = H^0(C^\nu, \nu_C^*I)$ given by the union of basis $B_v$ of each $H^0(C_v^\nu, L_v)$. Index the columns of $A_I(z)$ by $B$ and the rows by $E$. Then, the $(e, s)$ entry of $A_I(z)$  equals $s(p_{e,0}) - \alpha_ez_es(p_{e, 1})$. 

For each $\epsilon$, define the matrix $N_{\epsilon}$ whose entry at $(e, s)$ is $s(p_{e, \epsilon_e})$.  Then define $F_I(z):=\det A_I(z)\in k[z_e:e\in E]$. By multilinearity of the determinant, the coefficient
of $z^\epsilon$ in $F_I$ is
\[
 (-1)^{\sum_e\epsilon_e}
 \left(\prod_{e\in E}\alpha_e^{\epsilon_e}\right)
 \det N_\epsilon.
\]


Set
\[
 R_v:=\{e\in E:p_{e,\epsilon_e}\in C_v^\nu\},\quad r_v = |R_v|.
\]
Note that $\bigsqcup_{v\in V(\Gamma_C)} R_v = E$. 
Group the rows of $N_\epsilon$ by $R_v$, and the columns by the summands
$H^0(C_v^\nu,L_v)$. The $r_v$ rows belonging to $R_v$ have
nonzero entries only in the $n_v$ columns belonging to $B_v$.
Since $\sum_v r_v=|E|=\sum_v n_v$, if  $r_v\ne n_v$ for some $v$,
then there exists $v$ such that $r_v>n_v$. Those rows are linearly dependent, so
$\det N_\epsilon=0$.

If instead $r_v=n_v$ for every $v$, equivalently
$\epsilon\in S_{\ud}$, the reordered matrix is block diagonal.
Its block at $v$ represents the evaluation map
\begin{equation} \label{eq: eval-cond}
H^0(C_v^\nu,L_v)\longrightarrow
 \bigoplus_{\substack{e\in E\\p_{e,\epsilon_e}\in C_v^\nu}}
 L_v|_{p_{e,\epsilon_e}},
\end{equation}
which is an isomorphism by our choice of $I$.
Thus $\det N_\epsilon$ is, up to sign, the product of
these nonzero evaluation determinants. Since every
$\alpha_e\ne0$, this proves that precisely the monomials
indexed by $S_{\ud}$ occur in \eqref{eq:forgetful-det},
each with nonzero coefficient.

Moreover,
     $F_I$ is square-free: its degree in each 
variable $z_e$ is at most one, since $S_{\ud}\subseteq\{0,1\}^E$.

We have that
$F_I(z)=0$ exactly when $\ker A_I(z)\ne0$. The kernel
identification above and the definition of $\Theta(C)$ therefore give
\[
 F_I(z)=0
 \Longleftrightarrow H^0(C,I\otimes M_z)\ne0
 \Longleftrightarrow I\otimes M_z\in\Theta(C).
\]

Let $\ud'\in\Sigma(C)$ and let
\begin{equation}
\label{eq:GS_d'}
 G(z)=\sum_{\epsilon\in S_{\ud'}}a_\epsilon z^\epsilon
 \in k[z_e:e\in E],
 \qquad \textrm{with } a_\epsilon\in k^*
 \quad\text{for every }\epsilon\in S_{\ud'}.
\end{equation}
Assume that
\[
 G(z)=0\Longleftrightarrow F_I(z)=0
 \qquad\text{for every }z\in\widetilde T.
\]
Since $S_{\ud'}\subseteq\{0,1\}^E$, the polynomial $G$
has degree at most one in each variable and is therefore
square-free. The Nullstellensatz gives
\[
 G(z)=cz^qF_I(z),\qquad \textrm{for some } c\in k^*,\quad q\in\mathbb Z^E.
\]
Comparing supports yields $S_{\ud'}=q+S_{\ud}$.
Since every coordinate projection of both sets is $\{0,1\}$,
we obtain $\{0,1\}=\{q_e,q_e+1\}$ for every $e$, hence $q=0$.
Thus $G=cF_I$ and $S_{\ud'}=S_{\ud}$.

Choose $\epsilon$ in this common nonempty set. Its defining
conditions give
\[
 d_v=w_\Gamma(v)-1+
 \#\{e\in E:p_{e,\epsilon_e}\in C_v^\nu\}=d'_v
 \qquad \textrm{for all }v\in V(\Gamma).
\]
Hence $\ud'=\ud$. This proves both uniqueness up to a nonzero
scalar and the comparison assertion.

Finally, let $\lambda\in\widetilde T$.  Since
$\nu_C^*(I\otimes M_\lambda)\cong\nu_C^*I$,
the resulting evaluation map \eqref{eq: eval-cond} is still an isomorphism when $I$ is replaced with $I\otimes M_\lambda$,
hence the above construction also defines a polynomial
$F_{I\otimes M_\lambda}$. Define then $G(z) = F_{I\otimes M_{\lambda}}(\lambda^{-1}z)$.

Then $G(z) = 0$ if and only if $H^0(C, I\otimes M_{\lambda}\otimes M_{\lambda^{-1}z})\neq 0$. Since $M_{\lambda}\otimes M_{\lambda^{-1}z} = M_z$, we have that $G(z)=0$ if and only if $F_I(z)=0$. Moreover, $G$ satisfies Equation \eqref{eq:GS_d'} for $\ud'=\ud$. The result then follows  from the uniqueness in the previous statement.
\end{proof}

\begin{proof} (Of Lemma \ref{lem:forgetful-aut}). Let $(\mathrm{id}_{J(C)},\psi)\in\ker u_C$.  The unique closed orbit $O_C$ is a
$J(C^\nu)$-torsor, and
$L_C:=\mathcal O_{\overline P_C^{g-1}}(\Theta(C))|_{O_C}$
induces its principal polarization
\cite[Fact~1.2.8(iv) and Remark~1.2.11]{Cap-Viv}.
Since the map $\psi$ is equivariant with respect to the $J(C)$-action on $\Pic$, we have that, for every $M\in J(C)$ and every pair $(S, I)\in \Pic$, the following holds: 
\[
\psi(S, I) = (S', I')\Longrightarrow \psi(S, \nu_S^*M\otimes I) = (S', \nu_{S'}^*M\otimes I'),
\]
where $\nu_S\colon C_S\to C$ and $\nu_{S'}\colon C_{S'}\to C$ are the normalization.

Moreover, since $O_C$ is the unique closed orbit, it must be fixed by $\psi$. In particular, we have 
\[
\psi|_{O_C}(C^{\sing}, I) = (C^{\sing}, t_a(I)),
\]
where $t_a(I):= a\otimes I$, for some $a\in J(C^\nu)$. Since $\Theta(C)$ is also preserved by $\psi$, we have $t_a^*L_C\cong L_C$, hence $a=0$. Thus $\psi$ fixes $O_C$ pointwise. If $C$ is smooth, we are done.

Equivariance with respect to $J(C)$ implies that $\psi$ permutes the maximal orbits. We now show that $\psi$ preserves each maximal orbit. Choose $I\in P_C^{\ud_I}$ general and write $\psi(\emptyset, I) = (\emptyset, I')$ with
$I'\in P_C^{\ud_{I'}}$.
By Lemma~\ref{lem:theta-gluing-polynomial}, both $I$ and $I'$
have associated square-free polynomials $F_I$ and $F_{I'}$
with expansions as in \eqref{eq:forgetful-det}.
Since $\psi$ is equivariant and preserves $\Theta(C)$, that lemma gives
\[
 F_I(z)=0\Longleftrightarrow F_{I'}(z)=0,\;
 \forall z\in\widetilde T.
\]
Hence, Lemma~\ref{lem:theta-gluing-polynomial} gives
$\ud_{I'}=\ud_I$ and $F_{I'}=cF_I$ for some $c\in k^*$.
Hence $\psi$ preserves each maximal orbit $P_C^{\ud}$.

Set $\ud:=\ud_I=\ud_{I'}$. Since $P_C^{\ud}$ is a $J(C)$-torsor and $\psi$ is $J(C)$-equivariant, the restriction $\psi_{|P_C^{\ud}}$ is a translation by some $b_{\ud}\in J(C)$. To conclude, we will show that  $b_{\ud}=0$.

The maps $\psi$ and the translation $t_{b_{\ud}}$ agree on $\overline P_C^{\ud}$,
which contains $O_C$. Since $\psi|_{O_C}=\mathrm{id}$, we have
$b_{\ud}\in T$. Choose a lift $\lambda\in\widetilde T$ of $b_{\ud}$. Recall that $I\in P_C^{\ud_I}$ is general and $\psi(\emptyset, I) = (\emptyset, I')$, hence $I'=I\otimes M_\lambda$. Therefore,
Lemma~\ref{lem:theta-gluing-polynomial} implies that $F_{I'}(z)$
is proportional to $F_I(\lambda z)$.
Together with $F_{I'}=cF_I$, this gives
\[
 F_I(\lambda z)=c'F_I(z)
 \qquad \textrm{for some } c'\in k^*.
\]
Recall the expansion of $F_I(z)$ in Equation \eqref{eq:forgetful-det}. Comparing coefficients and using $c_\epsilon\ne0$ for every $\epsilon\in S_{\ud}$, we obtain $\lambda^\epsilon=c'$ for every
$\epsilon\in S_{\ud}$. Hence
\begin{equation} \label{eq:lambda}
 \lambda^{\epsilon-\epsilon'}=1
 \qquad \textrm{for all }\epsilon,\epsilon'\in S_{\ud}.
\end{equation}

Orient each edge from its $0$-branch to its $1$-branch.
An integral cycle is a vector $m\in\mathbb Z^E$ satisfying
\[
 \sum_{\substack{e\in E\\p_{e,1}\in C_v^\nu}}m_e
 =
 \sum_{\substack{e\in E\\p_{e,0}\in C_v^\nu}}m_e
 \qquad \textrm{for all }v\in V(\Gamma).
\]
These vectors correspond bijectively to the characters of $T$ via
\[
 \chi_m\colon T\longrightarrow\mathbb G_m,
 \qquad M_z\longmapsto z^m.
\]
\emph{Claim.} The differences $\epsilon-\epsilon'$, with
$\epsilon,\epsilon'\in S_{\ud}$, generate the lattice of integral cycles.

\emph{Proof of the claim.}
For every $\mu\in S_{\ud}$ and every vertex $v$, we have
\[
 \sum_{\substack{e\in E\\p_{e,0}\in C_v^\nu}}(1-\mu_e)
 +
 \sum_{\substack{e\in E\\p_{e,1}\in C_v^\nu}}\mu_e
 =n_v.
\]
Subtracting these equalities for $\mu=\epsilon$ and
$\mu=\epsilon'$ shows that $\epsilon-\epsilon'$ is an
integral cycle.

Conversely, fix a strong orientation $\orient$ inducing $\ud$. Also, to simplify the argument, we will assume without loss of generality that $p_{e,1}\in C^\nu_{\tau_{\orient}(e)}$ for every $e\in E(\Gamma)$.
Set $\epsilon^0 = (1)_e\in S_{\ud}$. 

For every directed cycle $\gamma$ of $\Gamma$ directed by $\orient$, define $\epsilon^\gamma\in S_{\ud}$ by
\[
\epsilon^\gamma_e = \begin{cases}
    1 - \epsilon_{e}^0 & \text{ if } e\in \gamma;\\
    \epsilon_e^0 &\text{  if }e\notin \gamma.
\end{cases}
\]
The difference $\epsilon^\gamma-\epsilon^0$ is, up to sign,
the characteristic vector $\mathbf{1}_{\gamma}=(\mathbf{1}_{\gamma, e})\in \mathbb{Z}^E$ of $\gamma$, where
\[
\mathbf{1}_{\gamma,e}=\begin{cases}
1&\text{ if }e\in \gamma;\\
0&\text{ otherwise.}
\end{cases}
\]
By \cite[Theorem 2.2]{LoeblMatamala2001}, these vectors
generate the integral cycle lattice, thus proving the claim.

\smallskip
By the claim, every integral cycle can be written as
\[
 m=\sum_i a_i(\epsilon_i-\epsilon_i'),
 \qquad a_i\in\mathbb Z,\quad
 \epsilon_i,\epsilon_i'\in S_{\ud}.
\]
The relations established in \eqref{eq:lambda} then give
\[
 \chi_m(b_{\ud})
 =\lambda^m
 =\prod_i
   \bigl(\lambda^{\epsilon_i-\epsilon_i'}\bigr)^{a_i}
 =1.
\]
Thus every character of $T$ is trivial on $b_{\ud}$.
Since characters separate points of a torus, $b_{\ud}=0$.

We conclude that $\psi$ is the identity on the dense union of maximal
orbits, and hence on the reduced variety $\overline P_C^{g-1}$.
\end{proof}

We are now ready for our final result, relating automorphism groups of the SSAP and of the canonical triple.

\begin{theorem}\label{thm:main-aut}
 Let $C$ be a connected nodal curve of genus $g \geq 2$ with no separating nodes. Then there is a short exact sequence of groups
\[
\{\id\}\to \langle\iota_{A(C)}\rangle \to \Aut_{\SSAP}(A(C)) \stackrel{\rho_{C,C}}{\to} \Aut_{\CT}(\Can(C)) \to \{\id\},
\]
where
\(\langle\iota_{A(C)}\rangle\cong\mathbb Z/2\mathbb Z\) and $\iota_{A(C)}$ is the involution of $A(C)$. In particular, if the dual graph of $C$ is also $3$-edge-connected, we have
\[
\Aut_{\SSAP}(A(C))\cong
\begin{cases}
    \begin{array}{ll}
      \Aut(C)\times \mathbb Z/2\mathbb Z,   &  \text{ if $C$ is not hyperelliptic;}  \\
       \Aut(C),   &  \text{ if $C$ is hyperelliptic.} 
    \end{array}
\end{cases}
\]

\end{theorem}

For completeness, we also recall the proof in the smooth case.
\begin{proof}
By Remarks~\ref{rem:canonical_model_sep_stab_isomorphic}
and~\ref{rem:SSAPstab}, and by the naturality of $\rho$, replacing
$C$ by $C^{\rm st}$ does not change the sequence. Thus we may assume
that $C$ is stable. If $C$ is smooth, then $u_C$ identifies
$\Aut_{\SSAP}(A(C))$ with the automorphism group of the principally
polarized Jacobian. By the automorphism form of the classical Torelli theorem \footnote{This is the only part of the paper where we invoke the classical Torelli theorem for smooth curves.}
(see for example \cite[Theorem~12.1 and Section~13]{MilneJac}), its quotient by
$\langle[-1]\rangle$ is $\Aut(C)$ when $C$ is nonhyperelliptic, and
$\Aut(C)/\langle\iota_C\rangle$ when $C$ is hyperelliptic. These are
precisely the corresponding groups
$\Aut_{\CT}(\Can(C))$. 

Assume that $C$ is singular, and let $u_C$ be the homomorphism of
Lemma~\ref{lem:forgetful-aut}. If $\psi = (\psi_G, \psi_P) \in\ker\rho_{C,C}$, then the
construction of $\rho_{C,C}$ in Theorem~\ref{thm:main1} gives
\(
\mathbb P(d_e\psi_G)=\id.
\)
Hence
\(
d_e\psi_G=\lambda\id
\)
for some $\lambda\in k^\times$. Let $T\subseteq J(C)$ be the maximal torus and set
$r:=\dim T$. Since $C$ is singular and has no separating nodes,
\(
r=b_1(\Gamma_C)>0.
\)
Since $T$ is the unique maximal torus in $J(C)$, the automorphism $\psi_G$ restricts to an automorphism of $T$. Its action on the cocharacter lattice
\[
X_*(T):=\operatorname{Hom}(\mathbb G_m,T)\cong\mathbb Z^r
\]
is represented by a matrix $M\in\operatorname{GL}_r(\mathbb Z)$. Under the natural
identification
\(
X_*(T)\otimes_{\mathbb Z}k\cong T_eT,
\)
the differential of this restriction is $M\otimes1$. Therefore
\(
M\otimes1=\lambda\id_r.
\)
Since $M\in\operatorname{GL}_r(\mathbb Z)$ and $\operatorname{char}(k)=0$, it
follows that $\lambda=\pm1$.

Homomorphisms of semiabelian varieties in characteristic zero are
determined by their differentials. Indeed, if $f,g\colon G\to H$ have the same differential, then
$h:=f-g$ has zero differential. In characteristic zero, the
surjective homomorphism $G\to h(G)$ has smooth kernel and hence
surjective differential
\cite[Proposition~1.63 and Corollary~8.39]{MilneAG}.
Thus $\dim h(G)=0$, and $h=0$ by connectedness.

 Consequently, $\psi_G$ is
either $\id$ or $[-1]$. Since $u_C$ is injective and
$u_C(\iota_{A(C)})=[-1]$, every element of $\ker\rho_{C,C}$ is either
$\id$ or $\iota_{A(C)}$. Conversely,
\[
\rho_{C,C}(\iota_{A(C)})=\mathbb P(-\id)=\id.
\]
Thus
\(
\ker\rho_{C,C}=\langle\iota_{A(C)}\rangle .
\)

To conclude, we prove that for every isomorphism $\sigma\colon\Can(C)\xrightarrow{\sim}\Can(C)$ 
there exists an isomorphism $\widetilde{\sigma}\colon A(C)\xrightarrow{\sim}A(C)$ 
such that $\rho_{C,C}(\widetilde{\sigma})=\sigma$.  Indeed, by the surjectivity of
$\pi_{C,C}$ in Theorem~\ref{thm:T-equivalence}, there exists
$f\in\Isom_{\TT}(C,C)$ such that $\pi_{C,C}(f)=\sigma$. By Lemma~\ref{lem:CV-compatibility}, we have
\[
\rho_{C,C}(\alpha_f)
=\pi_{C,C}(f)
=\sigma.
\]
Thus, $\widetilde{\sigma}:=\alpha_f$ is a lift of $\sigma$.

For the final assertion, assume that $\Gamma_C$ is $3$-edge-connected.
Then every $\Cone$-set is a singleton, so
$\operatorname{Isom}_{\TT}(C,C)=\operatorname{Aut}(C)$. Moreover,
Proposition~\ref{prop:catanese_hh}(iv) shows that a proper honest
hyperelliptic subcurve would determine a separating pair, which is impossible since $\Gamma_C$ is $3$-edge-connected. Hence we have the following possibilities 
\[
N_C=
\begin{cases}
\{\id\},&\text{if $C$ is not hyperelliptic},\\
\langle\iota_C\rangle,&\text{if $C$ is hyperelliptic}.
\end{cases}
\]
Theorem~\ref{thm:T-equivalence} therefore identifies
$\operatorname{Aut}_{\CT}(\operatorname{Can}(C))$ with
$\operatorname{Aut}(C)$ in the first case and with
$\operatorname{Aut}(C)/\langle\iota_C\rangle$ in the second. The asserted formulas now follow from the natural action of
$\operatorname{Aut}(C)$ on $A(C)$ by inverse pullback.
\end{proof}

\begin{remark}
The $3$-edge-connectedness hypothesis in the final assertion of
Theorem~\ref{thm:main-aut} cannot be omitted. Indeed, let $C$ be
obtained from $n\ge3$ copies $(X_i;p_i,q_i)$ of a smooth
two-pointed curve $(X;p,q)$ with
$\operatorname{Aut}(X)=\{\id\}$, by identifying $q_i$ with
$p_{i+1}$ cyclically. Then
$\operatorname{Aut}(C)\cong\mathbb Z/n\mathbb Z$, whereas
Theorem~\ref{thm:T-equivalence} gives
$\operatorname{Aut}_{\CT}(\operatorname{Can}(C))\cong \Symm_n$, 
consisting of all permutations of the irreducible components
of the canonical model.
\end{remark}

\bibliographystyle{alpha}
\bibliography{biblio}
\end{document}